\documentclass{elsarticle}

\usepackage{amssymb,amsmath}
\usepackage{psfrag}
\usepackage{graphics}
\usepackage{epsfig}
\usepackage{amsthm}
\usepackage{float}
\usepackage{subfigure}

\makeatletter
\def\elsartstyle{%
    \def\normalsize{\@setfontsize\normalsize\@xiipt{14.5}}
    \def\small{\@setfontsize\small\@xipt{13.6}}
    \let\footnotesize=\small
    \def\large{\@setfontsize\large\@xivpt{18}}
    \def\Large{\@setfontsize\Large\@xviipt{22}}
    \skip\@mpfootins = 18\p@ \@plus 2\p@
    \normalsize
}
\@ifundefined{square}{}{\let\Box\square}
\makeatother

\newtheorem{theorem}{Theorem}
\newtheorem{corollary}{Corollary}
\newtheorem{proposition}{Proposition}
\newtheorem{definition}{Definition}[section]
\newtheorem{lemma}{Lemma}[section]
\newtheorem{example}{Example}[section]
\newtheorem{claim}{Claim}[section]

\renewcommand{\textwidth}{420pt}

\newcommand{\dto}{|}
\newcommand{\dof}{\backslash}
\newcommand{\cto}{.}
\newcommand{\cof}{/}

\graphicspath{{./Figures/eps/}}

\begin{document}

\begin{frontmatter}

\title{Decomposition of Quaternary Signed-Graphic Matroids\thanks{This research has been funded by the European Union (European Social Fund - ESF) 
and Greek national funds through the Operational Program ``Education and Lifelong Learning'' of the National Strategic Reference 
Framework (NSRF) - Research Funding Program: Thalis. Investing in knowledge society through the European Social Fund.}}

\author{Leonidas Pitsoulis}
\ead{pitsouli@auth.gr}

\author{Eleni-Maria Vretta}
\ead{emvretta@auth.gr}

\address{Department of Electrical and Computer Engineering \\
Aristotle University of Thessaloniki \\ 54124 Thessaloniki, Greece}

\date{October 2015}

\begin{abstract}
In this work we provide a decomposition theorem for the class of  quaternary and non-binary signed-graphic matroids. This generalizes previous results
for binary signed-graphic matroids in~\cite{PapPit:2013} and graphic matroids in~\cite{Tutte:1959}, and it provides the theoretical basis for a recognition
algorithm. 
\end{abstract}


\end{frontmatter}

\section{Introduction} \label{sec_introduction}

Some of the most celebrated theorems in matroid theory are characterization theorems for specific classes of matroids. We could classify these theorems into two main 
categories: (i) {\em excluded-minor} theorems and (ii) {\em decomposition} theorems. Excluded-minor theorems provide a list of forbidden minors for a class
of matroids while decomposition theorems provide a set of operations which decompose matroids of a given class into main building blocks. 
A representative example would be the case of regular matroids where we have their excluded-minor characterization by Tutte in~\cite{Tutte:1958}, 
as well as their decomposition by Seymour in~\cite{Seymour:1980}. 
In general, decomposition theorems for matroids are more difficult to obtain than excluded-minor characterizations and have important implications, such as
polynomial time recognition algorithms for the associated classes of matroids.
There is a handful of  recognition algorithms for matroids available, and these algorithms constitute the basic ingredient for recognizing classes of matrices 
representing these matroids. Specifically, the recognition 
algorithm for graphic matroids by Tutte in~\cite{Tutte:1960} provided the first practical and easily implementable polynomial-time recognition algorithm for network 
matrices~\cite{BixCunn:1980},  while the recognition of regular matroids by Seymour in~\cite{Seymour:1980} provided the only known polynomial-time recognition 
algorithm for totally unimodular matrices (see~\cite{Schrijver:1986}). 
Both classes of matrices are considered to be very important for optimization and integer programming problems since they are associated 
with integral polyhedra. 
Moreover, there is also the project by Geelen, Gerards and Whittle~\cite{GeGeW:07} which is currently taking place, 
and it will generalize the Graph Minors Theory developed by Robertson and Seymour, to representable matroids over finite fields. Upon completion the results of this 
project would
imply that for representable matroids over finite fields, we could test in polynomial time whether a given matroid contains another matroid as a minor. 
Therefore we could say that in theory, an excluded-minor characterization for a class of representable matroids would imply the existence of a
recognition algorithm for that class. However it is known from our experience with the Graph Minors Theory, that even if the project is completed, it would be 
far from an actual recognition algorithm for any given matroid class, since the algorithmic obstacles would be most likely immense.

In this work we will provide a decomposition theorem for signed-graphic matroids which are representable over the quaternary field but not on the binary field. Utilizing the
results of Pagano in~\cite{Pagano:1998} where he characterizes the signed graphic representations of quaternary matroids, we will provide structural results for such 
singed graphs, and based upon these we will develop the necessary ingredients that will form the decomposition characterization for the associated matroid class. This theorem
generalizes previous decomposition theorems for binary signed-graphic matroids in~\cite{PapPit:2013} and graphic matroids in~\cite{Tutte:1959}.

This work is organised as follows. In the next section, we give some preliminaries mainly for signed graphs and the associated matroids along with some known decomposition results which are used in this work. After the prelimaries, the special structure of the signed graphs representing $GF(4)$-representable signed-graphic matroids along with extensions of the bridge theory for matrois are utilized in order to provide a characterization in terms of decomposition for the examined class of matroids.

\section{Preliminaries}

\subsection{Graphs} \label{subsec_graphs}
Our main reference for graph theory is the book of Diestel~\cite{Diestel:05} and the works of Zaslavsky~\cite{Zaslavsky:1982} while for matroid theory the book of Oxley~\cite{Oxley:2011} and the book of Pitsoulis~\cite{Pitsoulis:2014}.

A graph $G:=(V,E)$ is defined as a finite set of vertices $V$ and a set of
edges $E\subseteq V\cup V^{2}$ where identical elements are allowed. 
Therefore, there are four types of edges: $e=\{u, v \}$ is called
a \emph{link}, $e=\{v, v \}$ a \emph{loop}, $e=\{v\}$ a \emph{half-edge}, 
while $e=\emptyset$ is a \emph{loose edge}. 
The set of vertices and the set of edges
of a graph $G$ are denoted by by $V(G)$ and $E(G)$, respectively. 

The \emph{deletion of an edge} $e$ from $G$ is the subgraph defined as 
$G\dof e := (V(G), E(G)-e)$.
The \emph{deletion of a vertex} $v$ from $G$ is defined as the deletion of all edges incident with $v$ and 
the deletion of $v$ from $V(G)$. \emph{Identifying}  two vertices $u$ and $v$ is the operation 
where we replace $u$ and $v$ with a new vertex $v'$ in both $V(G)$ and $E(G)$. 
The \emph{contraction of a link} $e=\{u,v\}$ is the subgraph denoted by
$G/e$ which results from $G$ by identifying $u,v$ in $G\dof e$.
The \emph{contraction of a half-edge} $e=\{v\}$ or a \emph{loop} $e=\{v\}$ is the subgraph denoted by
$G/e$ which results from the removal of $\{v\}$ and all half-edges and loops incident to it, while all other links
incident to $v$ become half-edges at their other end-vertex. Contraction of a loose-edge is the same as deletion.
A graph $G'$ is called a \emph{minor} of $G$ if it is obtained from a sequence of deletions and contractions 
of edges and deletions of vertices of $G$. For some $X\subseteq E(G)$ the subgraph \emph{induced} by $X$ is denoted 
by $G[X]$. For $S\subseteq E(G)$, we say that the subgraph $H$ of $G$ is the 
\emph{deletion} of $G$ \emph{to} $S$, denoted by $H=G\dto{S}$, if  $E(H)=S$ and $V(H)$ is the set of end-vertices of 
all edges in $S$. Clearly for set $S\subseteq E(G)$, $G\dto{S}$ is the graph obtained from $G\dof ({E(G)-S})$ by 
deleting the isolated vertices (if any). Moreover, for $S\subseteq E(G)$, a subgraph $K$ of $G$ is 
the \emph{contraction} of $G$ \emph{to} $S$, denoted by $K=G\cto{S}$, if $K$ is the graph obtained from $G/(E(G)-S)$ 
by deleting the isolated vertices (if any). 

Let $G$ be a $2$-connected graph. The graph obtained from $G$ by splitting $v\in V(G)$ into two
vertices $v_{1}, v_{2}$, adding a new edge $\{v_1, v_2\}$, and distributing the edges incident to $v$ among
$v_1$ and $v_2$ such that $2$-connectivity is maintained, is called an \emph{expansion} of $G$ at $v$.
The operation of \emph{twisting} (see \cite[Page 148]{Oxley:2011}), is defined as follows. Let $G_1$ and
$G_2$ be two disjoint graphs with at least two vertices $(u_1,v_1)$ and
$(u_2,v_2)$, respectively. Let $G$ be the graph obtained from $G_1$ and $G_2$  
by identifying $u_1$ with $u_2$ to a vertex $u\in{V(G)}$ and $v_1$ with $v_2$
to a vertex  $v\in{V(G)}$. If we identify, instead, $u_1$ with $v_2$ and $v_1$ with $u_2$ then we obtain
a graph $G'$ which is called  a \emph{twisted} graph of $G$ \emph{about}
$\{u,v\}$. The subgraphs $G_1$ and $G_2$ of $G$ and $G'$ are called the \emph{twisting parts} of the twisting.

Any partition $\{T,U\}$ of $V(G)$ for nonempty $T$ and $U$, defines a \emph{cut} of $G$ denoted by $E(T,U)\subseteq{E(G)}$
as the set of links incident to a vertex in $T$ and a vertex in $U$. A cut of the form
$E(v,V(G)-v)$ is called the \emph{star} of vertex $v$. 
There are several definitions of  connectivity in graphs that have appeared in the literature.  
In this paper we will employ the  \emph{Tutte $k$-connectivity} which we will refer to as $k$-connectivity, due to the fact that the
connectivity of a graph and its corresponding graphic matroid coincide under this definition.  
For $k\geq 1$, a \emph{$k$-separation} of a connected graph $G$ is a partition $\{A,B\}$ of the edges 
such that $\min\{|A|,|B|\}\geq{k}$ and $|V({G\dto A}) \cap V({G\dto B})|=k$. 
The connectivity number of a graph $G$ is defined as $\lambda(G)=min\{k:\text{G has a k-separation} \}$, and we say that $G$ is k-connected for any $k \leq \lambda(G)$. Thus, a k-connected graph is also l-connected for $l=0,\ldots, k-1$. If $G$ does not have a k-separation for any $k \geq 0$, then $\lambda(G)=\infty$. A vertical k-separation of $G$ is a k-separation $\{A,B\}$ where $V(A) \dof V(B) \neq \emptyset$ and $V(B) \dof V(A) \neq \emptyset$. A separation or vertical separation $\{A,B\}$ is said to be connected or to have connected parts when $G[X]$ and $G[Y]$ are both connected.
A \emph{block} is defined as a maximally 2-connected subgraph of $G$. 
Loops and half-edges are always blocks in a graph, since they are 2-connected (actually they are
infinitely connected) and they cannot be part of a 2-connected component because they induce
a 1-separation. 


\subsection{Signed graphs} \label{subsec_signed_graphs}

A \emph{signed graph} is defined as $\Sigma = (G,\sigma)$ where $G$ is a graph called the \emph{underlying graph} and
$\sigma$ is a sign function $\sigma:E(G)\rightarrow \{\pm 1\}$, where 
$\sigma(e) =-1$ if $e$ is a half-edge and  $\sigma(e) =+1$ if $e$ is a loose-edge. 
Therefore, a signed graph is a graph where the  edges are labelled as  positive or  negative, while all the half-edges
are negative and all the loose-edges are positive. 
We denote by $V(\Sigma)$ and $E(\Sigma)$ the vertex set and edge set of a signed graph $\Sigma$, respectively. 

All operations on signed graphs may be defined through a corresponding operation on the underlying graph and
the sign function. In the following definitions assume that we have a signed graph $\Sigma=(G,\sigma)$. 
The operation of \emph{switching} at a vertex $v$ results
in a new signed graph $(G,\bar{\sigma})$  where $\bar{\sigma}(e) = - \sigma (e)$ for each link $e$ incident
to $v$, while $\bar{\sigma}(e) = \sigma (e)$ for all other edges. Two signed graphs are \emph{switching equivalent} if there exist switchings that transform the one to the other. \emph{Deletion} of a vertex $v$ is defined as
$\Sigma \backslash v := (G\backslash v, \sigma)$. \emph{Deletion} of an edge $e$ is defined as
$\Sigma \dof e = (G\dof e, \sigma)$. The \emph{contraction} of an edge $e$ consists of three cases: 
\begin{enumerate}
\item if $e$ is a positive loop, then $\Sigma \cof e = (G\dof e, \sigma)$.
\item if $e$ is a half-edge, negative loop or a positive link, then $\Sigma \cof e = (G\cof e, \sigma)$. 
\item if $e$ is a negative link, then $\Sigma \cof e = (G\cof e, \bar{\sigma})$  where $\bar{\sigma}$ is a 
      switching at either one of the end-vertices of $e$.
\end{enumerate} 
The \emph{expansion} at a vertex $v$, results in a signed graph $(\bar{G},\bar{\sigma})$, where $\bar{G}$ is the expansion of $G$ at $v$, and
$\bar{\sigma}$ is the same as $\sigma$ except for the new edge so created by the expansion, which is given a positive sign. All remaining notions used for a signed graph are as defined for graphs (as applied to its underlying graph). For example, for some $S\subseteq{E(\Sigma)}$ we have that $\Sigma[S]=(G[S],\sigma)$, $\Sigma$ is $k$-connected if and only if $G$ is $k$-connected. 

The \emph{sign of a cycle} is the product of the signs of its edges, so we have a 
\emph{positive cycle} if the number of negative edges in the cycle is  even, otherwise the cycle is a \emph{negative cycle}.
Both negative loops and half-edges are negative cycles with a single edge.
A signed graph is called \emph{balanced} if it contains no negative cycles.  A connected signed graph containing exactly one cycle is called a \emph{negative $1$-tree} if the cycle is negative. Furthermore, we define the \emph{b-star} of a vertex $v$ of a signed graph $\Sigma$, denoted by $st_{\Sigma}(v)$, as the set of edges having $v$ as an end-vertex and are not positive loops. 
A vertex $v\in V(\Sigma)$ is called a \emph{balancing vertex} if $\Sigma \dof v$ is balanced. 

For $k \geq 1$, a \emph{k-biseparation} of a signed graph $\Sigma$ is a partition $\{A,B\}$ of $E(\Sigma)$ such that $min\{|A|,|B|\}\geq k$ that satisfies one of the following three properties: $|V(G[A]) \cap V(G[B])|=k+1$ and both $\Sigma[A], \Sigma[B]$ are balanced. $|V(G[A]) \cap V(G[B])|=k$ and exactly one of $\Sigma[A], \Sigma[B]$ is balanced. $|V(G[A]) \cap V(G[B])|=k-1$ and both $\Sigma[A], \Sigma[B]$ are unbalanced. A \emph{vertical k-biseparation} of $\Sigma$ is a k-biseparation $\{A,B\}$ that has $V(A) \dof V(B) \neq \emptyset$ and $V(B) \dof V(A) \neq \emptyset$. A connected signed graph is called \emph{k-biconnected} when it has no l-biseparation for $l=0,\ldots, k-1$. 

The following two results appear in \cite{Zaslavsky:1982} and allow us to perform switchings at the vertices of a balanced signed graph in order to make all its edges positive.

\begin{proposition}[Zaslavsky~\cite{Zaslavsky:1982}] \label{switchings}
Two signed graphs on the same underlying graph are switching equivalent if and only if they have the same list of balanced cycles.
\end{proposition} 

\begin{corollary}[Zaslavsky~\cite{Zaslavsky:1982}]
A signed graph is balanced if and only if it is switching equivalent to a graph with all positive edges and without half-edges.
\end{corollary}


\subsection{Signed-graphic matroids} \label{subsec_signed_graphic_matroids}

We assume that the reader is familiar with basic notions in matroid theory (see first chapters of~\cite{Oxley:2011}), and in particular with the
circuit axiomatic definition of a matroid and the notions of duality, connectivity, 
representability and minors. Given a matrix $A$ and a graph $G$, $M[A]$ and
$M(G)$ denote the vector and graphic matroids, respectively. For a matroid $M$
we denote by  $E(M)$ be the ground set, $\mathcal{C}(M)$ the family of circuits 
while $M^*$ is the dual matroid of $M$. The prefix `co-' dualizes the term mentioned and the asterisk dualizes the symbol
used. 

The following definition  for the matroid of a signed graph or \emph{signed-graphic matroid} is used in this work. 
\begin{theorem}[Zaslavsky~\cite{Zaslavsky:1982}] \label{th_sgg2}
Given a signed graph $\Sigma$ let $\mathcal{C} \subseteq 2^{E(\Sigma)}$ be the family of edge sets 
inducing a subgraph in $\Sigma$ which is either:
\begin{itemize}
\item[(i)] a positive cycle, or
\item[(ii)]  two vertex-disjoint negative cycles connected by a path which has no common vertex with the cycles apart from 
its end-vertices, or 
\item[(iii)] two negative cycles which have exactly one common vertex.
\end{itemize}
Then $M(\Sigma)=(\mathcal{C}, E(\Sigma))$ is a matroid on $E(\Sigma)$ with circuit family $\mathcal{C}$. 
\end{theorem}

Note that $M(\Sigma)$ is also known as the  {\em frame matroid} of the signed graph $\Sigma$. 
The subgraphs of $\Sigma$ induced by the edges corresponding to a circuit of $M(\Sigma)$ are called the 
\emph{circuits} of $\Sigma$. The circuits 
of $\Sigma$ described by (ii) and (iii) of Theorem~\ref{th_sgg2} are also called {\em tight} and {\em loose handcuffs} 
respectively (see Figure~\ref{fig_circuits_signed_graphs}). 
\begin{figure}[hbtp]
\begin{center}
\mbox{
\subfigure[Positive cycle]
{
\includegraphics*[scale=0.23]{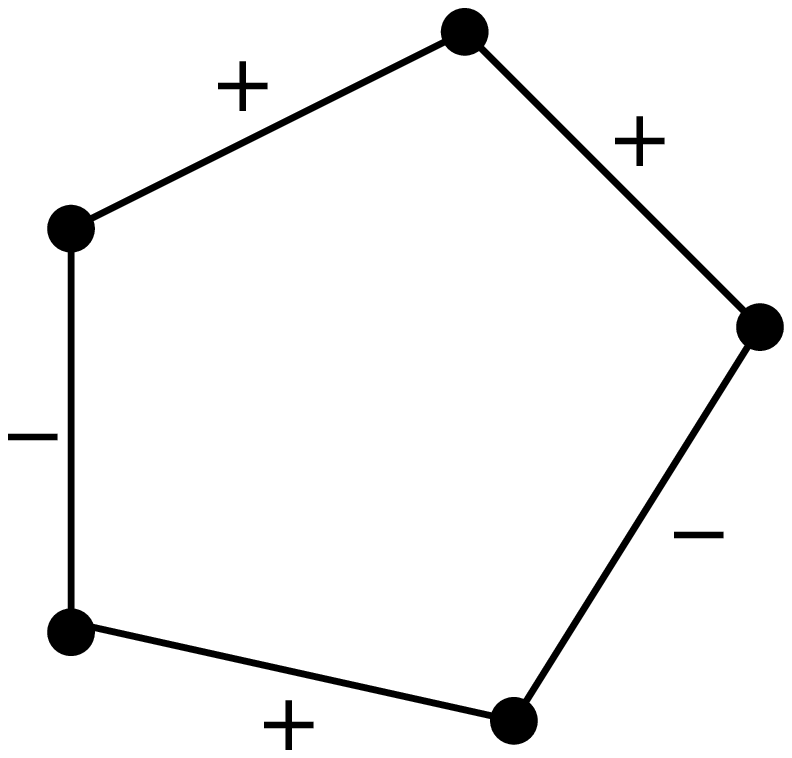}
}\quad
\subfigure[Tight handcuff]
{
\includegraphics*[scale=0.23]{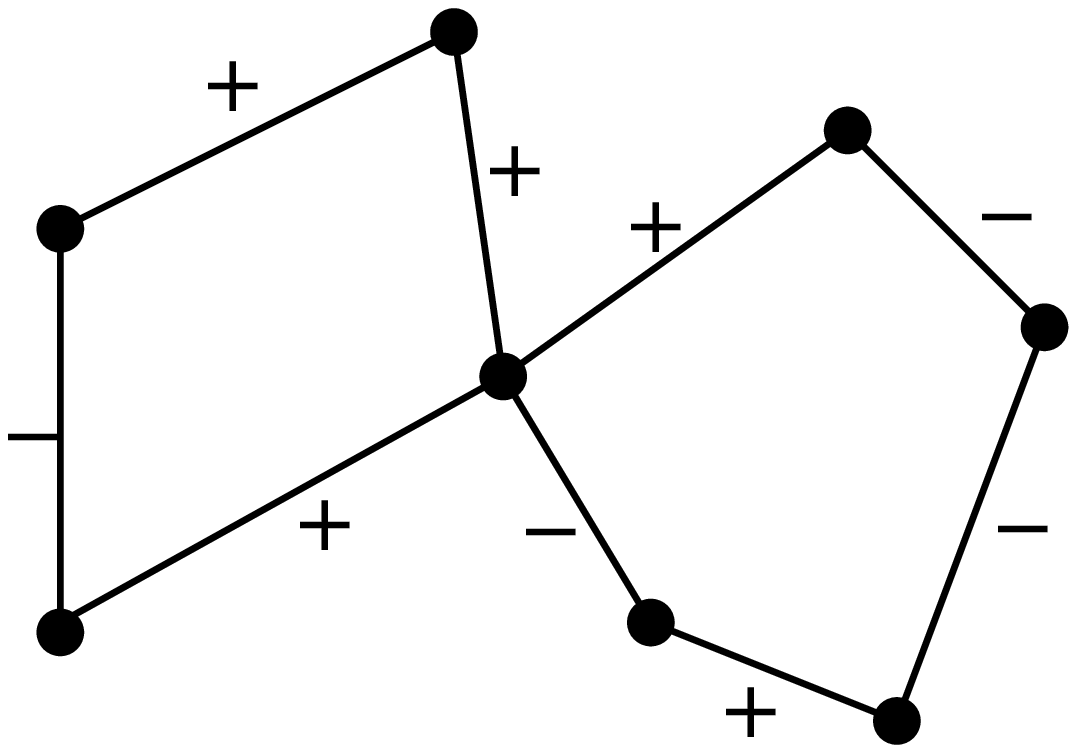}
}\quad
\subfigure[Loose handcuff]
{
\includegraphics*[scale=0.23]{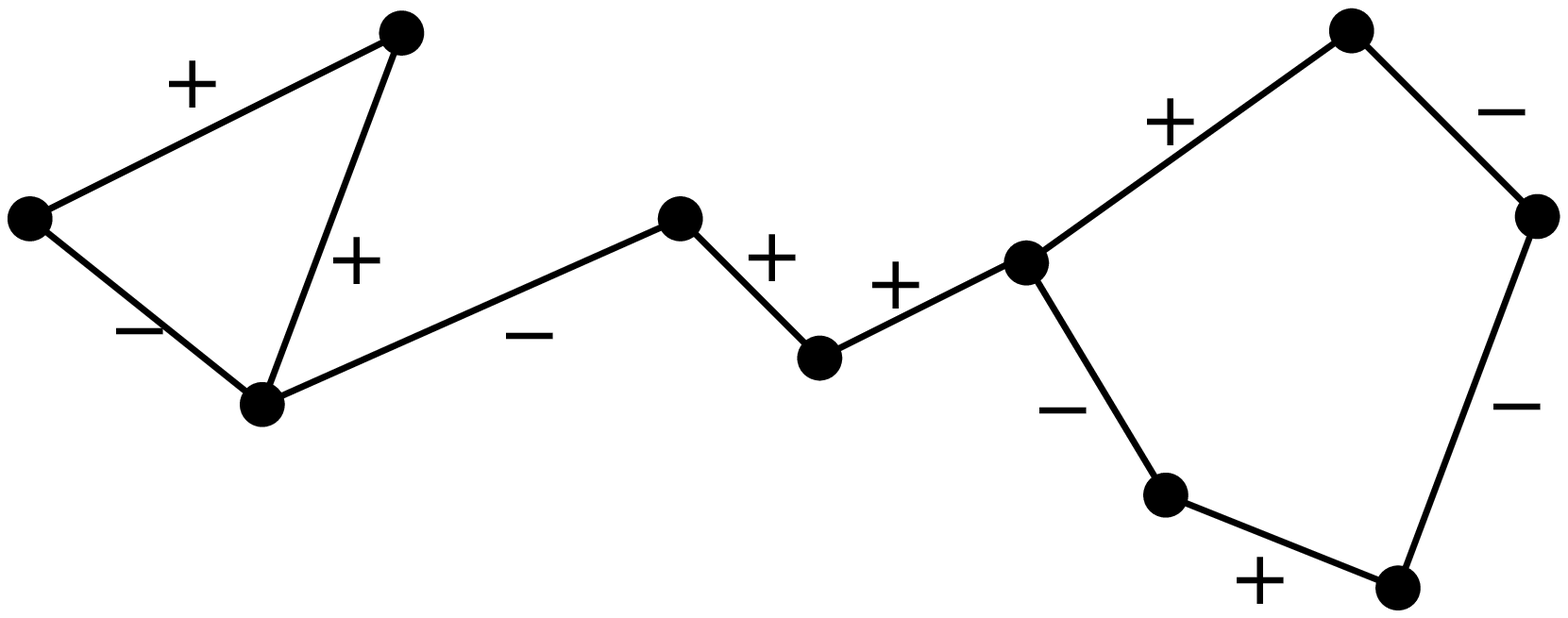}
}}
\end{center}
\caption{Circuits in a signed graph $\Sigma$.}
\label{fig_circuits_signed_graphs}
\end{figure}

Given a matroid $M$ and some set $X\subseteq E(M)$ the \emph{deletion} and \emph{contraction} of $X$ from $M$ will be denoted
by $M\dof X$ and $M/X$ respectively.  If $N$ is a minor of $M$, that is  $N=M\dof X / Y$ for disjoint $X,Y\subseteq E(M)$, we will write $M \succeq N$. 
For a matter of convenience in the analysis that will follow
we also employ the complement notions of deletion and contraction, 
that is the \emph{deletion to} a set $X\subseteq E(M)$ is defined as 
\[
M\dto X = M\dof (E(M) - X), 
\]
while the \emph{contraction to} a set $X\subseteq E(M)$ is defined as 
\[
M\cto X = M/(E(M) - X).
\]
There is an equivalence of the aforementioned matroid operations with respect to the associated 
signed-graphic operations of deletion and contraction
as indicated by Theorem~\ref{th_minsig}. 
\begin{theorem}[Zaslavsky~\cite{Zaslavsky:1982}] \label{th_minsig}
Let $\Sigma$ be a signed graph and $S\subseteq E(\Sigma)$. Then 
$M(\Sigma\dof{S})=M(\Sigma)\dof{S}$ and $M(\Sigma/S)=M(\Sigma)/S$.
\end{theorem}

The following two propositions provide necessary conditions under which certain operations on a signed graph
do not alter its matroid and under which a signed-graphic matroid is graphic. Proofs  can be found in, or easily derived from, the 
results in~\cite{SliQin:07,Zaslavsky:1982, Zaslavsky:1991a}.
\begin{proposition} \label{prop_samematroid}
Let $\Sigma$ be a signed graph. If $\Sigma'$: 
\begin{itemize}
\item[(i)] is obtained from $\Sigma$ by replacing any number of negative loops by half-edges and vice versa, or
\item[(ii)] is obtained from $\Sigma$ by switchings, or
\item[(iii)] is the twisted graph of $\Sigma$ about $(u,v)$ with $\Sigma_1, \Sigma_2$  the 
twisting parts of $\Sigma$,  where $\Sigma_1$ (or $\Sigma_2$) is balanced or all of its negative cycles contain $u$ and $v$, 
\end{itemize}
then $M(\Sigma) = M(\Sigma')$. 
\end{proposition}

\noindent
In view of $(i)$ from Proposition~\ref{prop_samematroid}, from now on we will refer to the negative loops and half-edges as \emph{joints}. 
\begin{proposition} \label{prop_sgn_graphic}
Let $\Sigma$ be a signed graph. If $\Sigma$: 
\begin{itemize}
\item[(i)] is balanced, or 
\item[(ii)] has no negative cycles other than joints, or
\item[(iii)] has a balancing vertex,
\end{itemize}
then $M(\Sigma)$ is graphic. 
\end{proposition}

\noindent
In the first two cases of Proposition~\ref{prop_sgn_graphic} we also have $M(\Sigma) = M(G)$. For the
third case, there exists a graph $G'$ obtained from $G$ by adding a new vertex $v$ and replacing any
joint by a link joining its end-vertex with $v$ such that $M(\Sigma) = M(G')$. 
Also a straightforward result which is a  direct consequence of  Proposition~\ref{prop_sgn_graphic} is 
that if $\Sigma$ is a B-necklace then $M(\Sigma)$ is graphic, since
the vertices of attachment in a B-necklace are balancing vertices. 

For a matroid $M$ and a positive integer $k$, a partition $(A,B)$ of $E(M)$ is a k-separation of $M$ if $min\{|A|,|B|\}\geq k$ and $r(A)+r(B) \leq r(M) +k-1$. The connectivity number of matroid $M$ is defined as $\lambda(M)=min\{k:\text{M has a k-separation for} ~k \geq 1\}$, while if $M$ does not have a k-separation for any $k \geq 1$ then $\lambda (M) =\infty$. We say that a matroid $M$ is k-connected for any $1 \leq k \leq \lambda(M)$. A k-separation $(A,B)$ is called \emph{exact} when $r(A)+r(B) = r(M) +k-1$. If $(A,B)$ is a k-separation of a signed-graphic matroid $M(\Sigma)$ such that $\Sigma[A], \Sigma[B]$ are connected, then $(A,B)$ is called \emph{connected k-separation} or k-separation with connected parts. 

The following three results that appear in \cite{SliQin:07}, \cite{SliQin:07c} determine the connectivity of a signed-graphic matroid in relation to the k-biconnectivity of its signed-graphic representation.

\begin{corollary}[Pagano~\cite{SliQin:07}] \label{2-bic}
If $\Sigma$ is a connected and unbalanced signed graph with at least three vertices then $M(\Sigma)$ is 2-connected iff $\Sigma$ is vertically 2-biconnected, has no balanced loops and has no balancing set of rank one.
\end{corollary}

\begin{proposition}[Slilaty and Qin~\cite{SliQin:07c}] \label{k-bisep} 
Let $\Sigma$ be a connected and unbalanced signed graph. \\
(1) If $(X,Y)$ is a k-biseparation of $\Sigma$, then $(X,Y)$ is a k-separation of $M(\Sigma)$. \\
(2) If $(X,Y)$ is an exact k-separation of $M(\Sigma)$ with connected parts, then $(X,Y)$ is a k-biseparation of $\Sigma$.
\end{proposition}

\begin{theorem}[Slilaty, Qin~\cite{SliQin:07c}] \label{3biconnected}
If $\Sigma$ is a connected and unbalanced signed graph with at least three vertices, then $M(\Sigma)$ is 3-connected iff $\Sigma$ is vertically 3-biconnected, simple, and has no balancing bond of rank one or two.
\end{theorem}


\subsection{Bonds and cocircuits} \label{subsec_bonds_circuits}

With the following
theorem we characterize the sets of edges in a signed graph $\Sigma$ which correspond to circuits
of $M^{*}(\Sigma)$. 
\begin{theorem}[Zaslavsky~\cite{Zaslavsky:1982}] \label{thrm_bonds}
Given a signed graph $\Sigma$ and its corresponding matroid $M(\Sigma)$, $Y\subseteq E(\Sigma)$ is
a cocircuit of $M(\Sigma)$ if and only if $Y$ is a minimal set of edges 
whose deletion increases the number of balanced components of $\Sigma$.
\end{theorem}

The sets of edges defined in Theorem~\ref{thrm_bonds} are called \emph{bonds} of 
a signed graph.
In analogy with the different types of circuits a signed-graphic matroid has, 
bonds can also be classified into different types according to the signed
graph obtained upon their deletion. Specifically, for a 
given connected and unbalanced signed graph $\Sigma$, the deletion of a bond $Y$ 
results in a signed graph $\Sigma\dof{Y}$  with exactly one
balanced component due to the minimality of $Y$. 
Thus,  $\Sigma\dof{Y}$ may be a balanced connected graph
in which case we call $Y$ a \emph{balancing  bond} or it may consist of one balanced component and some unbalanced
components. In the latter case, if the balanced component is a vertex,
i.e. the balanced component is empty of edges, then we say that
$Y$ is a \emph{star bond}, while in the case that the balanced component is not empty of edges  $Y$
can be either an \emph{unbalancing bond} or a \emph{double bond}.
Specifically, if the balanced component is not empty of edges and there is no edge in $Y$ such that both of its end-vertices are vertices of the balanced component, then $Y$ is an unbalancing bond. On the other hand, if there exists at least one edge of $Y$ whose both end-vertices are vertices of the balanced component then $Y$ is a double bond (see Figure~\ref{fig_bonds_signed_graphs}).

\begin{figure}[hbtp]
\begin{center}
\mbox{
\subfigure[balancing bond]
{
\includegraphics*[scale=0.29]{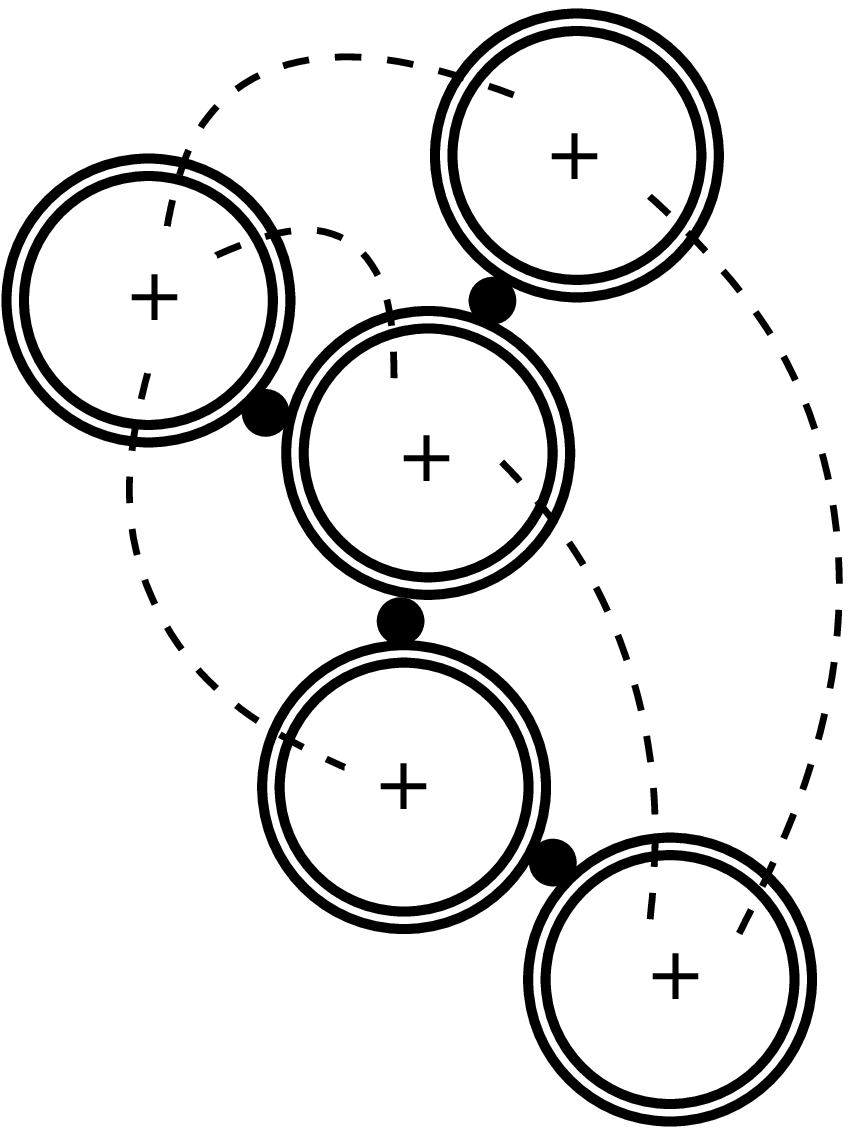}
}\quad
\subfigure[star bond]
{
\includegraphics*[scale=0.29]{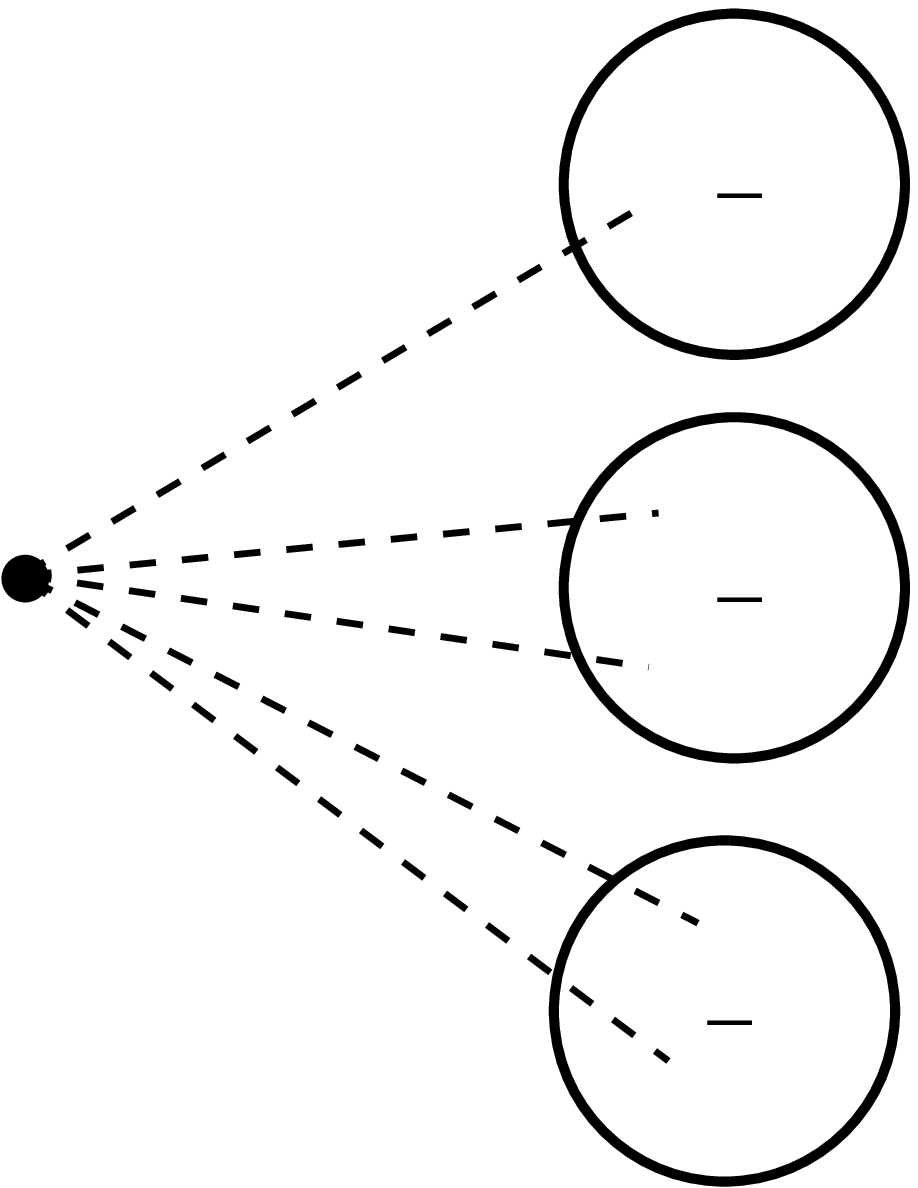}
}\quad
\subfigure[unbalancing bond]
{
\includegraphics*[scale=0.29]{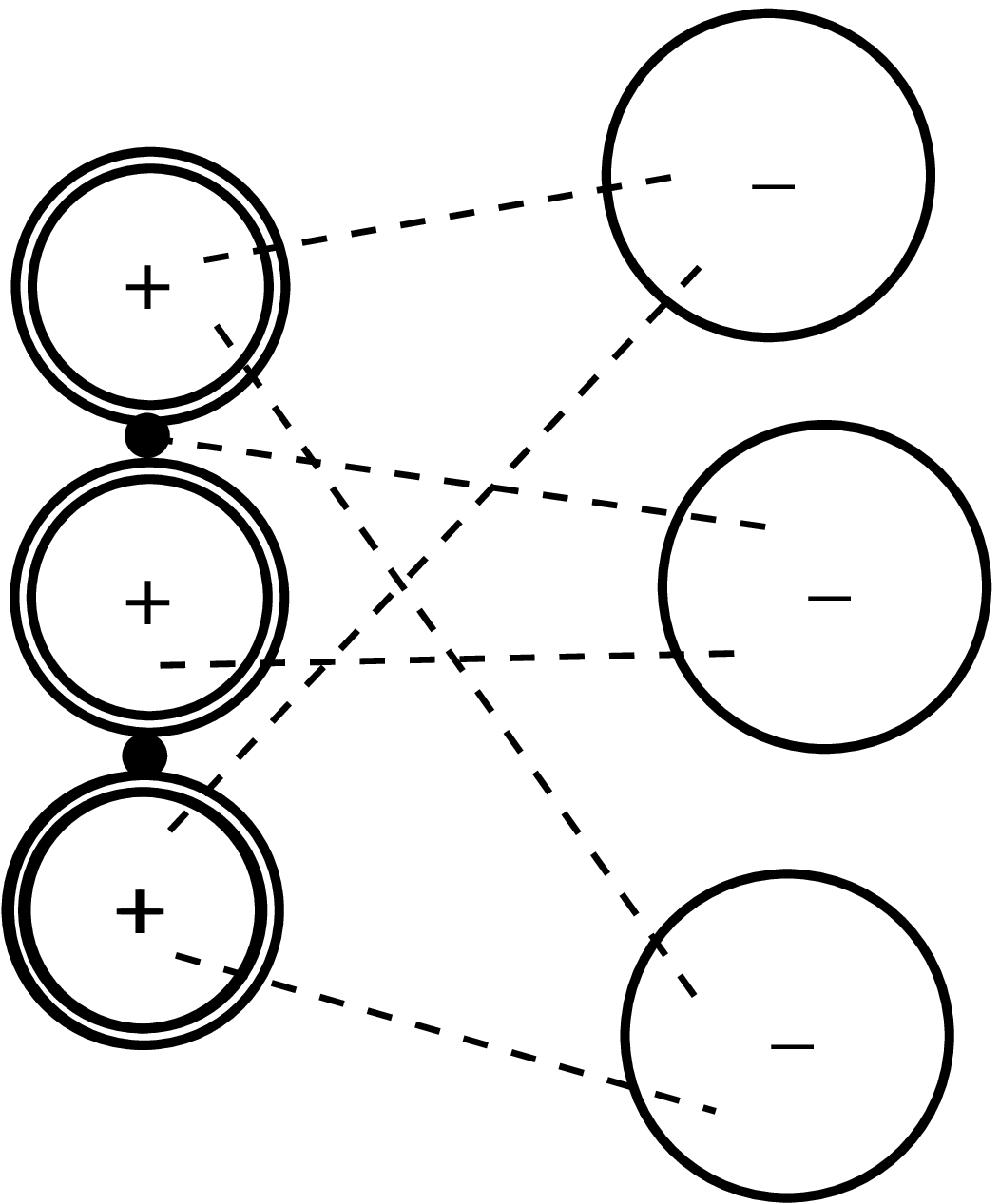}
}\quad
\subfigure[double bond]
{
\includegraphics*[scale=0.29]{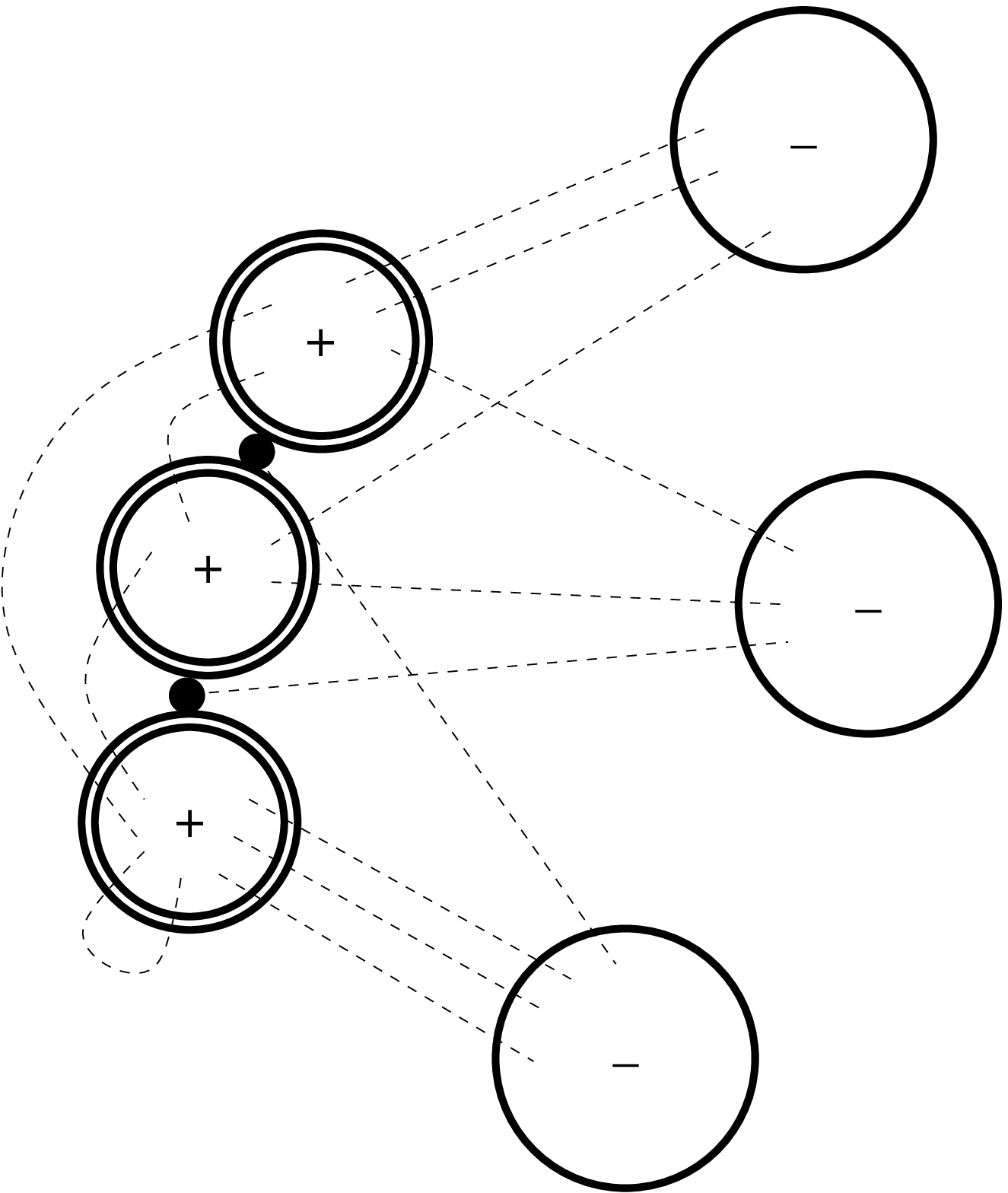}
}} \caption{Bonds in Signed Graphs}
\label{fig_bonds_signed_graphs}
\end{center}
\end{figure}

A further classification of bonds is based on whether the matroid
$M(\Sigma)\dof{Y}$ is connected or not. 
In the case that $M(\Sigma)\dof{Y}$ is disconnected we call $Y$ as \emph{separating bond}
of $\Sigma$, otherwise we say that $Y$ is a \emph{nonseparating bond}. 

In~\cite{Zaslavsky:1982,Zaslavsky:1991a}
the edge sets of a signed graph which correspond to elementary separators in the associated signed-graphic
matroid are determined. 
Before we present this result in Theorem~\ref{th_Zasl11} we have to provide some necessary definitions.    
An \emph{inner block}  of $\Sigma$ is a block that 
is unbalanced or lies on a path between two unbalanced 
blocks. Any other block is called \emph{outer}. 
The \emph{core} of $\Sigma$  is the union of all inner blocks. 
A \emph{B-necklace} is a special type of 2-connected unbalanced signed graph, 
which is composed of maximally 2-connected balanced 
subgraphs $\Sigma_i$ joined in a cyclic fashion as illustrated in Figure~\ref{fig_necklace}. Moreover, the unique common vertex between consecutive 
subgraphs in a B-necklace is called \emph{vertex of attachment}.   Note that in 
Figure~\ref{fig_necklace} as well as in the other figures that follow, a single line boundary depicts a connected graph while 
a double line boundary is used to depict a block, 
where in each case a positive (negative) sign is used to indicate whether the connected or 2-connected
component is balanced (unbalanced). 
Observe that any negative cycle in a B-necklace has to contain at least one edge from each $\Sigma_i$. 
\begin{figure}[h] 
\begin{center}
\centering
\psfrag{S1}{\footnotesize $\Sigma_1$}
\psfrag{S2}{\footnotesize $\Sigma_2$}
\psfrag{S3}{\footnotesize $\Sigma_3$}
\psfrag{S4}{\footnotesize $\Sigma_4$}
\psfrag{S5}{\footnotesize $\Sigma_5$}
\psfrag{Si}{\footnotesize $\Sigma_i$}
\psfrag{Sn}{\footnotesize $\Sigma_n$}
\includegraphics*[scale=0.3]{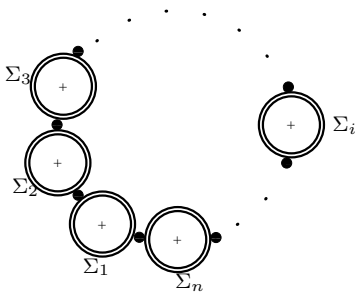}
\end{center}
\caption{A B-necklace.}
\label{fig_necklace}
\end{figure}

In the lemma that follows, a structural property of B-necklaces is shown. 
\begin{lemma}\label{th_bneck}
The expansion of a B-necklace, with at least three blocks,  at any
vertex is a B-necklace. 
\end{lemma}

In the following theorem, the elementary separators of a signed-graphic matroid are characterized with respect to the edge set of the 
corresponding signed
 graph. 
\begin{theorem}[Zaslavsky~\cite{Zaslavsky:1991a}] \label{th_Zasl11}
Let $\Sigma$ be a connected signed graph. The elementary separators of 
$M(\Sigma)$ are the edge sets of each outer block and the core,  
except that when the core is a B-necklace each block in the
B-necklace is also an elementary separator. 
\end{theorem}

\noindent
Let $B$ be an elementary separator of $M(\Sigma)$. The subgraph ${B}$ of $\Sigma$ is called a \emph{separate} of $\Sigma$. 

The minimal subset of edges of a double bond $Y$ of a connected signed graph $\Sigma$ that disconnects $\Sigma$ shall be called the \emph{unbalancing part} of the double bond, while the remaining set of edges will be the \emph{balancing part} of $Y$. Equivalently, the unbalancing part of the double bond contains the edges of $Y$ which have exactly one endvertex in the balanced component of $\Sigma \dof Y$ while the balancing part of $Y$ consists of edges of $Y$ which have both their endvertices in the balanced component of $\Sigma \dof Y$. 

In a jointless connected and unbalanced signed graph $\Sigma$, the existence of a double bond induces the existence of an unbalancing bond. Specifically, when $Y$ is a double bond in $\Sigma$, there is an edge that has both endvertices in the balanced component of $\Sigma \dof Y$. Then the set of edges of $\Sigma$ having a common endvertex with that edge contains an unbalancing bond. This is stated in the following lemma.

\begin{lemma} \label{doubleinducesunbalancing}
In a jointless connected and unbalanced signed graph for every double bond there is an unbalancing bond.
\end{lemma}

\begin{proof}
Consider a jointless connected and unbalanced signed graph $\Sigma$ and a double bond $Y$ in $\Sigma$. Then $\Sigma \dof Y$ consists of one balanced component denoted by $\Sigma_{2}$ and some unbalanced components $S_{1}, \ldots, S_{n}$. Consider an edge $e=\{v_{1},v_{2}\}$ of $Y$ which belongs to the balancing part of $Y$ and the partition $(\{v_{1},v_{2}\}, V(\Sigma) \dof \{v_{1},v_{2}\})$ of $V(\Sigma)$. We shall denote with $H$ all the edges of $E(\Sigma)$ which have one endvertex in $\{v_{1},v_{2}\}$ and the other in $V(\Sigma) \dof \{v_{1},v_{2}\}$. Assume first that the subgraph induced by $V(\Sigma) \dof \{v_{1},v_{2}\}$ is connected. Then the later is unbalanced since it contains at least one unbalanced component of $\Sigma \dof Y$ as a subgraph. Consider without loss of generality $v_{1}$ from $\{v_{1},v_{2}\}$ and a vertex $v$ in $V(\Sigma) \dof \{v_{1},v_{2}\}$. Since $\Sigma$ is connected each $v_{1}v$-path contains an edge of $H$. The edge $e$ is a balanced component in $\Sigma \dof H$ and the deletion of $H$ from $\Sigma$ increases the number of balanced components. Moreover, the subgraph of $\Sigma$ which is obtained from $\Sigma$ by deleting all edges of $H$ apart from one is connected so $H$ is minimal with respect to the property of increasing the balanced components. All edges of $H$ have one endvertex in the balanced component $\Sigma[\{v_{1},v_{2}\}]$ of $\Sigma \dof H$ and the other in the unbalanced component $\Sigma[V(\Sigma)\dof \{v_{1},v_{2}\}]$. Thereby $H$ is an unbalancing bond in $\Sigma$. Otherwise the subgraph induced by $V(\Sigma) \dof \{v_{1},v_{2}\}$ is disconnected (see Figure~\ref{db_ub}) and consists of some connected components denoted by $B_{1}, \ldots, B_{m}$. Suppose that all $B_{i}, i=1,\ldots,m$ are unbalanced. Then $\Sigma \dof H$ consists of one balanced component $\Sigma[e]$ and $B_{i}, i=1,\ldots,m$. $H$ is minimal with respect of increasing the balanced components and all its edges have one endvertex in $\Sigma[e]$ and the other in some $B_{i}$. Therefore $H$ is an unbalancing bond. Otherwise there is at least one $B_{j}, j \in \{1,\ldots,m\}$ which is balanced. Consider a vertex $v' \in V(B_{j})$ and without loss of generality $v_{1}$ from $\{v_{1},v_{2}\}$. Each $v'v_{1}$-path contains an edge of $H$. The proper subset of edges of $H$ whose one endvertex belongs in $V(B_{j})$ and the other in $\{v_{1},v_{2}\}$ is a minimal set in $\Sigma$ whose deletion increases the number of balanced components. Thus, it is an unbalancing bond in $\Sigma$.
\end{proof}

\begin{figure}[hbtp]
\begin{center}
\psfrag{v1}{\footnotesize $v_1$}
\psfrag{v2}{\footnotesize $v_2$}
\psfrag{e}{\footnotesize $e$}
\includegraphics*[scale=0.31]{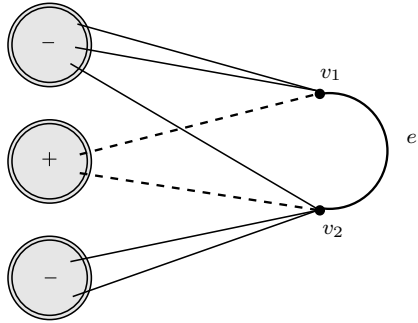}
\end{center} 
\caption{a double bond induces an unbalancing bond}
\label{db_ub}
\end{figure}

Let $Y$ be a cocircuit of a matroid $M$. For any bridge $B$ of $Y$ in $M$, we denote by $\pi(M,B,Y)$ the family of all minimal non-null subsets of $Y$ which are intersections of cocircuits of $M \cto (B \cup Y)$ (cf. \cite{Pitsoulis:2014}). In \cite{Tutte:1965}, Tutte proved that if $Y$ is a cocircuit of a binary matroid $M$ then the members of $\pi(M,B,Y)$ are disjoint and their union is $Y$. By [\cite{PapPit:2013}, Corollary~(2.14)] we derive that $\pi(M,B,Y)= \mathcal{C}^{*}(M \cto (B \cup Y) \dto Y)$. Hence the cocircuits of $M \cto (B \cup Y) \dto Y$ partition $Y$. This does not hold for quaternary matroids. We quote a counterexample where $\Sigma$ is a cylindrical signed graph whose signed-graphic matroid is quaternary non-binary, $Y$ is an a double bond in $\Sigma$ and $\pi(M(\Sigma),B,Y) \neq \mathcal{C}^{*}(M(\Sigma) \cto (B \cup Y) \dto Y)$ for some bridge $B$ of $Y$. 

\begin{figure}[h]
\begin{center}
\includegraphics*[scale=0.45]{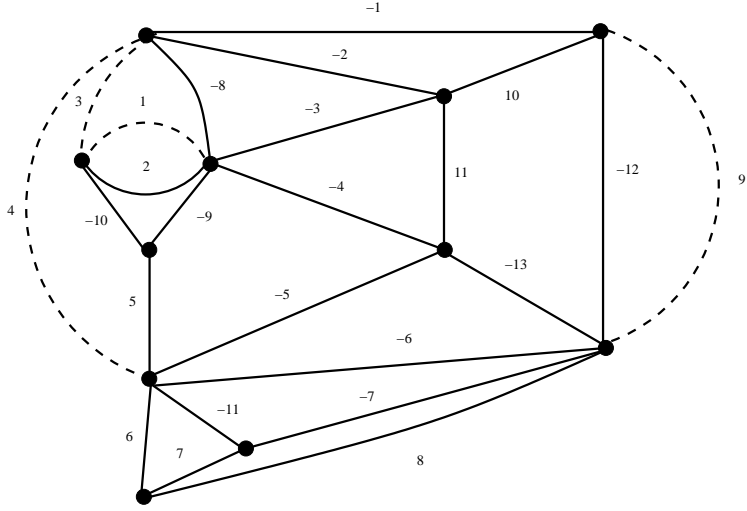}
\end{center} 
\caption{Cylindrical signed graph whose $M(\Sigma)$ is 2-connected, cylindrical2conn}
\label{cyl2conn}
\end{figure}

\begin{example}
Let $\Sigma$ be the 2-connected cylindrical signed graph in Figure~\ref{cyl2conn}. A solid line depicts an edge of positive sign while a dashed line depicts an edge of negative sign. Y=\{-1, -2, -3, -4, -5, -6, -7, 8, 9\} is a nongraphic cocircuit of $M(\Sigma)$ and a double bond in $\Sigma$. The bridges of $Y$ in $M(\Sigma)$ are $B_{1}=\{1,2,3,4,5,-8,-9,-10\}$, $B_{2}=\{6,7,-11\}$ and $B_{3}=\{10,11,-12,-13\}$. As concerns $B_{1}$, $\pi(M(\Sigma),B_{1},Y)=\{\{-1,-2\},\{-3,-4\},\{-5,-6,-7,8\},\{9\}\}= \mathcal{C}^{*}(M(\Sigma) \cto (B_{1} \cup Y) \dto Y)$. Next we present bridges $B_{2}$ and $B_{3}$ in detail. The cocircuits of $M(\Sigma).(B_{2} \cup Y)$ are 

\begin{eqnarray*}
C^{*}_{1} &=&  \{ -7, -11, 7 \} \\
C^{*}_{2} &=&  \{ -7, -11, 6, 8 \} \\
C^{*}_{3} &=&  \{ -1, -2, -3, -4, -5, -6, -11, 6, 9 \} \\
C^{*}_{4} &=&  \{  7, -11, 6, 8 \} \\
C^{*}_{5} &=&  \{ 7, -2, -3, -4, -5, -6, 8, 9, 6, -1 \} \\
C^{*}_{6} &=&  \{ 7, 8, 6, -7 \} \\
C^{*}_{7} &=&  \{ 7, -2, -3, -4, -5, -6, 8, -11, 9, -1 \} \\
C^{*}_{8} &=&  \{ 7, -2, -3, -4, -5, -6, -7, 9, 6, -1 \} \\
C^{*}_{9} &=&  \{ 8, -2, -3, -4, -5, -6, -7, 9, -1 \} 
\end{eqnarray*}

\begin{figure}[hb]
\begin{center}
\mbox{
\subfigure[Bridge $B_{1}$]
{
\includegraphics*[scale=0.28]{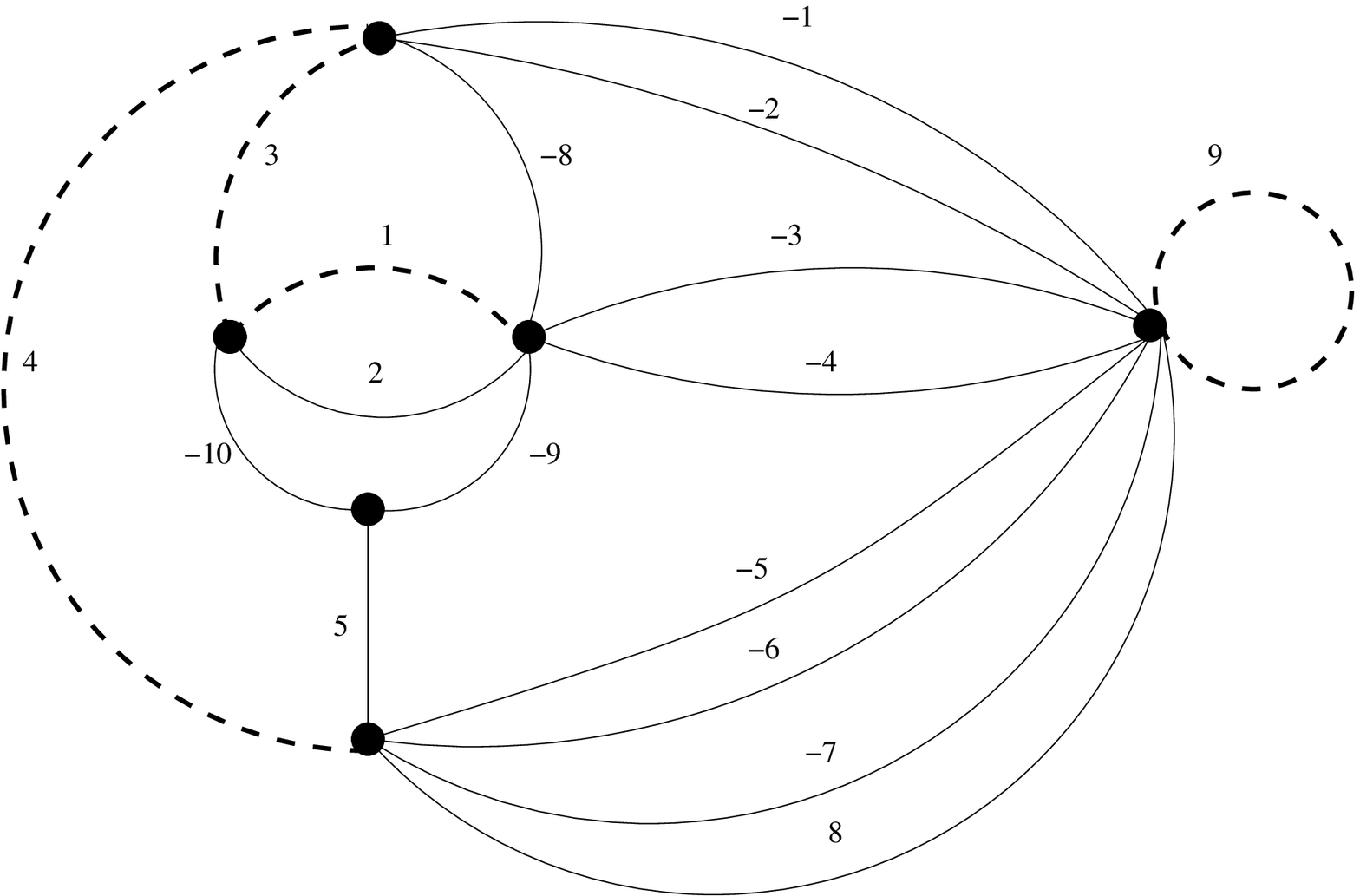}
}\quad
\subfigure[Bridge $B_{2}$]
{
\includegraphics*[scale=0.38]{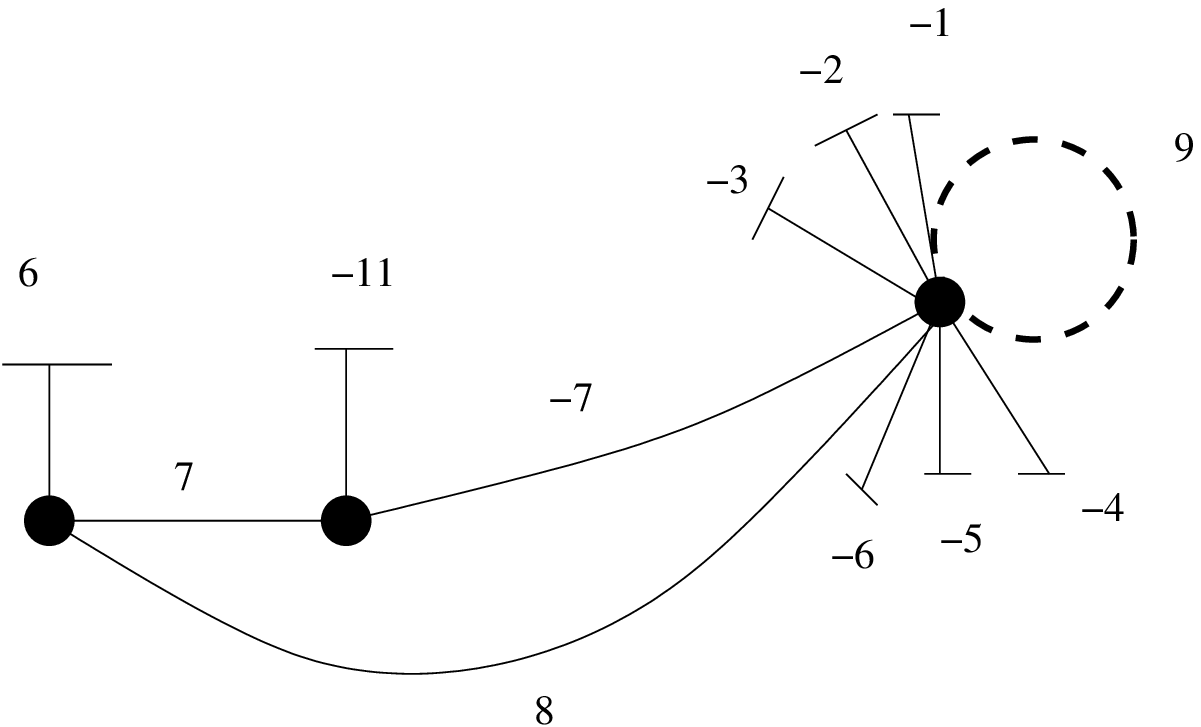}
}\quad
\subfigure[Bridge $B_{3}$]
{
\includegraphics*[scale=0.50]{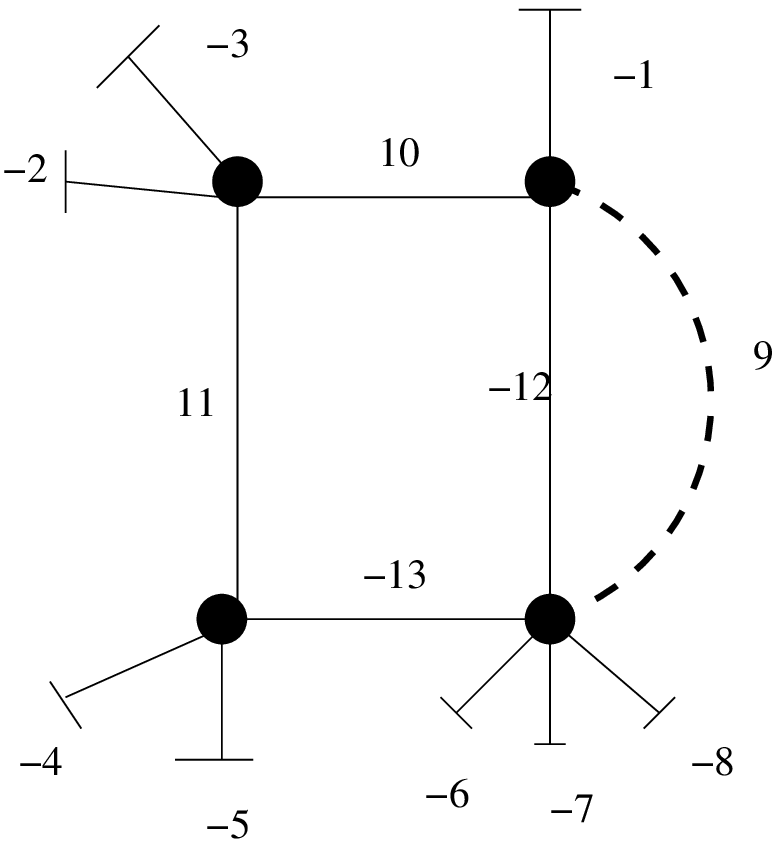}
}}
\end{center}
\caption{The $\Sigma .(B_{i} \cup Y) $ for $i=1,2,3$}
\label{Y-comp}
\end{figure}

Thus, the family of all the minimal nonempty subsets of Y which are intersections of cocircuits  of $M(\Sigma).(B_{2} \cup Y)$ is 
\begin{center}
 $\pi (M, B_{2}, Y) =\{\{-1, -2, -3, -4, -5, -6, 9 \},\{-7\}, \{8\} \}$ 
\end{center} 
The bonds of $\Sigma \cto (B_{2} \cup Y) |Y$ (see Figure \ref{Y-comp}) are the cocircuits:
\begin{center}
$\mathcal{C}^*(M(\Sigma) \cto (B_{2} \cup Y) | Y)= \{\{-7\}, \{8\},\{-1,-2,-3,-4,-5,-6,9\}\}$
\end{center}
We observe that there is partition of Y induced by the members of $\mathcal{C}^*(M(\Sigma) \cto (B_{2} \cup Y) | Y)$ and 
\begin{center}
$\pi (M, B_{2}, Y) = \mathcal{C}^*(M(\Sigma) \cto (B_{2} \cup Y) | Y)$.
\end{center}

The cocircuits of $M(\Sigma).(B_{3} \cup Y)$ are
\begin{eqnarray*}
C^{*}_{1} &=& \{ -1, -2, -3, -4, -5, -6, -7, -12, -13, 8 \} \\
C^{*}_{2} &=& \{ -1, -2, -3, -4, -5, -12, -13, 9 \} \\
C^{*}_{3} &=& \{ -2, -3, -4, -5, -13, 10 \} \\
C^{*}_{4} &=& \{ -4, -5, -13, 11 \} \\
C^{*}_{5} &=& \{ 8, -6, -7, -12, -13, 9 \} \\
C^{*}_{6} &=& \{ 8, 10, -6, -7, -12, -13, -1\} \\
C^{*}_{7} &=& \{ 9, 10, -12, -1 \} \\
C^{*}_{8} &=& \{ 10, -3, 11, -2 \} \\
C^{*}_{9} &=& \{ 8, 10, -6, -7, 9, -13, -1 \} \\
C^{*}_{10} &=& \{ 9, 8, -5, -6, -7, -12, 11, -4 \} \\
C^{*}_{11} &=& \{ 8, 10, -4, -5, -6, -7, 9, 11, -1 \} \\
C^{*}_{12} &=& \{ 9, 8, -3, -4, -5, -6, -7, -12, 10, -2 \} \\
C^{*}_{13} &=& \{ 8, 10, -4, -5, -6, -7, -12, 11, -1 \} \\
C^{*}_{14} &=& \{ 8, -2, -3, 11, -6, -7, -12, -13, -1 \} \\
C^{*}_{15} &=& \{ 9, -2, -3, 11, -12, -1 \} \\
C^{*}_{16} &=& \{ 8, -2, -3, 11, -6, -7, 9, -13, -1 \} \\
C^{*}_{17} &=& \{ 8, -2, -3, -4, -5, -6, -7, 9, -1 \} \\
C^{*}_{18} &=& \{ 8, -2, -3, -4, -5, -6, -7, -12, 10, -1 \} \\
C^{*}_{19} &=& \{ 8, -2, -3, -4, -5, -6, -7, -12, 11, -1 \}
\end{eqnarray*}
 
Thus, $\pi (M, B_{3}, Y) =\{\{-6,-7,8\}, \{-4,-5\}, \{-1\}, \{9\}, \{-2,-3\}\}$ and the bonds of $\Sigma . (B_{3} \cup Y) |Y$ (see figure~\ref{2-csg_B3|Y}) are the cocircuits

\begin{center}
$\mathcal{C}^*(M(\Sigma) \cto (B_{3} \cup Y) | Y)= \{\{-1,-6,-7,8\}, \{-1,9\},\{-2,-3\},\{-4,-5\},\{-6,-7,8,9\}\}$
\end{center}

\begin{figure}[ht]
\begin{center}
\centering
\includegraphics*[scale=0.40]{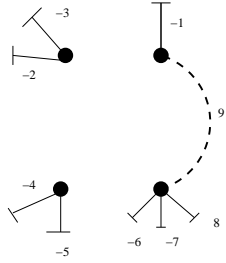}
\end{center}
\caption{$\Sigma . (B_{3} \cup Y) |Y$}
\label{2-csg_B3|Y}
\end{figure}

We observe that there is partition of Y induced by the members of $\pi (M, B_{3}, Y)$ but 
\begin{center}
$\pi (M, B_{3}, Y) \neq \mathcal{C}^*(M(\Sigma) \cto (B_{3} \cup Y) | Y)$.
\end{center}

\end{example}


\subsection{Known decomposition results} \label{subsec_known_results}
In \cite{PapPit:2013} the following decomposition result for the case of binary signed-graphic matroids is given.
\begin{theorem}[Papalamprou, Pitsoulis~\cite{PapPit:2013}]\label{thrm_decomposition}
Let $M$ be a connected binary matroid and $Y\in \mathcal{C}^{*}(M)$ be a bridge-separable cocircuit such that 
$M \dof Y$ is not  graphic. Then the $Y$-components of $M$ are graphic except for one which is signed-graphic
if and only if $M$ is signed-graphic.
\end{theorem}
In this work we will utilise the above result along with the following result appearing in~\cite{SliQin:07} in order to decompose the class of quaternary signed-graphic matroids.
\begin{theorem}[Pagano~\cite{SliQin:07}] \label{th_SliDe}
If $\Sigma$ is connected and $M(\Sigma)$ is quaternary, then either
\begin{itemize}
\item[(1)] $M(\Sigma)$ is binary,
\item[(2)] $\Sigma\backslash{J_{\Sigma}}$ has a balancing vertex,
\item[(3)]$\Sigma\backslash{J_{\Sigma}}$ is cylindrical,
\item[(4)]$\Sigma\backslash{J_{\Sigma}}\cong{T_6}$, or
\item[(5)] $\Sigma\backslash{J_{\Sigma}} = Y_1 \oplus_{k} Y_2$ for $k\in{\{1, 2, 3\}}$ where each $M(Y_i)$ is quaternary.
\end{itemize}
Also, if $\Sigma$ is a connected signed graph that satisfies one of (1)-(4), then $M(\Sigma)$ is quaternary.
\end{theorem}
In the sections that follow we shall examine separately each of the classes mentioned in the above result starting from the class of cylindrical signed graphs.

\section{Matroids of cylindrical signed graphs: Structural results} \label{sec_structural_cylindrical_matroids}

Let $P$ denote the plane. Consider a cylindrical signed graph $\Sigma$ and its planar embedding with at most two negative faces denoted also by $\Sigma$. Let $C$ be a cycle of $\Sigma$ of arbitrary sign that bounds a disc $D$. Let $H$ be the subgraph of $\Sigma$ which lies in the interior of $D$. The connected components of $P \dof H$ are the faces of $H$ with respect to $C$. The infinite face is called the outer face of $H$ and all the rest are the inner faces of $H$. An inner face of $H$ with respect to $C$ is said to be contained in $C$. Two faces are \emph{incident} if they share at least one edge and they are \emph{vertex-disjoint} if they have no vertex in common.

In the rest of the paper we assume that $\Sigma$ is 2-connected. Most results can be easily extended to the 1-connected case by some technical modifications. In the result for graphs that follows we show that in a 2-connected plane graph every face contains a face that is incident to its cycle-boundary. Using this result in the next three lemmas it is proved that a negative face is contained in every negative cycle, that any positive cycle contains an even number of negative faces and that the number of the negative faces in a 2-connected plane signed graph is even.

\begin{lemma} \label{lem_PCF}
In a planar 2-connected graph $G$, with outer cycle $C_{0}$ and $C$ a nonempty path of $C_{0}$ there is a face $F$ in $G$ (other than the outer) such that $E(C) \cap E(F) \neq \emptyset $ and $G[ E(C) \cap E(F) ]$ is a path.
\end{lemma}
\begin{proof}
Let $F$ be a face in $G$ such that $E(C) \cap E(F) \neq \emptyset$, since any edge of $C$ is adjacent to exactly two faces. If $G[ E(C) \cap E(F) ]$ is a path it holds. Otherwise we delete $E(F) \cap E(C_{0})$ from $G$. Since we delete disconnected paths from $C_{0}$, more than one components are formed in $G$ that contain the remaining faces of $G$. Moreover, since we delete $G[ E(C) \cap E(F) ]$ disconnected paths from $C$, some components will contain the paths of $C$ which do not belong to $F$. Let $K$ be the components of $G$ such that $E(K) \cap E(C) \neq \emptyset $, then $G[ E(C) \cap E(F) ]$ is a path. Because if after the deletion of the disconnected paths of $C$, contained in $E(F) \cap E(C_{0})$, the common edges of $C$ and the components are not paths then $C$ is not a path.  We continue inductively, as long as there are still edges of $C$. Thus there are faces by planarity and components $K$ that contain them. In a component $K$, we choose either a face $F$ such that $G[ E(C) \cap E(F) ]$ is a path, or after the deletion of the disconnected paths $G[ E(F) \cap E(C_{0}) ]$ we choose a new component $K$ with fewer faces. We end up with a minimum component $K$ regarding the number of faces which have common edges with $C$. Thus, the minimum component $K$ is a face $F$ that has common edges with $C$ such that $G[ E(C) \cap E(F) ]$ is a path. If not, then $C$ is not a path. 
\end{proof}
 
\begin{lemma} \label{lem_PNCF}
In a planar signed graph, every negative cycle contains a negative face. 
\end{lemma}
\begin{proof}
Let $C$ be a negative cycle in a planar signed graph. If $|f_{C}|=1$ it holds. Let $F$ be a face contained in $C$ such that their common edges form a path. By Lemma~\ref{lem_PCF} consider the case  where the path $C$ is the cycle. We shall denote with $P= E(C) \cap E(F)$, $H=E(F) \setminus E(C)$ and $K=E(C) \setminus E(F)$. If for the  edges $ (E(C) \cap E(F))^{-} $ contained in $P$,  $| (E(F) \cap E(C))^{-}| $ odd then $P$ is a negative path denoted $P^{-}$. If $ E^{-}(F) \setminus E^{-}(C) $ contained in $H$ it holds that $| E^{-}(F) \setminus E^{-}(C)| $ even then $H$ is a positive path denoted $H^{+}$. Thus $F$ which consists of $P^{-}$ and $H^{+}$ is a negative face and there is nothing more to prove. If $| (E^{-}(F) \setminus E^{-}(C))|$ is odd then $H$ is a negative path, denoted $H^{-}$, therefore $F$ consisting of $P^{-}$ and $H^{-}$ is a positive face. As a result $| (E^{-}(C) \setminus E^{-}(F))|$ contained in $K$ is even and $K$ is a positive path denoted $K^{+}$. The cycle which is formed by $H^{-}$ and $K^{+}$ is negative.
    
Respectively, we assume $P^{+}$ and we have either $| (E^{-}(F) \setminus E^{-}(C))|$ odd, $H^{-}$, and $F$ is a negative face or $| (E^{-}(F) \setminus E^{-}(C))|$ even, $H^{+}$, and $F$ is a positive face. $| (E^{-}(C) \setminus E^{-}(F))|$ contained in $K$ is odd and $K$ is a negative path denoted $K^{-}$. As a result the cycle which is formed by $H^{+}$ and $K^{-}$ is negative. The number of faces is finite and inductively we come across either a negative face or a negative cycle.
\end{proof}

\begin{lemma} \label{lem_COE}
In a planar signed graph, every negative cycle contains an odd number of negative faces whereas every positive cycle contains an even number.
\end{lemma}
\begin{proof}
Let $C$ be a cycle. We apply induction on the number of faces contained in it. If $C$ contains one face it holds. For the induction hypothesis, we assume that it holds for every $C$ which has fewer faces than $n$, i.e. every negative cycle with fewer faces than n contains an odd number of negative faces whereas every positive cycle contains an even number. We will show that it holds for every $C$ with $n$ faces. Since $C$ contains $n$ faces there is a path contained in it whose endvertices only belong to $V(C)$. Thus a theta graph is formed by the outer cycle and the path. If $C$ is negative, it is divided into a positive and a negative cycle which by the induction hypothesis have even and odd number of negative faces respectively. Thus $C$ has odd number of negative faces. If $C$ is positive, it is divided either into two positive or two negative cycles. In any case, by the induction hypothesis $C$ has even number of negative faces. 
\end{proof}
 
From the above result, it is evident that, in a planar signed graph, the sign of a cycle is equal to the product of the signs of the faces contained in it.

\begin{lemma} \label{allfaceseven}
The number of negative faces in a 2-connected planar signed graph is even.
\end{lemma}
\begin{proof}
Let us first examine the case in which the planar signed graph is $2$-connected. In the subcase that the outer cycle is negative then, by Lemma~\ref{lem_COE}, it contains an odd number of faces. The unbounded face of that graph is also negative and therefore there is an even number of negative faces. The subcase that the outer cycle is positive follows similarly.

Let us now consider the case in which a planar signed graph $\Sigma$ is connected and consists of several blocks. Take a connected subgraph of $\Sigma$ consisting of two blocks $B_1$ and $B_2$. The outer cycles of $B_1$ and $B_2$ may be positive or negative. In all (four) cases, using Lemma~\ref{lem_COE} and taking into account the sign of the unbounded face, it can be shown that the number of negative faces is an even number $f$. Now let us describe a procedure: let $B_3$ be another block such that $B_1\cup{B_2}\cup{B_3}$ is connected. Then, if the outer cycle of $B_3$ is positive then it contains an even number of negative faces while the unbounded face does not change sign; while if  the outer cycle of $B_3$ is negative then it contains an odd number of negative faces while the unbounded face does change sign. In both cases we have that  in the subgraph $B_1\cup{B_2}\cup{B_3}$  there exist an even number of negative face. By applying the same procedure so that all blocks of $\Sigma$ are considered, tne result follows.  
\end{proof}

Using the fact that a cylindrical signed graph can have at most two negative faces:
\begin{corollary}
In a 2-connected cylindrical signed graph, there are zero or two negative faces.
\end{corollary}

The above structural properties of cylindrical signed graphs enabled us to derive results regarding the associated signed-graphic matroids. Specifically, the number of negative faces and their adjacency in a planar embedding of a cylindrical signed graph declares whether the associated signed-graphic matroid is binary.

\begin{lemma} \label{lem_cybin}
In a cylindrical signed graph $\Sigma$, if there are no negative faces or the two negative faces are not vertex-disjoint, $M(\Sigma)$ is binary.
\end{lemma}

\begin{proof}
If $\Sigma$ has no negative faces then from Lemma~\ref{lem_COE}, $\Sigma$ is balanced and $M(\Sigma)$ is binary. In case $\Sigma$ has two negative faces, assume on the contrary that $M(\Sigma)$ is not binary. Then $\Sigma$ has two vertex-disjoint negative cycles. By Lemma~\ref{lem_PNCF},  the negative faces within the boundaries of these negative cycles are vertex-disjoint; a contradiction. 
\end{proof}

\begin{lemma} \label{lem_cygra}
In a cylindrical signed graph $\Sigma$, if there are no negative faces or the two negative faces are not vertex-disjoint, $M(\Sigma)$ is graphic.
\end{lemma}

\begin{proof}
If $\Sigma$ has no negative faces then from Lemma~\ref{lem_COE}, $\Sigma$ is balanced and $M(\Sigma)$ is graphic. Otherwise the two negative faces, say $C_{1}$ and $C_{2}$, are not vertex-disjoint. Let $V=V(C_{1}) \cap V(C_{2})$ and consider a negative cycle $C$ of $\Sigma$, where $C\neq{C_1,C_2}$. By Lemma~\ref{lem_COE}, $C$ contains exactly one from $C_1$ and $C_2$; say $C_1$. By planarity combined with the fact that $C$ does not contain $C_2$ gives us that $V\subseteq{V(C)}$. All the vertices in ${V}$, which by hypothesis is non-empty, are also vertices of any negative cycle of $\Sigma$. Therefore, any vertex in $V$ is a balancing vertex of $\Sigma$ which implies that $M(\Sigma)$ is graphic. 
\end{proof}

\begin{lemma}\label{lem_cycyci}
In a cylindrical signed graph $\Sigma$, if the two negative faces are vertex-disjoint, $M(\Sigma)$ nonbinary.
\end{lemma}
\begin{proof}
Since there are two vertex-disjoint negative faces, there are two vertex-disjoint negative cycles thus $M(\Sigma)$ nonbinary.
\end{proof}

The following theorem determines the number of negative faces and their adjacency in a planar embedding of a cylindrical signed graph whose corresponding signed-graphic matroid is graphic. Combining Lemmata~\ref{lem_cygra} and~\ref{lem_cycyci} the following result is easily obtained.

\begin{theorem} \label{th_cygicr}
In a cylindrical signed graph $\Sigma$, there are no negative faces or the two negative faces are not vertex disjoint iff $M(\Sigma)$ graphic.
\end{theorem}

\subsection{Structural results for cylindrical signed graphs} \label{subsec_structural_cylindrical_graphs}

The next results describe structural characteristics of a special planar embedding of cylindrical signed graphs. 

\begin{claim}
Every cylindrical signed graph which has a planar embedding with a negative face has also a planar embedding with a negative outer face.
\end{claim}
\begin{proof}
Assume a planar embedding of a cylindrical signed graph with a negative face. Moreover, assume that the outer face of it is positive. Copy the planar embedding onto the sphere so that the north pole lies in the interior of one of the negative faces. Then apply the Riemann stereographic projection. Thus, there is another planar embedding where the boundary of the negative face is the outer. 
\end{proof}

When no confusion arises, we will refer to the faces of a planar embedding of a cylindrical signed graph as faces of the cylindrical signed graph. Moreover, in all proofs concerning cylindrical signed graphs with non-binary signed-graphic matroids we shall always consider their planar embedding where the outer face is negative. 

\begin{claim} \label{connectedcomponents}
If $\Sigma$ is a 2-connected cylindrical signed graph with $M(\Sigma)$ quaternary non-binary and $Y$ is an non-balancing bond in $\Sigma$ then 
\begin{enumerate}[(i)]
\item $\Sigma \dof Y$ consists of one balanced and one unbalanced connected component, 
\item $\Sigma \dof Y$ has one unbalanced separate.
\end{enumerate}
\end{claim}
\begin{proof}
For (i) by Theorem 7 $\Sigma$ has a planar embedding with two vertex disjoint negative faces. Consider the planar embedding which has a negative outer face. Since $Y$ is a non-balancing bond in $\Sigma$, $\Sigma \dof Y$ has exactly one balanced component. Assume that the unbalanced components of $\Sigma \dof Y$ are two or more. Each one has a negative cycle. Thus by lemma~\ref{lem_PNCF}, there are two or more distinct negative faces different from the outer which is a contradiction to the assumption that the planar embedding has exactly two negative faces. For (ii) the argument is the same. 
\end{proof}
 
The following three results present structural characteristics of the planar embedding of a cylindrical signed graph with a negative outer face that a double bond imposes. It is shown that any negative face can contain at most one edge of the balancing part of a double bond and that the balancing part of a double bond consists only of one edge. Moreover, a technical result is proved according to which the unique edge of the balancing part of a double bond has a common endvertex with an edge of its unbalancing part.
 
\begin{claim} \label{balancingpartY}
If $\Sigma$ is a 2-connected cylindrical signed graph with $M(\Sigma)$ connected quaternary non-binary, $\Sigma'$ is the planar embedding of $\Sigma$ with two negative faces and $Y$ is a nongraphic cocircuit and a double bond of $\Sigma'$, then a negative face of $\Sigma'$ can contain exactly one edge of the balancing part of $Y$. 
\end{claim}
\begin{proof}
Suppose that $\Sigma'$ is the planar embedding of $\Sigma$ with two negative faces that are vertex disjoint by theorem~\ref{th_cygicr} and one of them is the outer. By hypothesis $Y$ is a double bond and by claim~\ref{connectedcomponents}, $\Sigma' \dof Y$ consists of one unbalanced and one balanced component denoted by $\Sigma_{1}$ and $\Sigma_{2}$ respectively.  By proposition~\ref{switchings}, we may perform switchings and make all edges of the balanced separates of $\Sigma \dof Y$ positive. It follows that only $Y$ and the core may contain edges of negative sign. Due to the balancing part of $Y$, $Y_{2}$, being a balancing bond of the unbalanced subgraph $\Sigma_{2} \cup Y_{2}$ and $\Sigma_{2}$ having all positive edges, we can perform switchings and make all edges of $Y_{2}$ negative. Since $\Sigma$ is 2-connected all faces of $\Sigma'$ are bounded by cycles. Assume that there exists a negative face $F$ which contains more than one edges of $Y_{2}$. Obviously their number in $F$ is odd and they cannot have the same endvertices. Since every edge of $Y_{2}$ has both endvertices to $\Sigma_{2}$, $F$ is a negative cycle of $\Sigma_{2} \cup Y_{2}$. Moreover, $F$ consists of edges of $Y_{2}$ whose endvertices are connected by positive paths of $\Sigma_{2}$. Due to the fact that $\Sigma'$ has only two negative faces, the one contained in the core of $\Sigma' \dof Y$ and the outer, $F$ is clearly the outer. Consider a common negative edge $e=\{v_{1},v_{2}\}$ of $Y_{2}$ and $F$. There is a positive $v_{1}v_{2}$-path of $\Sigma_{2}$ that belongs to the outer face of $\Sigma' \dof e$. The latter path and the cycle-boundary of $F$ form a theta graph. The $v_{1}v_{2}$-path of $F$, not $e$, is positive since it contains an even number of edges of $Y_{2}$. Thus, the cycle formed by $v_{1}v_{2}$-path of $\Sigma_{2}$ and the $v_{1}v_{2}$-path of $F$ is positive. Therefore by lemma~\ref{lem_COE} it contains at least two negative faces, since it contains also the face of the core. This is a contradiction because the outer is a third negative face. 
\end{proof}

\begin{proposition} \label{uniqueedge}
If $\Sigma$ is a 2-connected cylindrical signed graph with $M(\Sigma)$ connected quaternary non-binary and $Y$ is a nongraphic cocircuit and a double bond of $\Sigma$, then the balancing part of $Y$ contains one edge.
\end{proposition}
\begin{proof}
Consider a planar embedding of $\Sigma$ with two negative faces that are vertex disjoint by theorem~\ref{th_cygicr} and one of them is the outer. By hypothesis $Y$ is a double bond and by claim~\ref{connectedcomponents}  $\Sigma \dof Y$ consists of one unbalanced and one balanced component denoted by $\Sigma_{1}$ and $\Sigma_{2}$ respectively. By proposition~\ref{switchings}, we may perform switchings and make all edges of the balanced separates of $\Sigma \dof Y$ positive. It follows that only $Y$ and the core may contain edges of negative sign. Due to $Y_{2}$ being a balancing bond of the subgraph $\Sigma_{2} \cup Y_{2}$ and $\Sigma_{2}$ having all positive edges, we can perform  switchings and make all edges of the balancing part of $Y$, $Y_{2}$, negative. Since $\Sigma$ is 2-connected all faces of $\Sigma$ are bounded by cycles. Assume that $Y_{2}$ contains $n$ edges with $n \geq 2$. By minimality of $Y$, each edge of $Y_{2}$ belongs to a negative cycle in $\Sigma_{2} \cup Y_{2}$ which contains only one of them. By lemma~\ref{lem_PNCF} each of the aforementioned cycles contains a negative face. Since the edges of $Y_{2}$ are the only negative edges of $\Sigma_{2} \cup Y_{2}$, they must belong to the cycle-boundaries of the negative faces. By claim~\ref{balancingpartY} the cycle-boundaries of the negative faces contain exactly one edge of $Y_{2}$. Due to the fact that $\Sigma$ has only two negative faces, the one contained in the core of $\Sigma \dof Y$ and the outer, the aforementioned faces coincide and constitute the outer face of $\Sigma$. However, the cycle-boundaries of the negative faces are distinct since each of them contains exactly one edge of $Y_{2}$ that cannot belong to any other. Thus we reach a contradiction.
\end{proof}

Consider a 2-connected cylindrical signed graph $\Sigma$ with $M(\Sigma)$ connected quaternary non-binary and $Y$ a nongraphic cocircuit and a double bond of $\Sigma$. From the above lemma we deduce that the outer negative face of a planar embedding of $\Sigma$ is formed by the unique edge of the balancing part of $Y$ and a path of the balanced component of $\Sigma \dof Y$ connecting its endvertices.

\begin{proposition} \label{commonendvertex}
If $\Sigma$ is a 2-connected cylindrical signed graph with $M(\Sigma)$ internally 4-connected quaternary non-binary and $Y$ is a nongraphic cocircuit and a double bond of $\Sigma$, then the unique edge of the balancing part of $Y$ has a common endvertex with an edge of the unbalancing part of $Y$.
\end{proposition}
\begin{proof}
Consider a planar embedding of $\Sigma$ with two vertex disjoint negative faces where one of them is the outer. We denote this embedding by $\Sigma'$. By hypothesis $Y$ is a nongraphic cocircuit and a double bond in $\Sigma'$. Let $\Sigma_{1}, \Sigma_{2}$ be the unbalanced and the balanced connected component of $\Sigma' \dof Y$ respectively. Moreover, let $Y_{1},Y_{2}$ be the unbalancing and the balancing part of $Y$ respectively and $e=\{x,y\}$ be the unique edge of $Y_{2}$ by proposition~\ref{uniqueedge}. By lemma~\ref{doubleinducesunbalancing}, $Y$ induces the existence of an unbalancing bond with elements the edges of $\Sigma_{2}$ incident to the endvertices of $e$, denoted by $Y_{u}$. Assume that $Y_{u}$ has no common edge with $Y_{1}$. Then $\Sigma' \dof Y_{u}$ consists of two connected components, a balanced component, the edge $e$ and an unbalanced component containing $Y_{1}$. Then $Y_{1}$ and the edges of $Y_{u}$ incident with only one endvertex of $e$, say $x$, constitute another double bond. The deletion of the aforementioned double bond from $\Sigma'$ has two connected components $\Sigma_{1}$ as unbalanced component and a balanced component consisting of $\Sigma_{2},e$ and the edges of the unbalancing bond that are incident with $y$. Since $Y$ is a nongraphic cocircuit so is the new double bond since the deletion of the latter from $\Sigma'$ and $\Sigma \dof Y$ have the same core. Since all negative cycles of $\Sigma_{2} \cup e$ contain $e$, the deletion of the edges of $\Sigma_{2}$ which are incident to $x$ or $y$ from $\Sigma_{2} \cup e$ is connected and balanced. The new double bond has a balancing part with more than one edges which is a contradiction to proposition\ref{uniqueedge}. Therefore $Y_{u}$ must consist of two edges of $\Sigma_{2}$ each incident with one endvertex of $e$. However the unbalancing bond and $e$ are the one part of a 2-biseparation of $\Sigma$ which is a contradiction since by theorem~\ref{3biconnected} $\Sigma$ is 3-biconnected with only minimal 3-biseparations.   
\end{proof}

\begin{figure}[h]
\begin{center}
\centering
\psfrag{e}{\footnotesize $e$}
\psfrag{Yu}{\footnotesize $Y_{u}$}
\psfrag{Y1}{\footnotesize $Y_{1}$}
\includegraphics*[scale=0.4]{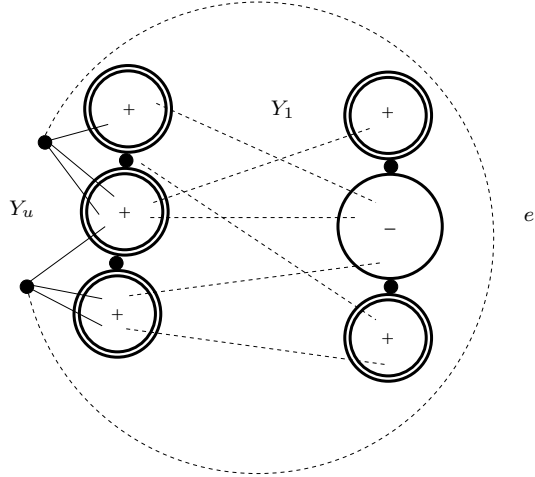}
\end{center}
\caption{$Y_{1} \cup e$ double bond of the signed graph}
\label{Y1Y2}
\end{figure}

We observe that the balancing part of a double bond may consist of one or more edges. However, from the examples we deduce that when $\Sigma$ is a 2-connected cylindrical signed graph with $M(\Sigma)$ internally 4-connected quaternary non-binary and $Y \in \mathcal{C}^{*}$ is a double bond of $\Sigma$ such that the balancing part of $Y$ has more than one elements then $Y$ is a graphic cocircuit of $M(\Sigma)$.

\section{Matroids of cylindrical signed graphs: Decomposition} \label{sec_decomposition}

Suppose that $B$ is a bridge of $Y$ in $M(\Sigma)$ and a separate in $\Sigma$. Moreover, let $\Sigma_{i}$ be the component of $\Sigma \dof Y$ such that $B \subseteq \Sigma_{i}$. Then, if $v$ is a vertex of $V(B)$, we denote by $C(B,v)$ the component of $\Sigma_{i} \dof B$ having $v$ as a vertex. We denote by $Y(B,v)$ the set of all $y \in Y$ such that exactly one end of $y$ in $\Sigma$ is a vertex of $C(B,v)$. 

It is a known fact that any face of a plane graph can become its outer. Thereby a cylindrical signed graph that has a planar embedding with a negative face, has also a planar embedding with a negative outer face. The aforementioned embedding will be considered in all of the following proofs. The next lemma derives easily from the orthogonality property of matroids, i.e., in a matroid $M$, let $C$ be a circuit and $C^{*}$ be a cocircuit, then $|C \cap C^{*}| \neq 1$. 

\begin{lemma}
The boundary of a face in a planar embedding of a $2$-connected cylindrical signed graph $\Sigma$ contains zero or two edges of an unbalancing bond $Y$ in $\Sigma$.
\end{lemma}
\begin{proof}
Since $\Sigma$ is cylindrical and $Y$ is an unbalancing bond in $\Sigma$ then $\Sigma \dof Y$ consists of one unbalanced and one balanced connected component denoted by $\Sigma_1$ and $\Sigma_2$, respectively. It is well-known that any cycle of a graph intersects any bond in an even number of edges. Therefore, since a face may also be viewed as a cycle (see Proposition~4.2.6 in~\cite{Diestel:05}), it remains to show that the boundary of any face can not contain more than two edges of $Y$. By way of contradiction, assume that the boundary of a face $F$ in $\Sigma$ has more than two common edges with $Y$. Let us traverse $F$ starting from an edge $y_1$ of $Y$ with endpoint $v_1$ in $\Sigma_1$ while let us call $y_2$ the next edge of $Y$ that we encounter in that traversal. Let also $v_2$ be the endpoint of $y_2$ in $\Sigma_1$. By the fact that for any two points of the plane lying in $F$ there exists a simple curve joining them (without crossing any edge), we can say that there is no path connecting $v_1$ and $v_2$ in $\Sigma_1$; a contradiction, since $\Sigma_1$ is connected.
\end{proof}

\begin{lemma} \label{lem_03}
Let $Y$ be a non-balancing bond of a $2$-connected cylindrical signed graph $\Sigma$ with $M(\Sigma)$ being quaternary and non-binary. Then for each separate $B$ of the unbalanced component of $\Sigma\backslash{Y}$ there exists at most one vertex of attachment $v\in{V(B)}$ with balanced $C(B,v)$ such that $Y(B,v)$ consists of edges of different sign. 
\end{lemma}
\begin{proof}
By Theorem~\ref{th_cygicr}, the $2$-connected cylindrical signed graph $\Sigma$ has a planar embedding with two vertex-disjoint negative faces. Since $Y$ is a non-balancing bond in $\Sigma$ then $\Sigma \dof Y$ consists of one balanced and one unbalanced connected component denoted by $\Sigma_1$ and $\Sigma_2$, respectively. Assume that we perform switchings so that all edges of the balanced separates of $\Sigma\backslash{Y}$ are positive. 

Suppose that $Y$ is an unbalancing bond. By way of contradiction, let $v_1$ and $v_2$ be two vertices of attachment of a bridge $B$ of $Y$ in $\Sigma_1$ such that each $C(B,v_i)$ is balanced and each $Y(B,v_i)$ consists of edges of different sign $(i=1,2)$. Let us call $y_i=(u_i,w_i)$ and $y_i'=(u_i',w_i')$ two edges of different sign in $Y(B,v_i)$, where $u_i$ and $u_i'$ are vertices in $\Sigma_1$ and $w_i$ and $w_i'$ are vertices in $\Sigma_2$. Due to the fact that $C(B,v_i)$ is balanced, there exists a positive path $P_i$ between $u_i$ and $u_i'$ in $\Sigma_i$ while, due to the fact that $\Sigma_i$ is balanced, there exists a positive path $P_i'$ between $w_i'$ and $w_2'$. Therefore, the cycle $C_i$ formed by $P_i$, $P_i'$, $y_i$ and $y_i'$ is of negative sign. Hence, by Lemma~\ref{lem_PNCF}, a negative face $F_i$ is contained in each $C_i$. Moreover, we can not have $F_1=F_2$, due to the fact that the common vertices that $C_1$ and $C_2$ may have, are vertices of $\Sigma_2$ (i.e. common vertices of $P_1'$ and $P_2'$). Finally, each of $F_1$ and $F_2$ is distinct from the negative face contained in the unbalanced block of $\Sigma_1$ since they have different boundaries. This means that $\Sigma$ had three distinct negative faces which is in contradiction with the hypothesis saying that $\Sigma$ is cylindrical. 
\end{proof}

Given any two bridges $B_1, B_2$ of $Y$ in a matroid $M$, $B_1, B_2$ are \emph{avoiding} if there exist cocircuits $C^{*}_{1} \in \mathcal{C}^{*}(M \cto (B_{1} \cup Y) \dto Y)$ and $C^{*}_{2} \in \mathcal{C}^{*}(M \cto (B_{2} \cup Y) \dto Y)$ such that $C^{*}_{1} \cup C^{*}_{2}=Y$. A cocircuit Y is called \emph{bridge-separable} if the bridges formed upon its deletion can be partitioned into two classes such that all members of the same class are avoiding with each other. 

In the following two theorems it is shown that an unbalancing bond and a double bond in a jointless cylindrical signed graph constitute cocircuits of the corresponding signed-graphic matroid with the bridge-separability property. 

\begin{theorem} \label{th_Ybridgeseparable}
Let $\Sigma$ be a $2$-connected jointless cylindrical signed graph such that $M(\Sigma)$ is quaternary and non-binary. If $Y$ is an unbalancing bond of $\Sigma$ then $Y$ is bridge-separable in $M(\Sigma)$.
\end{theorem}
\begin{proof}
Since $M(\Sigma)$ is non-binary and, therefore, non-graphic, we have, by Theorem~\ref{th_cygicr}, that $\Sigma$ has a planar embedding with two vertex disjoint negative faces. Consider the planar embedding of $\Sigma$ with a negative outer face denoted also by $\Sigma$ and suppose that $Y$ is an unbalancing bond. Let us call $\Sigma_1$ and $\Sigma_2$ the unbalanced and balanced component of $\Sigma\backslash{Y}$, respectively, and let $B_0$ be 
the unique unbalanced separate in $\Sigma_1$. Moreover, due to the fact that switching at vertices of $\Sigma$ do not alter $M(\Sigma)$, we can assume that all the edges in the balanced separates of $\Sigma\backslash{Y}$ are positive.  Consider any pair of bridges $B_1, B_2$ in either $M(\Sigma_1)$ or $M(\Sigma_2)$. To prove the theorem it suffices to show that
 there exist cocircuits $C_1^{*}\in{\mathcal{C}^{*}(M(\Sigma).(B_1\cup{Y})|{Y})}$ and $C_2^{*}\in{\mathcal{C}^{*}(M(\Sigma).(B_2\cup{Y})|{Y})}$ 
such that $C_1^{*}\cup{C_2^{*}}=Y$.
Let $v_1\in{V(B_1)}$ and $v_2\in{V(B_2)}$ be the vertices of attachment such that $B_2$ is contained in ${C(B_1,v_1)}$ and  $B_1$ is contained in ${C(B_2,v_2)}$, respectively. 
We have that $V(\Sigma_i)\subseteq{V(C(B_1,v_1))\cup{V(C(B_2,v_2))}}$ which implies that 
\footnote{For our purposes: This is the difference between unbalancing bonds and double bonds.} 
\begin{equation}
Y(B_1,v_1)\cup{Y(B_2,v_2)}=Y
\end{equation}

In each of the following cases we will show that there exist cocircuits $C_i^{*}\in{\mathcal{C}^{*}(M(\Sigma).(B_i\cup{Y})|{Y})}$, for $i=1,2$, such
 that $Y(B_1,v_1)\cup{Y(B_2,v_2)}\subseteq{C_1^{*}\cup{C_2^{*}}}$ which implies that $B_1$ and $B_2$ are avoiding bridges. 
If $B_1$ and $B_2$ are separates of $\Sigma_2$ they are both balanced. By the definition of contraction in signed graphs in \S\ref{subsec_signed_graphs}
since all the edges in the unbalanced component $\Sigma_1$ will be contracted, the graph $\Sigma.(B_1\cup{Y})|Y$ will consist of half-edges only attached to the
vertices of attachment of $B_1$. 
The edges $Y(B_1,v_1)$ are half-edges attached at $v_1$ in that graph, therefore $C_1^{*}=Y(B_1,v_1)\in{\mathcal{C}^{*}(M(\Sigma).(B_1\cup{Y})|{Y})}$. 
Similarly, we can find a cocircuit $C_2^{*}=Y(B_2,v_2)$.

In what follows let $B_1$ and $B_2$ be separates of $\Sigma_1$. 
By Lemma~\ref{lem_03}, for every separate $B$ of $\Sigma\backslash{Y}$ there exists at most one vertex of attachment $v$ such that $Y(B,v)$ consists of 
edges with different sign. Let by $v_1^{\pm}, v_2^{\pm}$ and $v_0^{\pm}$ denote these vertices for $B_1$, $B_2$ and $B_0$ respectively. We have the following cases:

\noindent
\underline{{\bf Case 1: $B_1,B_2\neq B_0$}}\\
The unbalanced separate $B_0$ may be contained in both $C(B_1,v_1)$ and $C(B_2,v_2)$ or in one of them. In the first case, 
$\Sigma.(B_1\cup{Y})|Y$ consists of edges with one common end-vertex while the edges of $Y(B_1,v_1)$ are half-edges in 
$\Sigma.(B_1\cup{Y})|Y$ (see Figure~\ref{fig_bridge_sep_1} where $v,w \neq v_1$). 
Therefore, irrespectively of the signs in the edges of $Y\dof Y(B_1,v_1)$  there exists a bond in $\Sigma.(B_1\cup{Y})|Y$ that contains
the half-edges $Y(B_1,v_1)$ which in turn implies that there exists a cocircuit $C_1^{*}\in{\mathcal{C}^{*}(M(\Sigma).(B_1\cup{Y})|{Y})}$ 
such that $Y(B_1,v_1)\subseteq{C_1^{*}}$. Similarly, we can find such a cocircuit $C_2^{*}$ for $B_2$.
\begin{figure}[hbtp] 
\begin{center}
\mbox{
\subfigure[signed graph $\Sigma$]
{
\psfrag{S1}{\footnotesize $\Sigma_1$}
\psfrag{S2}{\footnotesize $\Sigma_2$}
\psfrag{B1}{\footnotesize $B_1$}
\psfrag{B2}{\footnotesize $B_2$}
\psfrag{B0}{\footnotesize $B_0$}
\psfrag{v1}{\tiny $v_1$}
\psfrag{v2}{\tiny $v_2$}
\psfrag{Y}{\footnotesize $Y$}
\psfrag{Y1}{\tiny $Y(B_1,v_1)$}
\psfrag{C1}{\tiny $C(B_1,v_1)$}
\psfrag{C2}{\tiny $C(B_2,v_2)$}
\includegraphics*[scale=0.25]{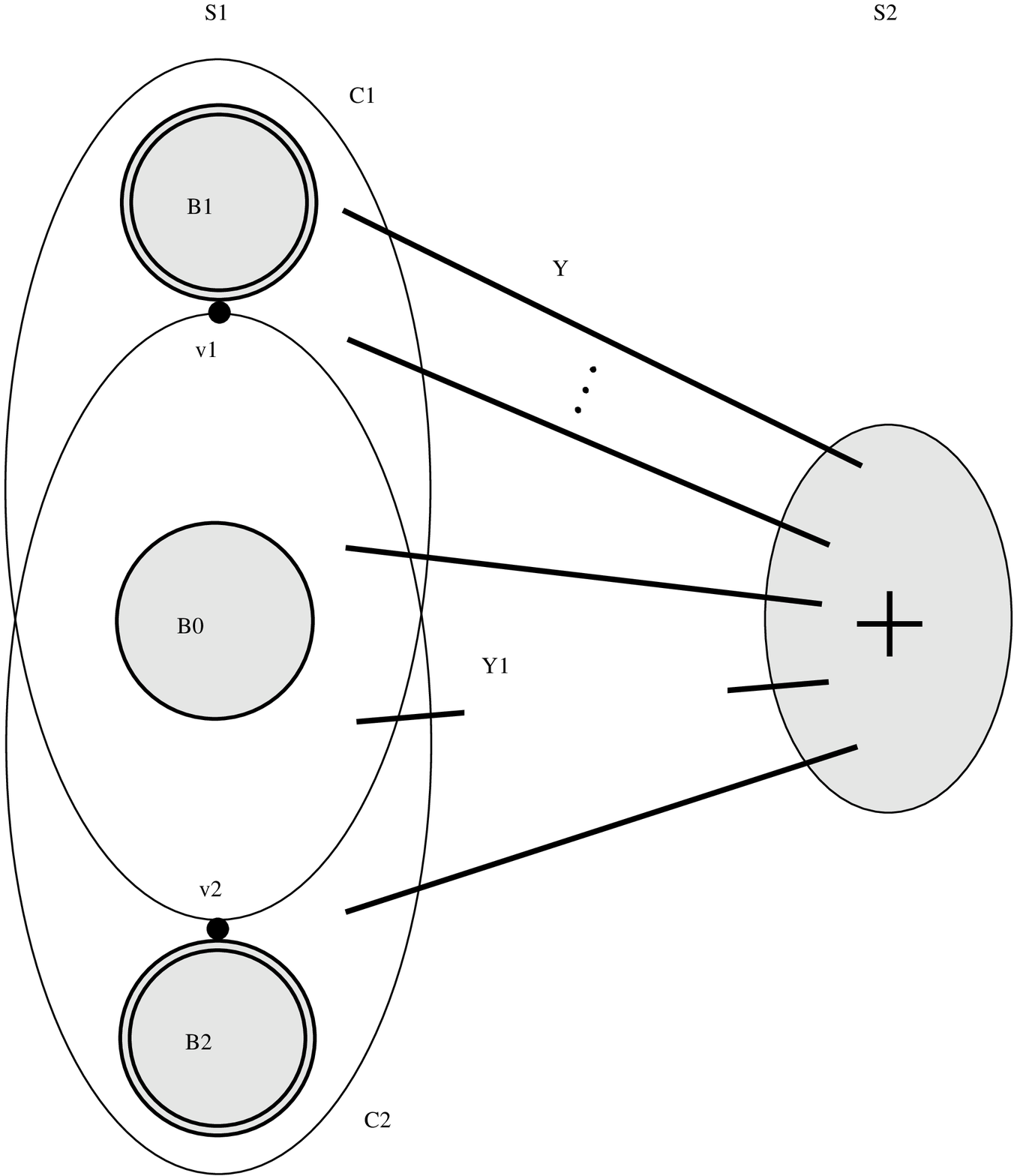}
}
\subfigure[$\Sigma.(B_1\cup{Y})|Y$]
{
\psfrag{Y1}{\tiny $Y(B_1,v_1)$}
\psfrag{Y-Y1}{\tiny $Y\dof Y(B_1,v_1) $}
\psfrag{v}{\footnotesize $v$}
\psfrag{w}{\footnotesize $w$}
\includegraphics*[scale=0.25]{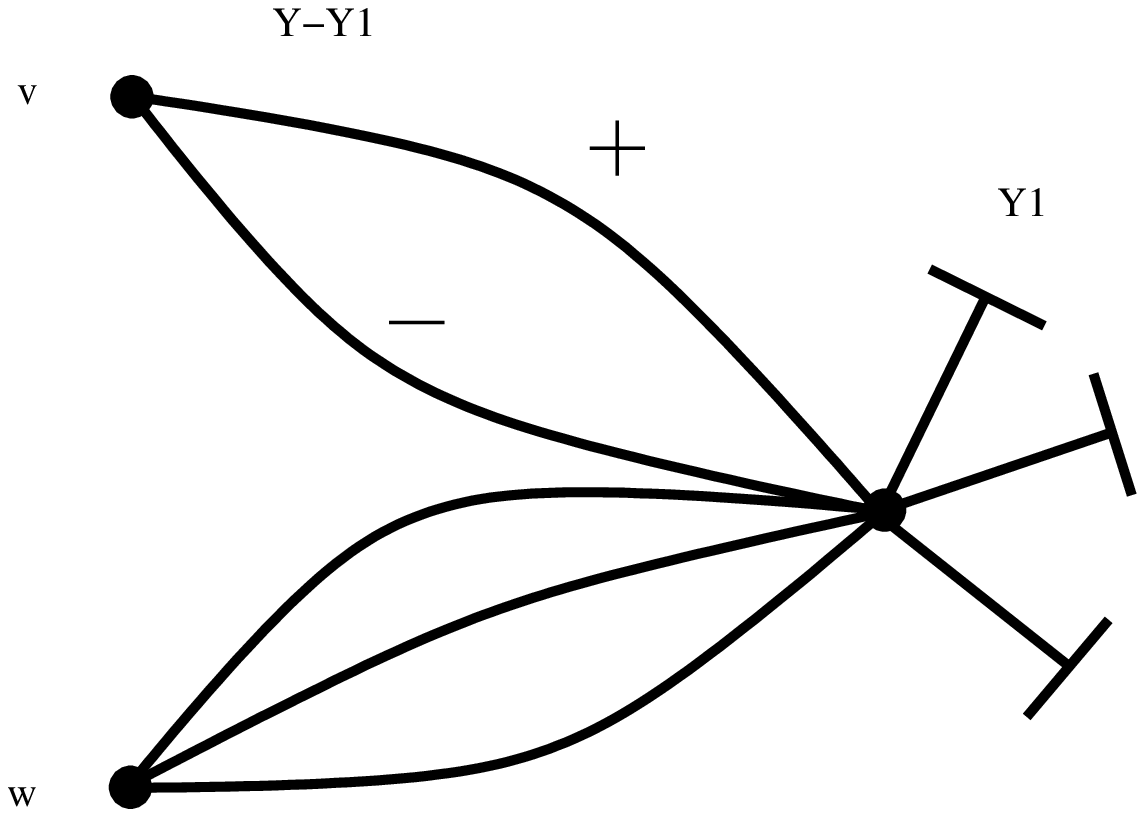}
}
}
\end{center} 
\caption{Case 1.}
\label{fig_bridge_sep_1}
\end{figure}

Consider now without loss of generality that $B_0$ is contained in $C(B_1,v_1)$ and not in $C(B_2,v_2)$. Then there must exist $v_0\in{V(B_0)}$ such that $B_1$ and $B_2$ are contained in $C(B_0,v_0)$. 
We have the following subcases. 

\noindent
\underline{\bf Case 1.a: $v_0\neq{v_0^{\pm}}$}\\
In this case either $C(B_1,v_1)$ or $C(B_2,v_2)$ is contained in $C(B_0,v_0)$; without loss of generality, consider the latter (Figure~\ref{fig_bridge_sep_2}(a)). 
Since $Y(B_2,v_2)\subseteq{Y(B_0,v_0)}$ the edges of $Y(B_2,v_2)$ have the same sign, thus, constitute a bond in  $\Sigma.(B_2\cup{Y})|Y$ (Figure~\ref{fig_bridge_sep_2}(b)), 
implying that $Y(B_2,v_2)\in{\mathcal{C}^{*}(M(\Sigma).(B_2\cup{Y})|{Y})}$. Given that $B_0$ is contained in  $C(B_1,v_1)$, the edges in $Y(B_1,v_1)$ are half-edges in 
$\Sigma.(B_1\cup{Y})|Y$ (Figure~\ref{fig_bridge_sep_2}(c)). Therefore, there exists a cocircuit $C_1^{*}\in{\mathcal{C}^{*}(M(\Sigma).(B_1\cup{Y})|{Y})}$ such that 
$Y(B_1,v_1)\subseteq{C_1^{*}}$.
\begin{figure}[hbtp] 
\begin{center}
\mbox{
\subfigure[signed graph $\Sigma$]
{
\psfrag{S1}{\footnotesize $\Sigma_1$}
\psfrag{S2}{\footnotesize $\Sigma_2$}
\psfrag{B1}{\footnotesize $B_1$}
\psfrag{B2}{\footnotesize $B_2$}
\psfrag{B0}{\footnotesize $B_0$}
\psfrag{v1}{\tiny $v_1$}
\psfrag{v2}{\tiny $v_2$}
\psfrag{v0}{\tiny $v_0$}
\psfrag{v0-}{\tiny $v_{0}^{\pm}$}
\psfrag{Y}{\footnotesize $Y$}
\psfrag{Y1}{\tiny $Y(B_1,v_1)$}
\psfrag{Y2}{\tiny $Y(B_2,v_2)$}
\psfrag{C1}{\tiny $C(B_1,v_1)$}
\psfrag{C2}{\tiny $C(B_2,v_2)$}
\psfrag{C0}{\tiny $C(B_0,v_0)$}
\includegraphics*[scale=0.25]{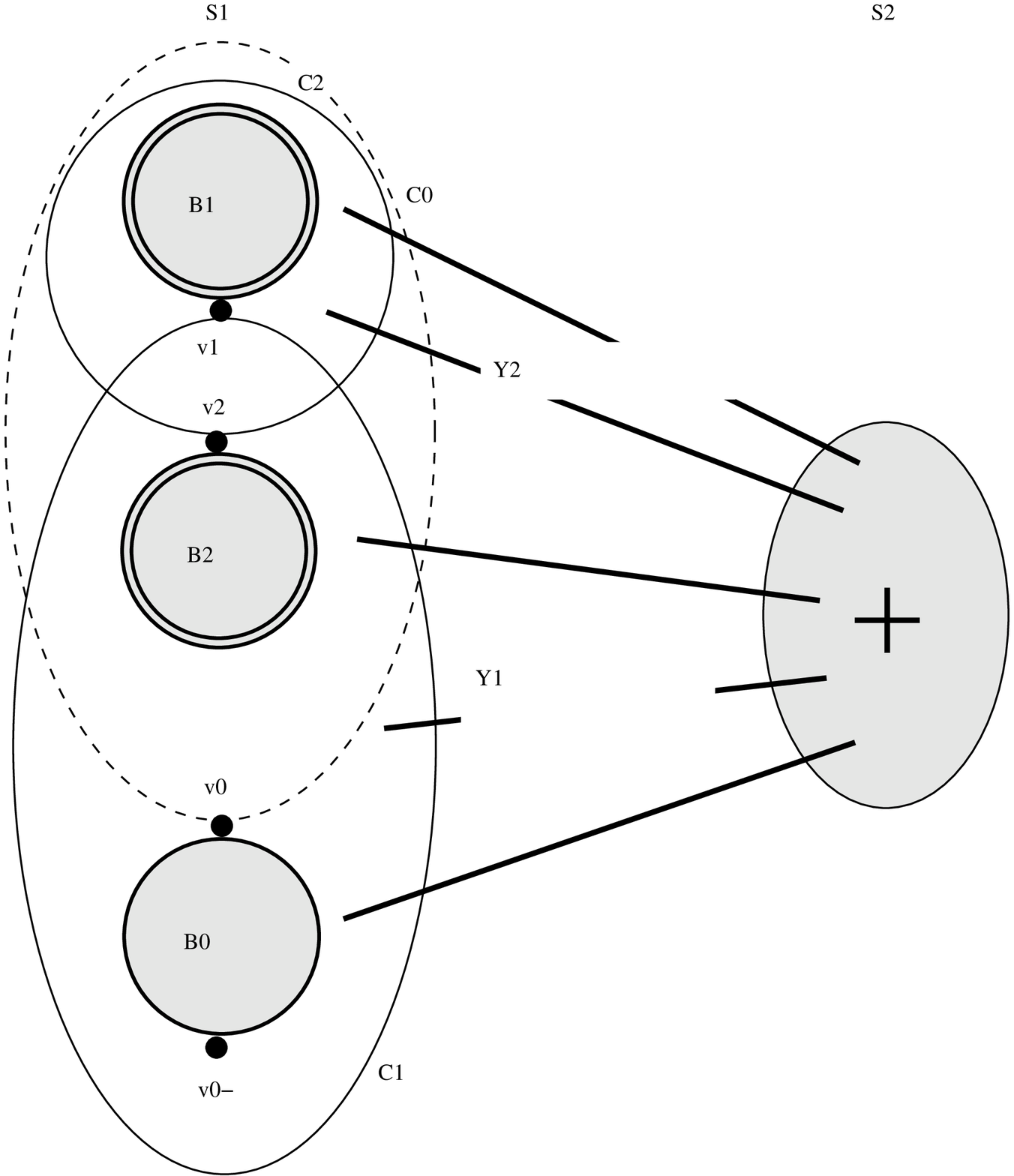}
}
\subfigure[$\Sigma.(B_1\cup{Y})|Y$]
{
\psfrag{Y1}{\tiny $Y(B_1,v_1)$}
\psfrag{Y-Y1}{\tiny $Y\dof Y(B_1,v_1) $}
\psfrag{v}{\footnotesize $v$}
\psfrag{w}{\footnotesize $w$}
\includegraphics*[scale=0.25]{bridge_sep_1_b.eps}
}
\subfigure[$\Sigma.(B_2\cup{Y})|Y$]
{
\psfrag{Y2}{\tiny $Y(B_2,v_2)$}
\psfrag{Y22}{\tiny $Y(B_2,w)$}
\psfrag{x}{\tiny $w \neq v_2$}
\includegraphics*[scale=0.25]{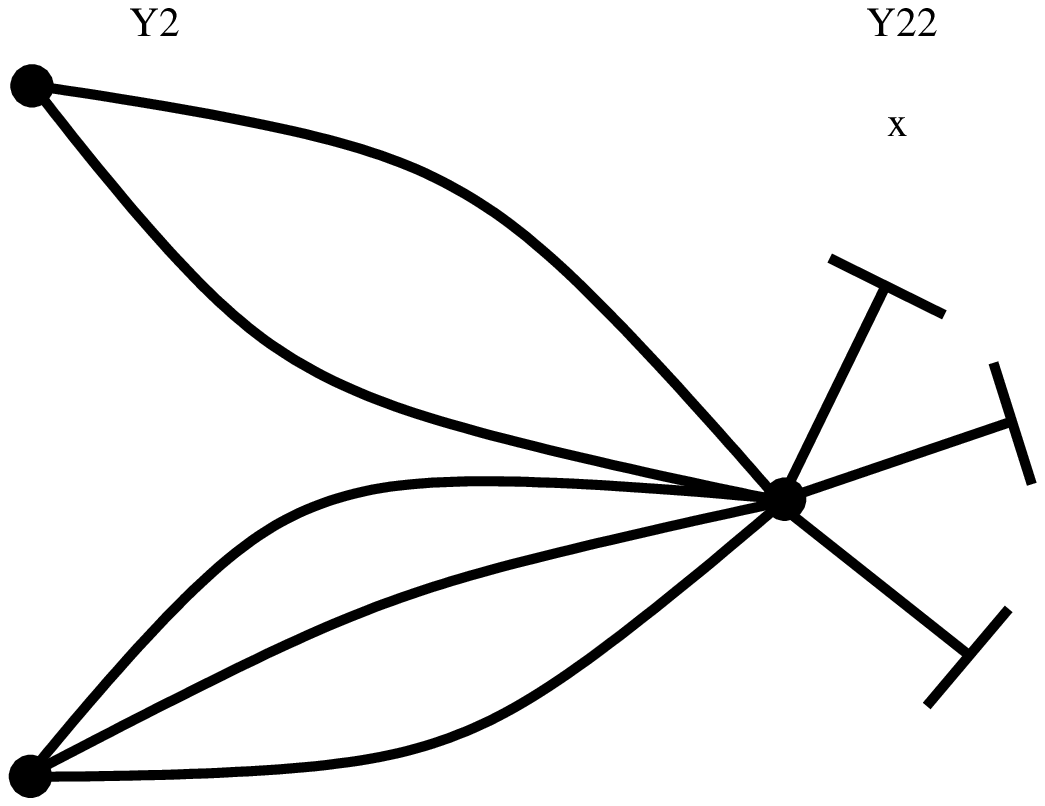}
}
}
\end{center} 
\caption{Case 1.a.}
\label{fig_bridge_sep_2}
\end{figure}

\noindent
\underline{\bf Case 1.b: $v_0 = v_0^{\pm}$}\\
Assume that $v_2\neq{v_2^{\pm}}$. Then $Y(B_2,v_2)$ consists of edges with equal signs, thus, 
$Y(B_2,v_2)\in{\mathcal{C}^{*}(M(\Sigma).(B_2\cup{Y})|{Y})}$. The edges in $Y(B_1,v_1)$ in $\Sigma.(B_1\cup{Y})|Y$ are half-edges with common end-vertex, 
therefore, $Y(B_1,v_1)\subseteq{C_1^{*}}\in{\mathcal{C}^{*}(M(\Sigma).(B_1\cup{Y})|Y)}$.

Now assume that $v_2 = v_2^{\pm}$ which means that $Y(B_2,v_2^{\pm})$ consists of edges with different sign (see Figure~\ref{fig_bridge_sep_3}). $B_{i}$, $i=1,2$ can have only one vertex in $\Sigma.(B_i\cup{Y})|{Y})$ that is incident with edges of $Y$ of different sign. Let this vertex be $v'_{1}$ for $B_{1}$ and $Y^{-}(B_{1}, v'_{1})$ be the set of edges of the same sign that can be incident only to $v'_{1}$ in $\Sigma.(B_1\cup{Y})|Y$. Without loss of generality assume them to be of negative sign. Thus, $Y(B_2,v_{2}^{\pm})=Y^{+}(B_{2},v_{2}^{\pm}) \cup Y^{-}(B_{1}, v'_{1})$ where $Y^{+}(B_{2},v_{2}^{\pm})$ is the subset of the positive edges of $Y(B_{2},v_{2}^{\pm})$. The edges in $Y(B_1,v_1)$ are half-edges in $\Sigma.(B_1\cup{Y})|Y$. Hence $Y(B_{1},v_{1}) \cup Y^{-}(B_{1}, v'_{1})={C_1^{*}} \in {\mathcal{C}^{*}(M(\Sigma).(B_1\cup{Y})|Y)}$. Let $v_2'$ be the vertex of attachment of $B_2$ such that $C(B_2,v_2')$ contains $B_0$; then, in $\Sigma.(B_2\cup{Y})|Y$, the edges in $Y(B_2,v_2')$ are 
half-edges. Moreover, the  set $Y(B_2,v_2')$ is non-empty since $\Sigma$ is $2$-connected. In  $\Sigma.(B_2\cup{Y})|Y$, from the sets of parallel edges of $Y$ only the set 
which has $v_2^{\pm}$ as a common end-vertex may consist of edges of different sign. Thus, $Y(B_2,v_2^{\pm}) \cup Y^{+}(B_{2},v_{2}^{\pm})={C_2^{*}} \in {\mathcal{C}^{*}(M(\Sigma).(B_2\cup{Y})|Y)}$.

\begin{figure}[hbtp] 
\begin{center}
\psfrag{S1}{\footnotesize $\Sigma_1$}
\psfrag{S2}{\footnotesize $\Sigma_2$}
\psfrag{B1}{\footnotesize $B_1$}
\psfrag{B2}{\footnotesize $B_2$}
\psfrag{B0}{\footnotesize $B_0$}
\psfrag{v1}{\tiny $v_1$}
\psfrag{v2-}{\tiny $v_2^{\pm}$}
\psfrag{v2'}{\tiny $v_2'$}
\psfrag{v0}{\tiny $v_0$}
\psfrag{v0-}{\tiny $v_{0}^{\pm}$}
\psfrag{Y}{\footnotesize $Y$}
\psfrag{Y1}{\tiny $Y(B_1,v_1)$}
\psfrag{Y2}{\tiny $Y(B_2,v_2^{\pm})$}
\psfrag{C1}{\tiny $C(B_1,v_1)$}
\psfrag{C2}{\tiny $C(B_2,v_2^{\pm})$}
\psfrag{C2'}{\tiny $C(B_2,v_2')$}
\psfrag{C0}{\tiny $C(B_0,v_0)$}
\includegraphics*[scale=0.28]{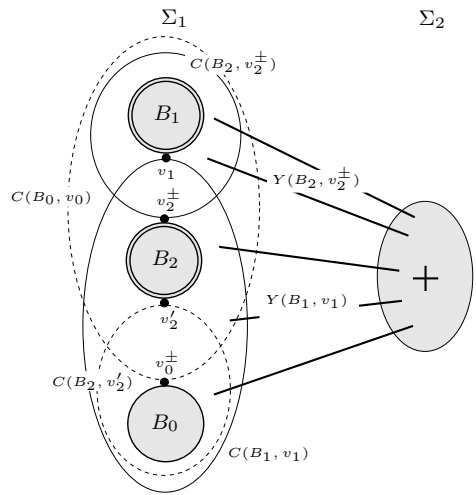}
\end{center} 
\caption{Case 1.b.}
\label{fig_bridge_sep_3}
\end{figure}

\noindent
\underline{{\bf Case 2: $B_2=B_0$}} \\
We shall consider the following cases:

\noindent
\underline{{\bf Case 2.a: $v_2\neq{v_2^{\pm}}$}}\\
In $\Sigma.(B_2\cup{Y})|Y$, the edges in $Y(B_2,v_2)$ are all those having $v_2$ as an end-vertex; therefore, 
$Y(B_2,v_2)\in{\mathcal{C}^{*}(M(\Sigma).(B_2\cup{Y})|Y)}$. In the signed graph $\Sigma.(B_1\cup{Y})|Y$, $Y(B_1,v_1)$ is the set of half-edges 
attached at a common vertex, while all the edges of $Y$ in $\Sigma.(B_1\cup{Y})|Y$ have that vertex as a common end-vertex. Therefore,
 $Y(B_1,v_1)\subseteq{C_1^{*}}\in{\mathcal{C}^{*}(M(\Sigma).(B_1\cup{Y})|Y)}$. \\
\underline{{\bf Case 2.b: $v_2=v_2^{\pm}$}}\\
Assume that $B_1$ has one balanced component $C(B_1,v_1^{\pm})$ such that the edges in $Y(B_1,v_1^{\pm})$ have different sign. Let $Y^{+}(B_i,v_i^{\pm})$ and  
$Y^{-}(B_i,v_i^{\pm})$ be the set of positive  and negative edges of  $Y(B_i,v_i^{\pm})$, respectively, for $i=1,2$ (see Figure~\ref{fig_bridge_sep_4}). 
Given that $B_2$ contains at least one
 negative face of $\Sigma$  with no edges of $Y$, the unique negative face defined by the edges of $Y(B_1,v_1^{\pm})$ has to be the outer face of $\Sigma$. This 
implies that all the edges of $Y$ with end-vertex either in $V(C(B_1,v))$ for $v\neq{v_1^{\pm},v_1}$ or in $V(C(B_1,v_1)\cap{V(C(B_2,v_2^{\pm}))}$ will have the 
same sign, say positive. This in turn implies that $Y^{-}(B_2,v_2^{\pm})=Y^{-}(B_1,v_1^{\pm})$. Examining the graphs $\Sigma.(B_i\cup{Y})|Y$ 
we have that $Y(B_1,v_1)\cup{Y^{-}(B_2,v_2^{\pm})}=Y(B_1,v_1)\cup{Y^{-}(B_1,v_1^{\pm})}\in{\mathcal{C}^{*}(M(\Sigma).(B_1\cup{Y})|Y)}$ and 
$Y^{+}(B_2,v_2^{\pm})\in{\mathcal{C}^{*}(M(\Sigma).(B_2\cup{Y})|Y)}$. Finally, if $Y(B_1,v)$ for every $v\neq{v_1}$ has edges of the same sign, the proof is the same as above. 
\begin{figure}[hbtp] 
\begin{center}
\psfrag{S1}{\footnotesize $\Sigma_1$}
\psfrag{S2}{\footnotesize $\Sigma_2$}
\psfrag{B1}{\footnotesize $B_1$}
\psfrag{B2}{\footnotesize $B_2$}
\psfrag{v1}{\tiny $v_1$}
\psfrag{v2-}{\tiny $v_2^{\pm}$}
\psfrag{v1-}{\tiny $v_{1}^{\pm}$}
\psfrag{Y1-}{\tiny $Y^{-}(B_1,v_1^{\pm})$}
\psfrag{Y1+}{\tiny $Y^{+}(B_1,v_1^{\pm})$}
\psfrag{C1}{\tiny $C(B_1,v_1)$}
\psfrag{C2}{\tiny $C(B_2,v_2^{\pm})$}
\psfrag{C1-}{\tiny $C(B_1,v_1^{\pm})$}
\includegraphics*[scale=0.31]{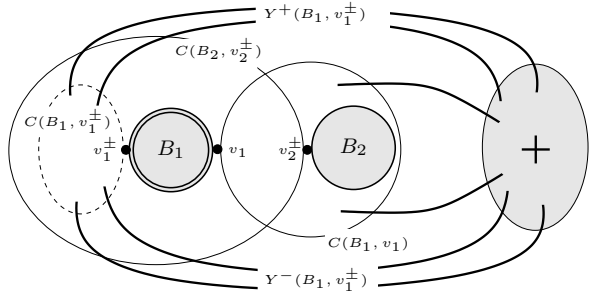}
\end{center} 
\caption{Case 2.b.}
\label{fig_bridge_sep_4}
\end{figure}

\end{proof}

\begin{corollary} \label{Yclasses}
If $Y$ is a bridge-separable cocircuit of $M(\Sigma)$ such that $\Sigma$ is 2-connected jointless and cylindrical signed graph then two possible classes $U^{-},U^{+}$ of all avoiding bridges of $Y$ contain the separates of the unbalanced and balanced connected component of $\Sigma \dof Y$ respectively.
\end{corollary}

\begin{corollary}
If $M(\Sigma)$ is a connected quaternary, non-binary matroid and $Y$ is a nongraphic cocircuit of $M(\Sigma)$ and an unbalancing bond in $\Sigma$ such that $\Sigma \dof J_{\Sigma}$ is cylindrical then $Y$ is bridge-separable.
\end{corollary}

\begin{theorem}
Let $\Sigma$ be a 2-connected jointless cylindrical signed graph such that $M(\Sigma)$ is internally 4-connected quaternary non-binary. If $Y$ is a nongraphic cocircuit and a double bond of $\Sigma$ then $Y$ is bridge-separable in $M(\Sigma)$. 
\end{theorem}
\begin{proof}
Consider a planar embedding of $\Sigma$ with two vertex disjoint negative faces where one of them is the outer. We denote this embedding also by $\Sigma$. By hypothesis $Y$ is a nongraphic cocircuit and a double bond in $\Sigma$. By claim~\ref{connectedcomponents} let $\Sigma_{1}, \Sigma_{2}$ be the unbalanced and the balanced connected component of $\Sigma \dof Y$ respectively. We shall denote by $B_{0}$ the unique unbalanced separate in $\Sigma_{1}$. Moreover, let $Y_{1},Y_{2}$ be the unbalancing and the balancing part of $Y$ respectively. By proposition~\ref{uniqueedge} $Y_2$ contains a unique edge $e$. We may perform switchings at vertices of $\Sigma$ and assume that all the edges of the balanced separates of $\Sigma \dof Y$ are positive. Consider any pair of bridges $B_1, B_2$ in either $M(\Sigma_1)$ or $M(\Sigma_2)$. To prove the theorem it suffices to show that
 there exist cocircuits $C_1^{*}\in{\mathcal{C}^{*}(M(\Sigma).(B_1\cup{Y})|{Y})}$ and $C_2^{*}\in{\mathcal{C}^{*}(M(\Sigma).(B_2\cup{Y})|{Y})}$ such that $C_1^{*}\cup{C_2^{*}}=Y$. The cases considered for $B_{1},B_{2}$ separates of $\Sigma_{1}$ are the same with those of the proof for an unbalancing bond. The only difference is that in all bonds which contain half-edges we include $Y_{2}$ which is a negative loop at the same vertex, where $\Sigma_{2}$ is contracted, with the half-edges in $\Sigma.(B_i\cup{Y})|{Y})$, $i=1,2$. Henceforth $B_{1}$ and $B_{2}$ are separates of $\Sigma_{2}$. Let $v_1\in{V(B_1)}$ and $v_2\in{V(B_2)}$ be the vertices of attachment such that $B_2$ is contained in ${C(B_1,v_1)}$ and  $B_1$ is contained in ${C(B_2,v_2)}$, respectively. We have that $V(\Sigma_i)\subseteq{V(C(B_1,v_1))\cup{V(C(B_2,v_2))}}$ which implies that  
\begin{equation}
Y(B_1,v_1)\cup{Y(B_2,v_2)}=Y_{1}
\end{equation}
\noindent
We distinguish two cases as regards the endvertices of the unique edge $e$ of $Y_{2}$. $e$ can have both endvertices to $B_{i}$, $i=1,2$ and $e$ has each endvertex to one of $B_i$. For the first case we assume without loss of generality that $e$ has both endvertices to ${B_2}$. \\
\noindent
\underline{{\bf Case 1: }}$e$ has both endvertices to ${B_2}$ \\
The edges of $Y(B_{1},v_{1})$ and $e$ are half-edges at $v_{1}$ in $\Sigma.(B_1\cup{Y})|{Y})$. Thus, $Y(B_{1},v_{1}) \cup e$ is a bond in $\Sigma.(B_1\cup{Y})|{Y})$ and a cocircuit ${C^{*}_{1}} \in {\mathcal{C}^{*}(M(\Sigma).(B_1\cup{Y})|Y)}$. Similarly, we can find such a cocircuit ${C^{*}_{2}}$ for $B_{2}$. \\
\noindent
\underline{{\bf Case 2: }}$e$ has each endvertex to one of ${B_1}, {B_2}$\\
Without loss of generality we will assume that $e$ has a common endvertex $v \in V(B_{1})$ with an edge of $Y_{1}$ by proposition~\ref{commonendvertex}. $Y(B_{2},v_{2})$ are half-edges at $v_{2}$ in $\Sigma.(B_1\cup{Y})|Y)$ which implies that $Y(B_{2},v_{2})=C^{*}_{2}\in {\mathcal{C}^{*}(M(\Sigma).(B_2\cup{Y})|Y)}$. In $\Sigma.(B_1\cup{Y})|Y)$ there is a bond which contains $Y(B_{1},v_{1})$, which are half-edges at $v_{1}$, $e$ and the half-edges incident at $v$. This bond is the cocircuit $C^{*}_{1}\in {\mathcal{C}^{*}(M(\Sigma).(B_1\cup{Y})|Y)}$ such that $C^{*}_{1} \cup C^{*}_{2}= Y$.
\end{proof}

Given a signed-graphic matroid $M(\Sigma)$, $Y \in \mathcal{C^{*}}(M(\Sigma))$, $B$ a bridge of $Y$ in $M(\Sigma)$ and $C^{*}_B$ a bond in $\Sigma.(B \cup Y)| Y$. We will say that $C^{*}_{B}$ \emph{determines vertex} $v \in V(B)$ if $v$ is a common endvertex of all edges in $C^{*}_B$. In the case that the edges of $C^{*}_B$ have no common endvertex in $V(B)$ then we will say that $C^{*}_B$ determines the vertex of attachment of $B$ whose corresponding component contains the core.

\begin{theorem}
Let $M$ be a connected quaternary nonbinary signed-graphic matroid whose signed-graphic representations are all cylindrical. If $Y$ is a u-cocircuit of $M$ such that no two bridges of $Y$ in $M$ overlap then there exists a 2-connected cylindrical signed graph $\Sigma$ such that $Y$ is the star of a vertex and $M=M(\Sigma)$.
\end{theorem}
\begin{proof}
We choose the cylindrical signed graph $\Sigma$ representing $M$ so that the balanced component, denoted by $\Sigma_{2}$, of $\Sigma \dof Y$ has the least number of edges. Assume that $|E(\Sigma_{2})| >0$. The unbalanced component, denoted by $\Sigma_{1}$, contains the unbalanced block which implies that $|E(\Sigma_{2})| >0$. The cylindrical signed graph $\Sigma$ has a planar embedding with a negative outer face. We will denote the planar embedding of $\Sigma$ where $\Sigma_{2}$ has the least possible number of edges and the outer face is negative, also by $\Sigma$. Fix a balanced bridge of $Y$ in $M(\Sigma_{2})$, denoted by $B_{1}$ and choose the unbalanced bridge of $Y$ in $M(\Sigma_{1})$ denoted by $B_{0}$. By hypothesis, $B_{0}, B_{1}$ are avoiding thus there exist $C_{B_{0}}^{*} \in \mathcal{C}^{*}(M(\Sigma.(B_{0} \cup {Y}) | Y))$ and $C_{B_{1}}^{*} \in \mathcal{C}^{*}(M(\Sigma.(B_{1} \cup {Y}) | Y))$ such that 
\begin{equation} \label{eq:1} 
C_{B_{0}}^{*} \cup C_{B_{1}}^{*} =Y
\end{equation}
We will call the cocircuits whose union is $Y$ avoiding. The pair of avoiding cocircuits for $B_{0}, B_{1}$ is picked arbitrarily. 
Any cocircuit in $\mathcal{C}^{*}(M(\Sigma.(B_{0} \cup {Y}) | Y))$ corresponds to a set of parallel edges of the same sign incident to a vertex of $B_{0}$ in $\Sigma. (B_{0} \cup Y)|Y$. By definition $C_{B_{0}}^{*}$ determines a vertex, say $v_{0} \in V(B_{0})$, in $\Sigma$. Then $C_{B_{0}}^{*}=Y^{*}(B_{0},v_{0})$ where $Y^{*}(B_{0},v_{0})$ is either $Y^{+}(B_{0},v_{0})$ which denotes the set of positive edges of $Y$ with endvertices in $C(B_{0},v_{0})$, or $Y^{-}(B_{0},v_{0})$ which denotes the set of negative edges of $Y$ with endvertices in $C(B_{0},v_{0})$. Any cocircuit in $\mathcal{C}^{*}(M(\Sigma.(B_{1} \cup {Y}) | Y))$ corresponds to a set of half-edges incident to a vertex of $B_{1}$ in $\Sigma. (B_{1} \cup Y)|Y$. By definition $C_{B_{1}}^{*}$ determines a vertex, say $v_{1} \in V(B_{1})$, in $\Sigma$. Then $C_{B_{1}}^{*}=Y(B_{1},v_{1})$. By \eqref{eq:1} it follows that
\begin{equation} \label{eq:2}
Y^{*}(B_{0},v_{0}) \cup Y(B_{1},v_{1})=Y
\end{equation}
If there is no balanced bridge in $\Sigma_{1}$ then there is a 2-separation $\{T(v_{1}), E(\Sigma) \dof T(v_{1})\}$ with $T(v_{1})=C(B_{1},v_{1}) \cup Y(B_{1},v_{1}) \cup F(B_{0},v_{0})$ since by \eqref{eq:2} there is no edge of $Y$ with endvertices in $F(B_{0},v_{0})$ and $F(B_{1},v_{1})$. 

\begin{figure}[hbt]
\begin{center}
\psfrag{S1}{\footnotesize $\Sigma_1$}
\psfrag{S2}{\footnotesize $\Sigma_2$}
\psfrag{B0}{\footnotesize $B_0$}
\psfrag{B1}{\footnotesize $B_1$}
\psfrag{v0}{\tiny $v_0$}
\psfrag{v1}{\tiny $v_1$}
\psfrag{Y1}{\tiny $Y(B_1,v_1)$}
\psfrag{Y}{\footnotesize $Y$}
\psfrag{Y-Y1}{\tiny $Y\dof Y(B_1,v_1) $}
\psfrag{C1}{\tiny $C(B_1,v_1)$}
\psfrag{F0}{\tiny $F(B_0,v_0)$}
\psfrag{F1}{\tiny $F(B_1,v_1)$}
\includegraphics*[scale=0.31]{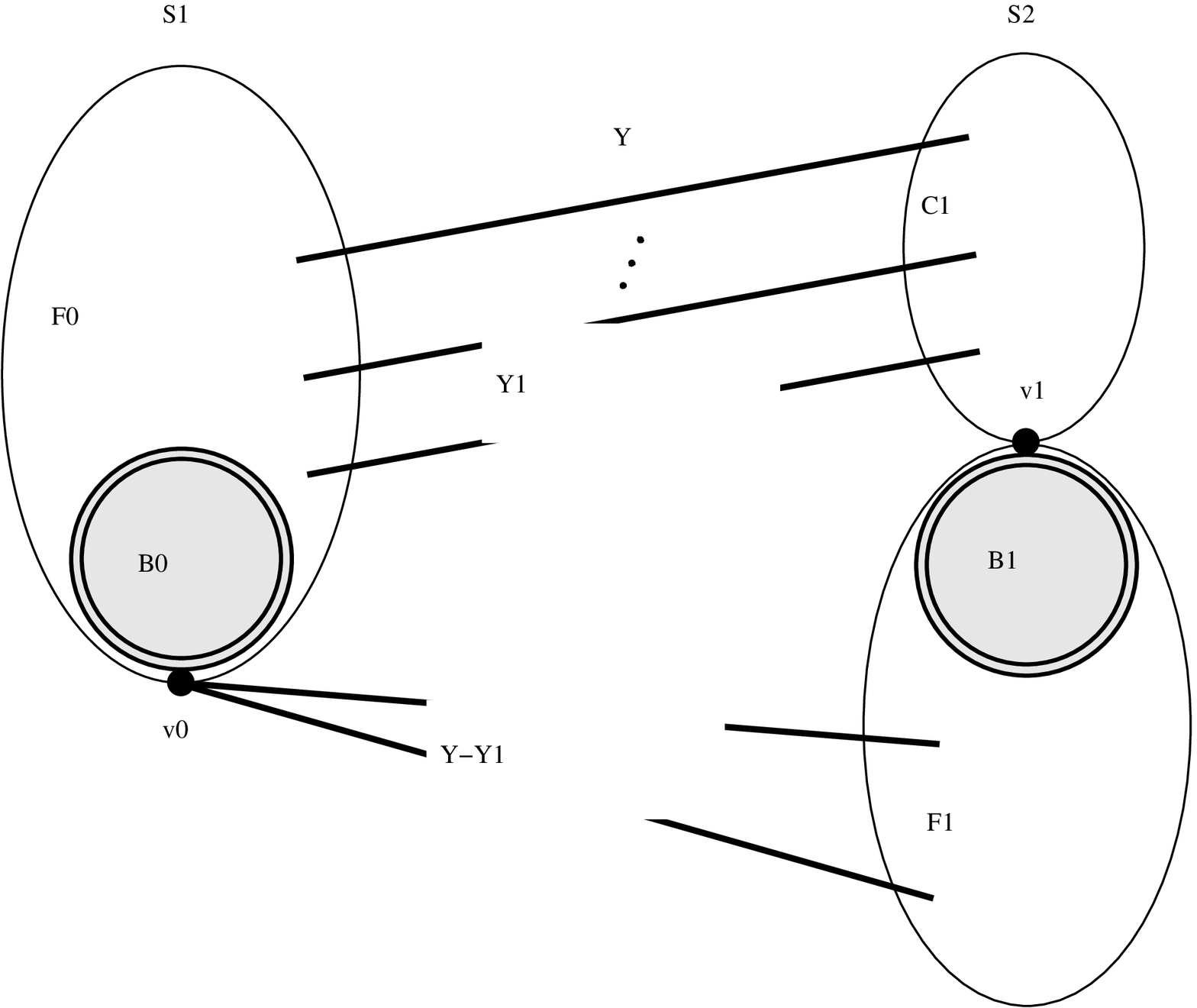}
\end{center} 
\caption{no $Bi$.}
\label{fig_star_th_1}
\end{figure}

By twisting about $v_{0},v_{1}$ we create a signed graph with balanced component $H=C(B_{1},v_{1})$ which has fewer edges than $\Sigma_{2}$, contradiction (see Figure~\ref{fig_star_th_1}). Assume that there is at least one balanced bridge $B_{i}$, $i=2, \ldots, k$ in $\Sigma_{1}$ such that $v_{0} \in V(B_{i})$. We define $F(B_{i}, v_{i})= \Sigma_{1} \dof E(C(B_{i},v_{i}))$ where $v_{i}$ is a vertex of attachment of $B_{i}$. By hypothesis $B_{i}, B_{1}$ are avoiding thus there are $C_{B_{i}}^{*} \in \mathcal{C}^{*}(M(\Sigma.(B_{i} \cup {Y}) | Y))$ and $C_{B_{1}}^{*} \in \mathcal{C}^{*}(M(\Sigma.(B_{1} \cup {Y}) | Y))$ such that 
\begin{equation} \label{eq:3}
C_{B_{i}}^{*} \cup C_{B_{1}}^{*}=Y
\end{equation}
Assume that $C_{B_{i}}^{*},C_{B_{1}}^{*}$ determine the vertices $w \in V(B_{i})$ and $v \in V(B_{1})$ in $\Sigma$, respectively. From all possible pairs of avoiding cocircuits we choose $C_{B_{i}}^{*},C_{B_{1}}^{*}$ for which $C_{B_{i}}^{*}$ determines $v_{0}$ in $\Sigma$. In case there is no such pair, we choose arbitrarily. Then we distinguish the following three cases:

\noindent
\underline{{\bf Case 1: $w \neq v_{0}$ for some $B_{i}$}} \\
Assume that there is at least one $B_{i}$, $i=2, \ldots, k$ for which there is no cocircuit $C_{B_{i}}^{*}$ which determines $v_{0}$ in $\Sigma$ and $C_{B_{i}}^{*} \cup C_{B_{1}}^{*}=Y$ for every $C_{B_{1}}^{*} \in \mathcal{C}^{*}(M(\Sigma.(B_{1} \cup {Y}) | Y))$. By assumption $C_{B_{i}}^{*}$ determines $w \neq v_{0}$ in $\Sigma$, then by definition each $Y(B_{i},v)$ consists of edges of the same sign for every $v$ vertex of attachment of $B_{i}$. There is always a cocircuit in $\mathcal{C}^{*}(M(\Sigma.(B_{i} \cup {Y}) | Y))$ which determines $v_{0}$ in $\Sigma$, say $C$. Then by hypothesis $C \cup C_{B_{1}}^{*} \subset Y$ for every $C_{B_{1}}^{*}$. In $\Sigma.(B_{i} \cup Y) | Y$, $C$ is a set of half-edges and in $\Sigma$ $C$ is the set of edges of $Y$ which have an endvertex in $C(B_{i},v_{0})$. By assumption $C_{B_{1}}^{*}$ determines $v \in V(B_{1})$ in $\Sigma$, then $C_{B_{1}}^{*}=Y(B_{1},v)$. Therefore $ Y^{*}(B_{i},v_{0}) \cup Y(B_{1},u) \subset Y$ where $u$ is any vertex of attachment of $B_{1}$. By the above relation (with $B_{i}, B_{1}, v_{0}$ fixed) there exists an edge of $Y$ with endvertices in $F(B_{i},v_{0})$ and $F(B_{1},u)$ for every $u$ (see Figure~\ref{fig_star_case1}). 

\begin{figure}[hbtp]
\begin{center}
\psfrag{S1}{\footnotesize $\Sigma_1$}
\psfrag{S2}{\footnotesize $\Sigma_2$}
\psfrag{B0}{\footnotesize $B_0$}
\psfrag{B1}{\footnotesize $B_1$}
\psfrag{Bi}{\footnotesize $B_i$}
\psfrag{v0}{\tiny $v_0$}
\psfrag{v1}{\tiny $v_1$}
\psfrag{w}{\tiny $w$}
\psfrag{v}{\tiny $v$}
\psfrag{u}{\tiny $u$}
\psfrag{Y}{\footnotesize $Y$}
\psfrag{C1}{\tiny $C(B_1,v_1)$}
\psfrag{F0}{\tiny $F(B_0,v_0)$}
\psfrag{Fi}{\tiny $F(B_i,v_0)$}
\includegraphics*[scale=0.31]{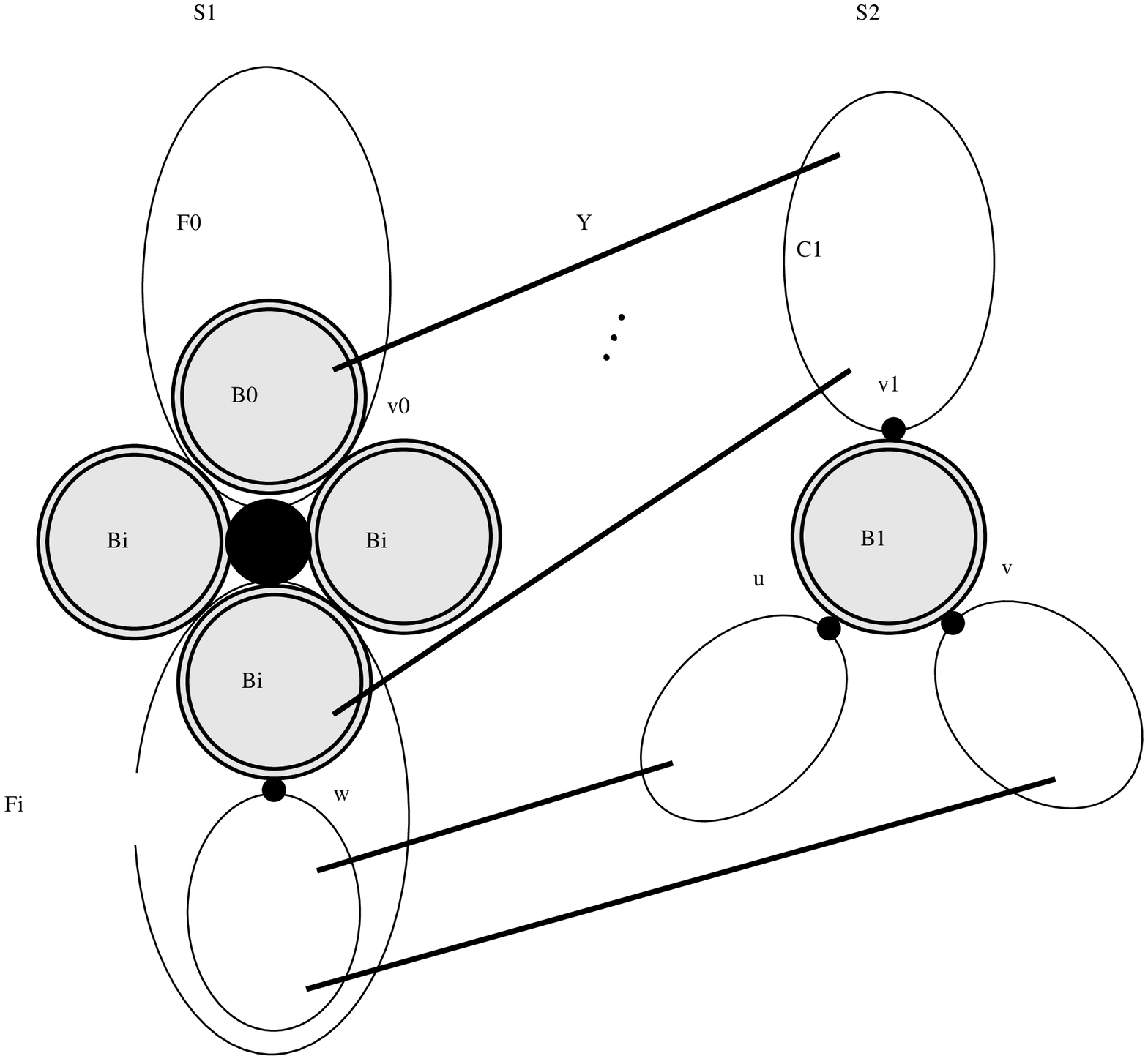}
\end{center} 
\caption{Case 1.}
\label{fig_star_case1}
\end{figure}

By hypothesis $C_{B_{i}}^{*}$ determines $w \in V(B_{i})$ in $\Sigma$ which implies that $C_{B_{i}}^{*}$ is a bond of parallel edges of the same sign incident to $w$ in $\Sigma.(B_{i} \cup {Y}) | Y$. Therefore, the edges of $C_{B_{i}}^{*}$ in $\Sigma$ are edges of $Y$ of the same sign with an endvertex in $C(B_{i},w)$ implying that $C_{B_{i}}^{*}=Y(B_{i},w)$. Then \eqref{eq:1} and \eqref{eq:3} become: 
\begin{equation*} \tag{\ref{eq:2}}
Y^{*}(B_{0},v_{0}) \cup Y(B_{1},v_{1})=Y
\end{equation*}
\begin{equation} \label{eq:4}
Y(B_{i},w) \cup Y(B_{1},v)=Y
\end{equation}

\noindent
Since $\Sigma$ is 2-connected there is an edge of $Y$ with an endvertex in $F(B_{0},v_{0})$. Then this edge has neither an endvertex in $F(B_{1},v_{1})$ by \eqref{eq:2} nor an endvertex in $F(B_{1},v)$ by \eqref{eq:4} since $F(B_{0},v_{0}) \subseteq F(B_{i},w)$. Thus, the edge has no endvertex in $\Sigma_{2}$ which is a contradiction since $Y$ is an unbalancing bond. Therefore, $v=v_{1}$ in \eqref{eq:4} for every $B_{i}$, $i=2, \ldots, k$ whose $C_{B_{i}}^{*}$ determines a vertex $w \neq v_{0}$ in $\Sigma$. By \eqref{eq:4} there is no edge of $Y$ with endvertices in $F(B_{i},w)$ and $F(B_{1},v_{1})$. Thus, every edge of $Y$ with endvertices in $F(B_{i},v_{0})$ and $F(B_{1},u)$ with $u \neq v_{1}$ has an endvertex in $C(B_{i},w)$. Assume that $C(B_{i},w)$ is a vertex. Then there is a 2-separation $\{T(v_{1}), E(\Sigma) \dof T(v_{1})\}$ with $T(v_{1})=C(B_{i},w) \cup Y(B_{i},w) \cup F(B_{1},v_{1})$.
By twisting about $w,v_{1}$ we form a new signed graph with balanced component $H=C(B_{1},v_{1})$ which has fewer edges than $\Sigma_{2}$, contradiction. Otherwise, $C(B_{i},w)$ contains at least one balanced bridge, say $B$. If there is only one edge of $Y$ with endvertex in $C(B_{i},w)$, say $x$, then there is only one vertex of attachment $u \neq v_{1}$ in $B_{1}$. Then there is a 2-separation $\{T(v_{1}), E(\Sigma) \dof T(v_{1})\}$ with $T(v_{1})=F(B_{1},v_{1}) \cup Y(B_{1},v_{1}) \cup C(B,x)$ where $C(B,x)$ is a vertex and $x$ is the vertex of attachment of $B$. By twisting about $x,v_{1}$ we form a new balanced component $H=C(B_{1},v_{1})$ which has less edges than $\Sigma_{2}$, contradiction. Otherwise, there are at least two edges of $Y$ with endvertices in $C(B_{i},w)$. These edges belong to two different cocircuits in $\mathcal{C}^{*}(M(\Sigma.(B \cup {Y}) | Y))$. Choosing either of the two cocircuits the remaining edges of $Y$ are partitioned into at least two cocircuits in $\mathcal{C}^{*}(M(\Sigma.(B_{1} \cup {Y}) | Y))$. Then there is no pair of avoiding cocircuits for $B$ and $B_{1}$ implying that $B,B_{1}$ are nonavoiding, contradiction.

\noindent
\underline{\bf{Case 2: $w=v_{0}$ and $v=v_{1}$ for all $B_{i}$}}

(i) $Y(B_{0},v_{0})$ consists of edges of the same sign i.e., $Y(B_{0},v_{0})=Y^{+}(B_{0},v_{0})$.

Then there is no balanced component $C(B_{i},v')$ with $v'$ vertex of attachment of $B_{i}$ whose associated $Y(B_{i},v')$ contains edges of $Y$ of different sign. By hypothesis $C_{B_{i}}^{*}$ determines $v_{0}$ for every $B_{i}$, $i=2, \ldots, k$ and $C_{B_{1}}^{*}$ determines $v=v_{1}$ for every $C_{B_{1}}^{*}$ so that \eqref{eq:3} holds. Then $C_{B_{i}}^{*}$ and $C_{B_{1}}^{*}$ correspond to bonds in $\Sigma.(B_{i} \cup {Y}) | Y$ and $\Sigma.(B_{1} \cup {Y}) | Y$ respectively which consist only of half-edges. Therefore, $C_{B_{i}}^{*}$ is the set of edges of $Y$ in $\Sigma$ with an endvertex in $C(B_{i},v_{0})$ i.e., $C_{B_{i}}^{*}=Y(B_{i},v_{0})$ and $C_{B_{1}}^{*}$ is the set of edges of $Y$ with an endvertex in $C(B_{1},v_{1})$ and $C_{B_{1}}^{*}=Y(B_{1},v_{1})$. Then \eqref{eq:1} and \eqref{eq:3} become:

\begin{equation*} \tag{\ref{eq:2}}
Y^{*}(B_{0},v_{0}) \cup Y(B_{1},v_{1})=Y
\end{equation*}
\begin{equation} \label{eq:5}
Y^{*}(B_{i},v_{0}) \cup Y(B_{1},v_{1})=Y
\end{equation}

\begin{figure}[hbt]
\begin{center}
\psfrag{S1}{\footnotesize $\Sigma_1$}
\psfrag{S2}{\footnotesize $\Sigma_2$}
\psfrag{B0}{\footnotesize $B_0$}
\psfrag{B1}{\footnotesize $B_1$}
\psfrag{Bi}{\footnotesize $B_i$}
\psfrag{v0}{\tiny $v_0$}
\psfrag{v1}{\tiny $v_1$}
\psfrag{Y}{\footnotesize $Y$}
\psfrag{C1}{\tiny $C(B_1,v_1)$}
\psfrag{F0}{\tiny $F(B_0,v_0)$}
\psfrag{Fi}{\tiny $F(B_i,v_0)$}
\psfrag{F1}{\tiny $F(B_1,v_1)$}
\includegraphics*[scale=0.31]{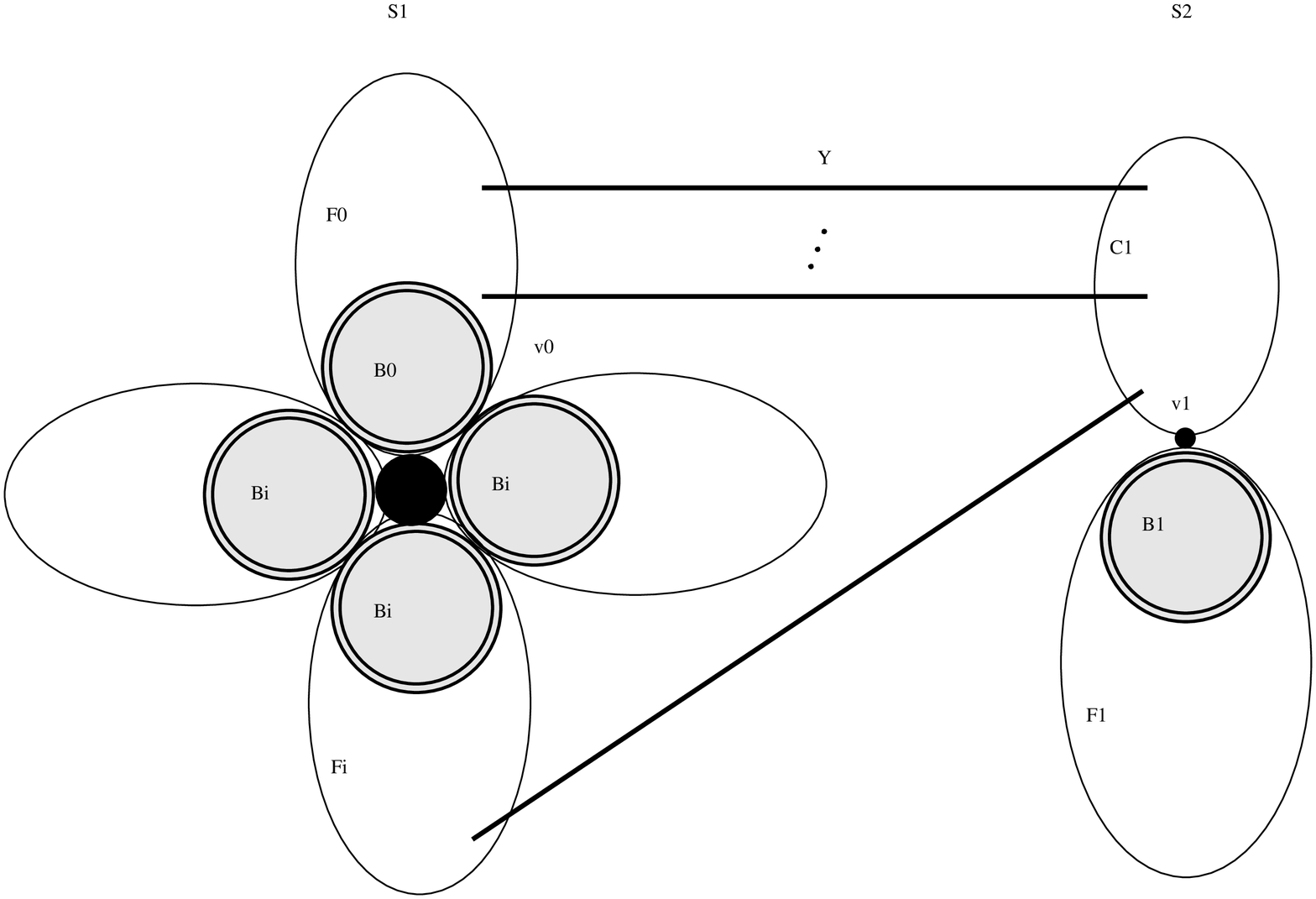}
\end{center} 
\caption{Case 2.}
\label{fig_star_case2}
\end{figure}

Since $\Sigma$ is 2-connected there exists an edge of $Y$ with an endvertex in $F(B_{1},v_{1})$ which does not have an endvertex in $F(B_{0},v_{0})$ by \eqref{eq:2} nor in $F(B_{i},v_{0})$ for every $B_{i}$ by \eqref{eq:5}. Then $Y$ has no endvertex in $\Sigma_{1}$ which is a contradiction since $Y$ is an unbalancing bond (see Figure~\ref{fig_star_case2}).

(ii) $Y(B_{0},v_{0})$ consists of edges of different sign i.e., $Y(B_{0},v_{0})=Y^{+}(B_{0},v_{0}) \cup Y^{-}(B_{0},v_{0})$.

If there is no balanced component $C(B_{i},v')$ with $v'$ vertex of attachment of $B_{i}$ whose associated $Y(B_{i},v')$ has edges of $Y$ of different sign then the case follows case 2(i). We shall consider the case where there is a $C(B_{i},v')$ whose associated $Y(B_{i},v')$ contains edges of $Y$ of different sign. By hypothesis $C_{B_{i}}^{*}$ determines $v_{0}$ in $\Sigma$ and corresponds to a bond in $\Sigma.(B_{i} \cup Y) |Y$ consisting of half-edges and the positive edges of $Y$ with an endvertex in $C(B_{i},v')$. We consider the case where $C_{B_{0}}^{*}=Y^{+}(B_{0},v_{0})$ and $C_{B_{i}}^{*}=Y(B_{i},v_{0}) \cup Y^{+}(B_{i},v')$ and all the others follow similarly. Then \eqref{eq:1} and \eqref{eq:3} become:
\begin{equation} \label{eq:6}
Y^{+}(B_{0},v_{0}) \cup Y(B_{1},v_{1})=Y
\end{equation}
\begin{equation} \label{eq:7}
Y(B_{i},v_{0}) \cup Y^{+}(B_{i},v') \cup Y(B_{1},v_{1})=Y
\end{equation}

By assumption there is a negative edge $e$ of $Y$ and a positive edge $h$ of $Y$ with an endvertex in $C(B_{i},v')$. By \eqref{eq:6} $e$ has its other endvertex in $C(B_{1},v_{1})$. Since $\Sigma$ is 2-connected by \eqref{eq:6} there is a positive edge $d$ of $Y$ with endvertices in $F(B_{1},v_{1})$ and $C(B_{0},v_{0})$ and another edge of $Y$ with endvertices in $F(B_{0},v_{0})$ and $C(B_{1},v_{1})$.

\begin{figure}[hbt]
\begin{center}
\psfrag{S1}{\footnotesize $\Sigma_1$}
\psfrag{S2}{\footnotesize $\Sigma_2$}
\psfrag{B0}{\footnotesize $B_0$}
\psfrag{B1}{\footnotesize $B_1$}
\psfrag{Bi}{\footnotesize $B_i$}
\psfrag{v0}{\tiny $v_0$}
\psfrag{v1}{\tiny $v_1$}
\psfrag{Y}{\footnotesize $Y$}
\psfrag{C1}{\tiny $C(B_1,v_1)$}
\psfrag{F0}{\tiny $F(B_0,v_0)$}
\psfrag{F1}{\tiny $F(B_1,v_1)$}
\psfrag{e}{\footnotesize $e$}
\psfrag{h}{\footnotesize $h$}
\psfrag{d}{\footnotesize $d$}
\includegraphics*[scale=0.31]{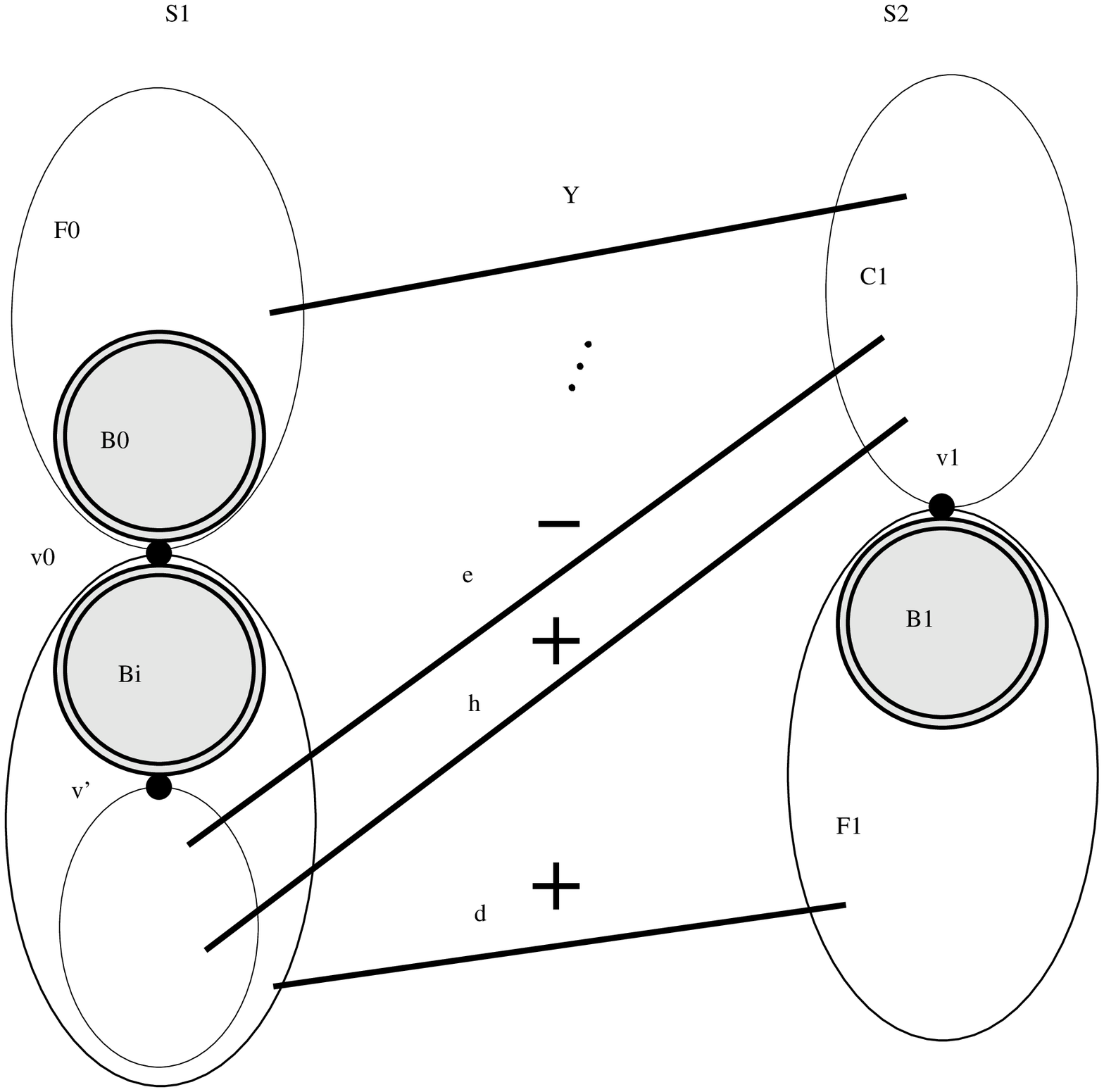}
\end{center} 
\caption{Case 2a.}
\label{fig_star_case2a}
\end{figure}

If $h$ has its other endvertex in $C(B_{1},v_{1})$, $e$ and $h$ are edges of a negative cycle because of planarity. The negative cycle implies the existence of a negative face $F$. $F$ is other than the negative face which is contained in the unbalanced bridge $B_{0}$ and other than the outer face. Hence $F$ constitutes a third negative face which is a contradiction since $\Sigma$ is cylindrical signed graph (see Figure~\ref{fig_star_case2a}). Otherwise, $h$ has its other endvertex in $F(B_{1},v_{1})$. If $h$ is distinct from $d$ then we reach a contradiction since $F$ constitutes a third negative face.

\begin{figure}[hbt]
\begin{center}
\psfrag{S1}{\footnotesize $\Sigma_1$}
\psfrag{S2}{\footnotesize $\Sigma_2$}
\psfrag{B0}{\footnotesize $B_0$}
\psfrag{B1}{\footnotesize $B_1$}
\psfrag{Bi}{\footnotesize $B_i$}
\psfrag{v0}{\tiny $v_0$}
\psfrag{v1}{\tiny $v_1$}
\psfrag{v'}{\tiny $v'$}
\psfrag{Y}{\footnotesize $Y$}
\psfrag{C1}{\tiny $C(B_1,v_1)$}
\psfrag{F0}{\tiny $F(B_0,v_0)$}
\psfrag{F1}{\tiny $F(B_1,v_1)$}
\psfrag{e}{\footnotesize $e$}
\psfrag{h}{\footnotesize $h$}
\includegraphics*[scale=0.31]{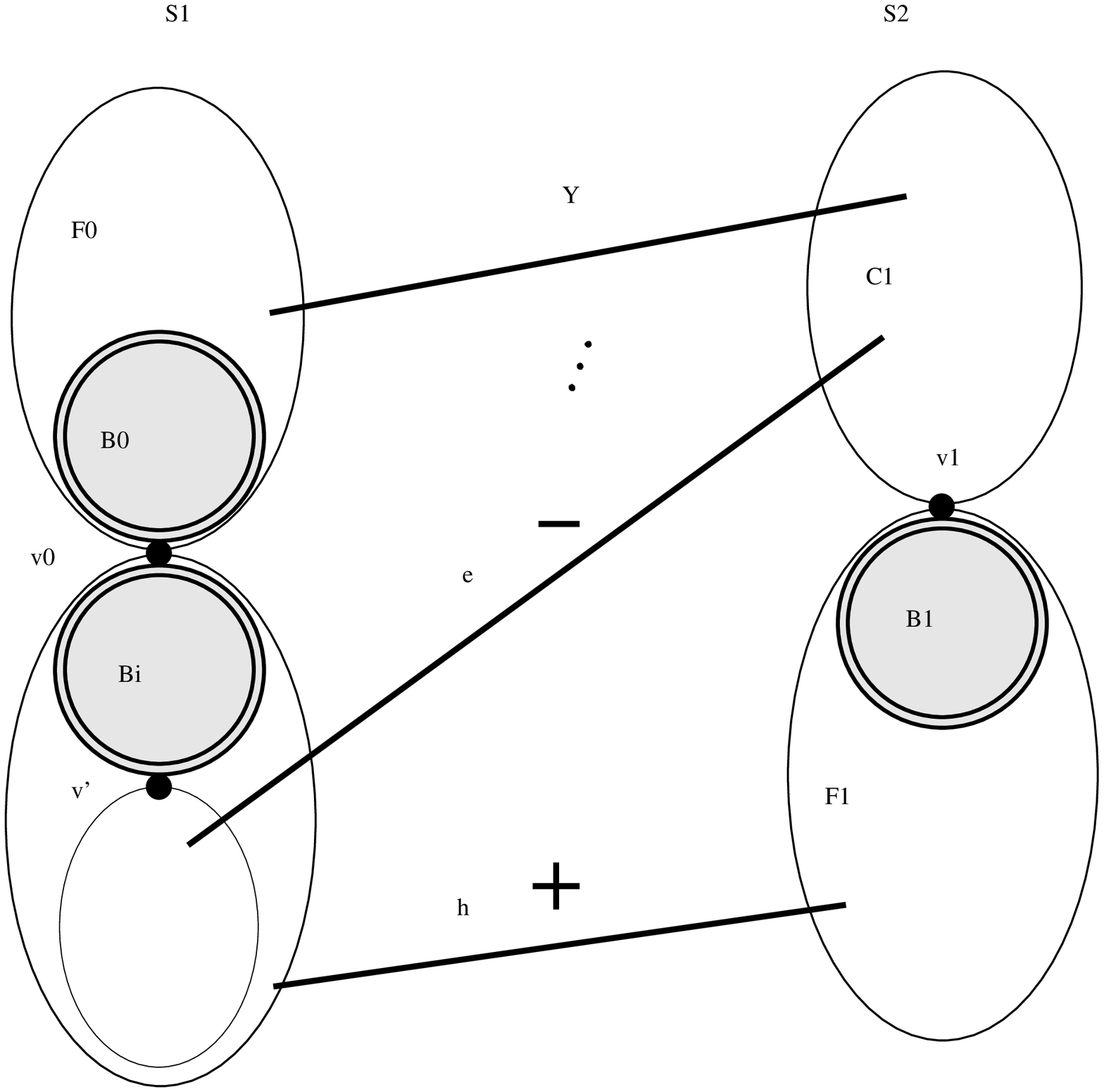}
\end{center} 
\caption{Case 2b.}
\label{fig_star_case2b}
\end{figure}

Assume that $h$ coincides with $d$ (see Figure~\ref{fig_star_case2b}). If $C(B_{i},v')$ is a vertex then there is a 2-separation $\{T(v_{1}), E(\Sigma) \dof T(v_{1})\}$ with $T(v_{1})=F(B_{1},v_{1}) \cup Y(B_{1},v_{1}) \cup C(B_{i},v')$. By twisting about $v_{1}$ and $C(B_{i},v')$ we create a new balanced component $H=C(B_{1},v_{1}) < \Sigma_{2}$, contradiction. Otherwise, $C(B_{i},v')$ consists of at least one balanced bridge, say $B$. $e$ and $h$ belong to different cocircuits in $\mathcal{C}^{*}(M(\Sigma.(B \cup {Y}) | Y))$. The remaining edges of $Y$ are partitioned into at least two cocircuits in $\mathcal{C}^{*}(M(\Sigma.(B_{1} \cup {Y}) | Y))$. Therefore, there is no pair of avoiding cocircuits for $B$ and $B_{1}$ implying that $B$ and $B_{1}$ are non avoiding, contradiction.

\noindent
\underline{\bf{Case 3: $w=v_{0}$ and $v \neq v_{1}$ for every $B_{i}$}}

\noindent
(i) $Y(B_{0},v_{0})$ consists of edges of the same sign i.e., $Y(B_{0},v_{0})=Y^{+}(B_{0},v_{0})$
\newline

Then there is no balanced component $C(B_{i},v')$ with $v'$ vertex of attachment of $B_{i}$ whose associated $Y(B_{i},v')$ contains edges of $Y$ with different sign. By assumption $C_{B_{i}}^{*}$ determines $v_{0}$ in $\Sigma$ for every $B_{i}$ and by definition $C_{B_{i}}^{*}$ corresponds to a set of half-edges in $\Sigma.(B_{i} \cup {Y}) | Y$. Thus $C_{B_{i}}^{*}=Y(B_{i},v_{0})$. Then \eqref{eq:1} and \eqref{eq:3} become:
\begin{equation} \label{eq_8}
Y^{+}(B_{0},v_{0}) \cup Y(B_{1},v_{1})=Y
\end{equation}
\begin{equation} \label{eq_9}
Y(B_{i},v_{0}) \cup Y(B_{1},v)=Y
\end{equation}

By \eqref{eq_8} there are no edges of $Y$ with endvertices in $F(B_{0},v_{0})$ and $F(B_{1},v_{1})$. Moreover, by \eqref{eq_9} there are no edges of $Y$ with 
endvertices in $F(B_{i},v_{0})$ and $F(B_{1},v)$ for every $B_{i}$. $Q(v)$ will denote the group of balanced bridges $B_{i}$ such that \eqref{eq_9} holds. 

\begin{figure}[hbt]
\begin{center}
\psfrag{S1}{\footnotesize $\Sigma_1$}
\psfrag{S2}{\footnotesize $\Sigma_2$}
\psfrag{B0}{\footnotesize $B_0$}
\psfrag{B1}{\footnotesize $B_1$}
\psfrag{Bi}{\footnotesize $B_i$}
\psfrag{v0}{\tiny $v_0$}
\psfrag{v1}{\tiny $v_1$}
\psfrag{v'}{\tiny $v'$}
\psfrag{Y}{\footnotesize $Y$}
\psfrag{C1}{\tiny $C(B_1,v_1)$}
\psfrag{F0}{\tiny $F(B_0,v_0)$}
\psfrag{F1}{\tiny $F(B_1,v_1)$}
\psfrag{Fi}{\tiny $F(B_{i},v_0)$}
\includegraphics*[scale=0.31]{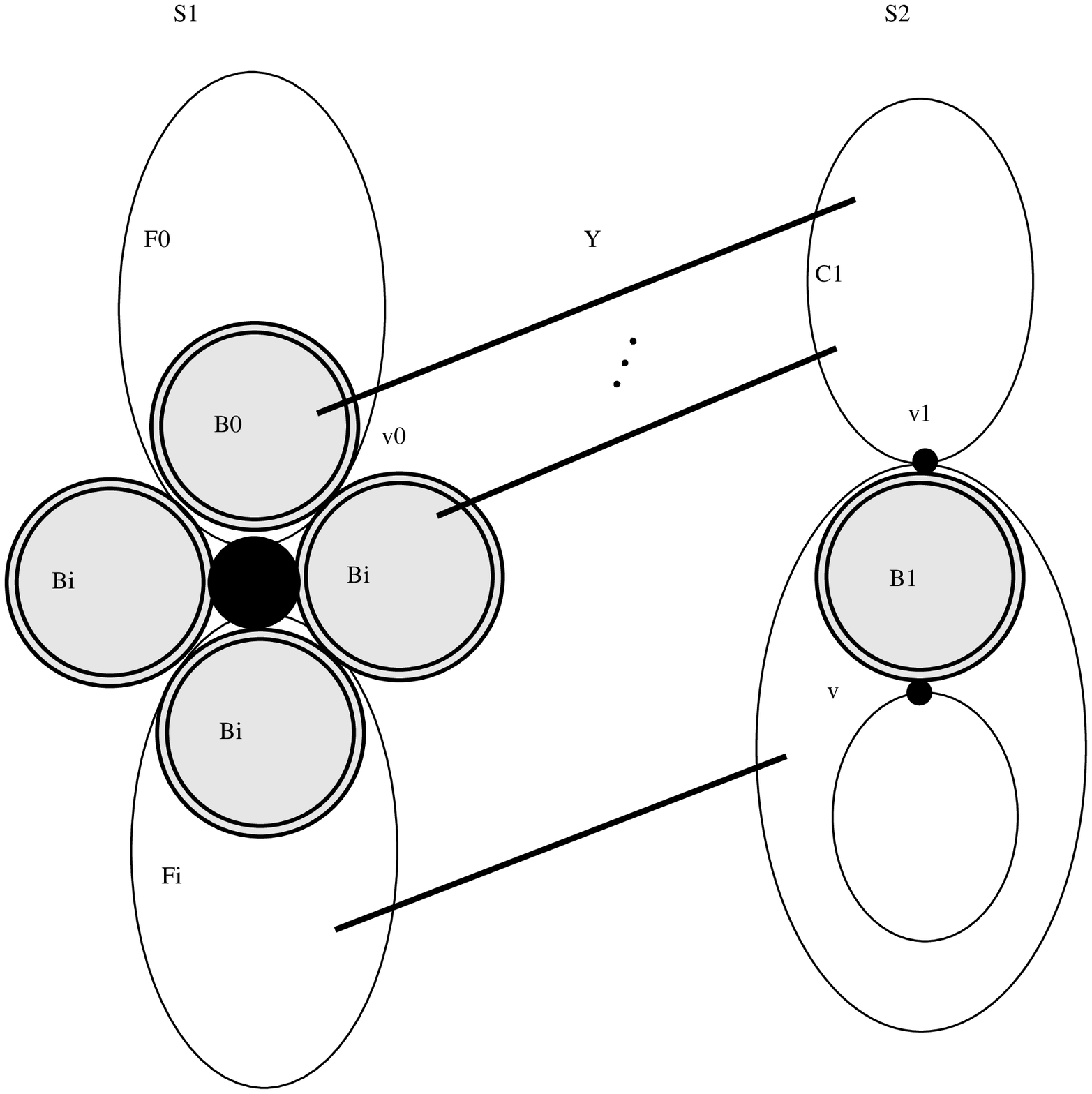}
\end{center} 
\caption{Case 3.}
\label{fig_star_case3}
\end{figure}

Then there is a 2-separation $\{T(v), E(\Sigma) \dof T(v)\}$ with $T(v)=C(B_{1},v) \cup Y(B_{1},v) \cup (\underset{B_i \in Q(v)} {\bigcup}F(B_{i},v_{0}))$. By twisting about $v_{0},v_{1}$ we create a new balanced component $H$ where $E(H)=E(\Sigma_{2})-E(B_{1})< E(\Sigma_{2})$, contradiction (see Figure~\ref{fig_star_case3}).

(ii) $Y(B_{0},v_{0})$ consists of edges of different sign i.e., $Y(B_{0},v_{0})=Y^{+}(B_{0},v_{0}) \cup Y^{-}(B_{0},v_{0})$
\\

If there is no balanced component $C(B_{i},v')$, $v' \in V(B_{i})$ of $B_{i}$ which contains the endvertices of edges of $Y$ of different sign then this case follows case 3(i). Otherwise, we will consider the case where $C_{B_{i}}^{*}$ corresponds to a bond in $\Sigma.(B_{i} \cup {Y}) | Y))$ consisting of half-edges and the positive edges of $Y$ with an endvertex in $C(B_{i},v')$ i.e., $C_{B_{i}}^{*}=Y(B_{i},v_{0}) \cup Y^{+}(B_{i},v')$. Then \eqref{eq:1} and \eqref{eq:3} become 
\begin{equation*} \tag{\ref{eq:6}}
Y^{+}(B_{0},v_{0}) \cup Y(B_{1},v_{1})=Y
\end{equation*}
\begin{equation} \label{eq:8}
Y(B_{i},v_{0}) \cup Y^{+}(B_{i},v') \cup Y(B_{1},v)=Y
\end{equation}

\begin{figure}[hbtp]
\begin{center}
\psfrag{S1}{\footnotesize $\Sigma_1$}
\psfrag{S2}{\footnotesize $\Sigma_2$}
\psfrag{B0}{\footnotesize $B_0$}
\psfrag{B1}{\footnotesize $B_1$}
\psfrag{Bi}{\footnotesize $B_i$}
\psfrag{v0}{\tiny $v_0$}
\psfrag{v1}{\tiny $v_1$}
\psfrag{v'}{\tiny $v'$}
\psfrag{v}{\tiny $v$}
\psfrag{Y}{\footnotesize $Y$}
\psfrag{C1}{\tiny $C(B_1,v_1)$}
\psfrag{F0}{\tiny $F(B_0,v_0)$}
\psfrag{F1}{\tiny $F(B_1,v_1)$}
\psfrag{Fi}{\tiny $F(B_{i},v_0)$}
\includegraphics*[scale=0.31]{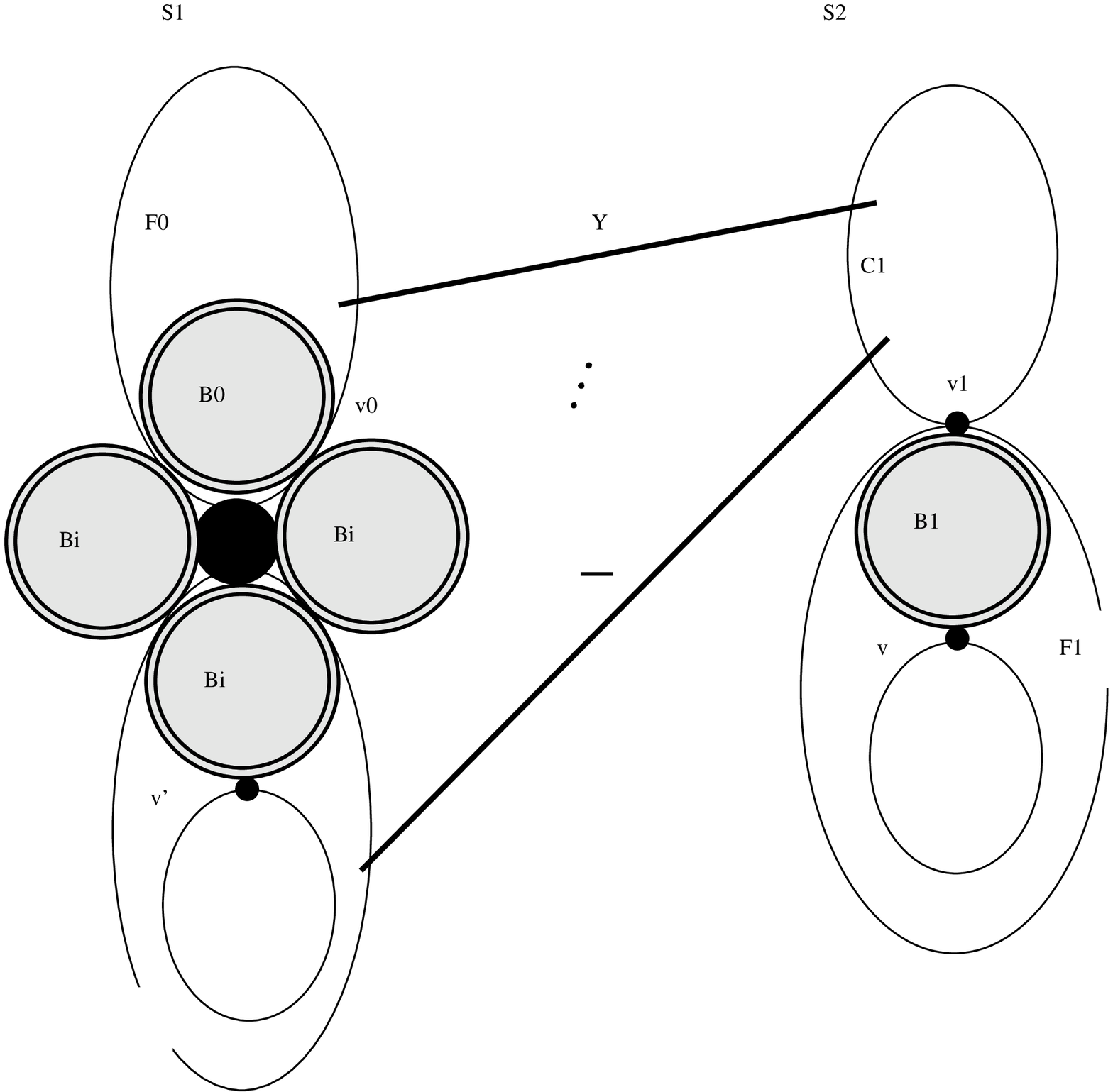}
\end{center} 
\caption{Case 3a.}
\label{fig_star_case3a}
\end{figure}

By hypothesis there is a negative edge of $Y$ with an endvertex in $C(B_{i},v')$. By \eqref{eq:6} this edge has its other endvertex in $C(B_{1},v_{1})$ which contradicts \eqref{eq:8} since $C(B_{i},v') \subseteq F(B_{i},v_{0}) \subseteq C(B_{0},v_{0})$ (see Figure~\ref{fig_star_case3a}).

To complete the proof it must be shown that the twistings which are defined by the 2-separations in the above cases form a new signed graph each time whose signed-graphic matroid is $M$. Since $\Sigma$ is a planar embedding of a cylindrical signed graph whose signed-graphic matroid is quaternary and nonbinary it has two negative faces. $\Sigma$ has a negative outface thus there is one negative inner face. This inner face can belong only to one of the twisting parts which are defined by a 2-separation. Therefore, the twisting part which does not contain the negative inner face does not contain negative cycles. Thus it is balanced and the twisting leave the matroid unchanged.

\end{proof}

Nongraphic cocircuits of a signed-graphic matroid correspond to either an unbalancing bond or to a double bond in their signed-graphic representations. The last results of this section constitute an attempt to separate an unbalancing bond from a double bond in matroidal terms. For the moment this is only possible for the class of signed-graphic matroids (see theorem~\ref{characterisationub}). 

\begin{definition}
$\mathcal{U}$-\textbf{cocircuit} $Y$ is a cocircuit of a connected signed-graphic matroid $M$ such that for any two elements $e,f$ of $Y$ with $r(\{e,f\})=2$ there is an exact connected 2-separation $(A, B\cup \{e,f\})$ of $M \dof (Y-\{e,f\})$ such that $(A\cup \{e,f\}, B)$ is an exact connected 1 or 2-separation of $M \dof (Y-\{e,f\})$.
\end{definition}

\begin{proposition}
If $Y \in \mathcal{C}^{*}(M)$ and $e,f \in Y$ then $\{e,f\} \in \mathcal{C}^{*}(M \dof (Y-\{e,f\}))$. 
\end{proposition}
\begin{proof}
Since $\mathcal{C}^{*}(M \dof (Y-\{e,f\}))=\mathcal{C}(M^{*} \cof (Y-\{e,f\}))=\text{minimal nonempty} \\ \{C-(Y-\{e,f\}) : C \in \mathcal{C}(M^{*}(\Sigma) \}$, the cocircuits of $M \dof (Y-\{e,f\})$ are the minimal nonempty subsets of $C-(Y-\{e,f\})$ such that $C$ is a cocircuit of $M(\Sigma)$. Then $\{e,f\} \in \mathcal{C}^{*}(M \dof (Y-\{e,f\})$ since otherwise $\{e\}$ or $\{f\}$ is a cocircuit of $M \dof (Y-\{e,f\})$ implying that $Y-\{f\}$ or $Y-\{e\}$ is a cocircuit of $M$ which is a contradiction.
\end{proof}

The following theorem constitutes a characterisation of an unbalancing bond for the class of signed-graphic matroids that upon its deletion from a connected signed graph we obtain exactly one unbalanced connected component.

\begin{theorem} \label{characterisationub}
Let $M(\Sigma)$ be a connected signed-graphic matroid such that $\Sigma$ is connected jointless and unbalanced, then $Y$ is a $\mathcal{U}$-cocircuit of $M(\Sigma)$ if and only if $Y$ is an unbalancing bond in $\Sigma$ with one unbalanced connected component.
\end{theorem}

\begin{proof}
Suppose that $Y$ is a $\mathcal{U}$-cocircuit of $M(\Sigma)$ and a double bond in $\Sigma$. Hence $\Sigma \dof Y$ consists of one balanced component denoted by $\Sigma_{2}$ and one or more unbalanced components. Moreover, consider two elements $e,f$ of $Y$ with $r(\{e,f\})=2$ such that $e$ has one endvertex in some unbalanced connected component $S_{1}$ and the other in $\Sigma_{2}$ while $f$ has both endvertices in $\Sigma_{2}$ (see Figure~\ref{ef}). By definition of deletion in matroids if $Y \in \mathcal{C}^{*}(M(\Sigma))$ then $\{e,f\} \in \mathcal{C}^{*}(M(\Sigma) \dof (Y-\{e,f\}))$. Thus, $\{e,f\}$ is a double bond in $\Sigma \dof (Y-\{e,f\})$.  

\begin{figure}[hbtp]
\begin{center}
\mbox{
\psfrag{e}{\footnotesize $e$}
\psfrag{f}{\footnotesize $f$}
\psfrag{S2}{\footnotesize $\Sigma_{2}$}
\psfrag{S1}{\footnotesize $S_{1}$}
\psfrag{S3}{\footnotesize $S_{2}$}
\psfrag{S4}{\footnotesize $S_{3}$}
\includegraphics*[scale=0.30]{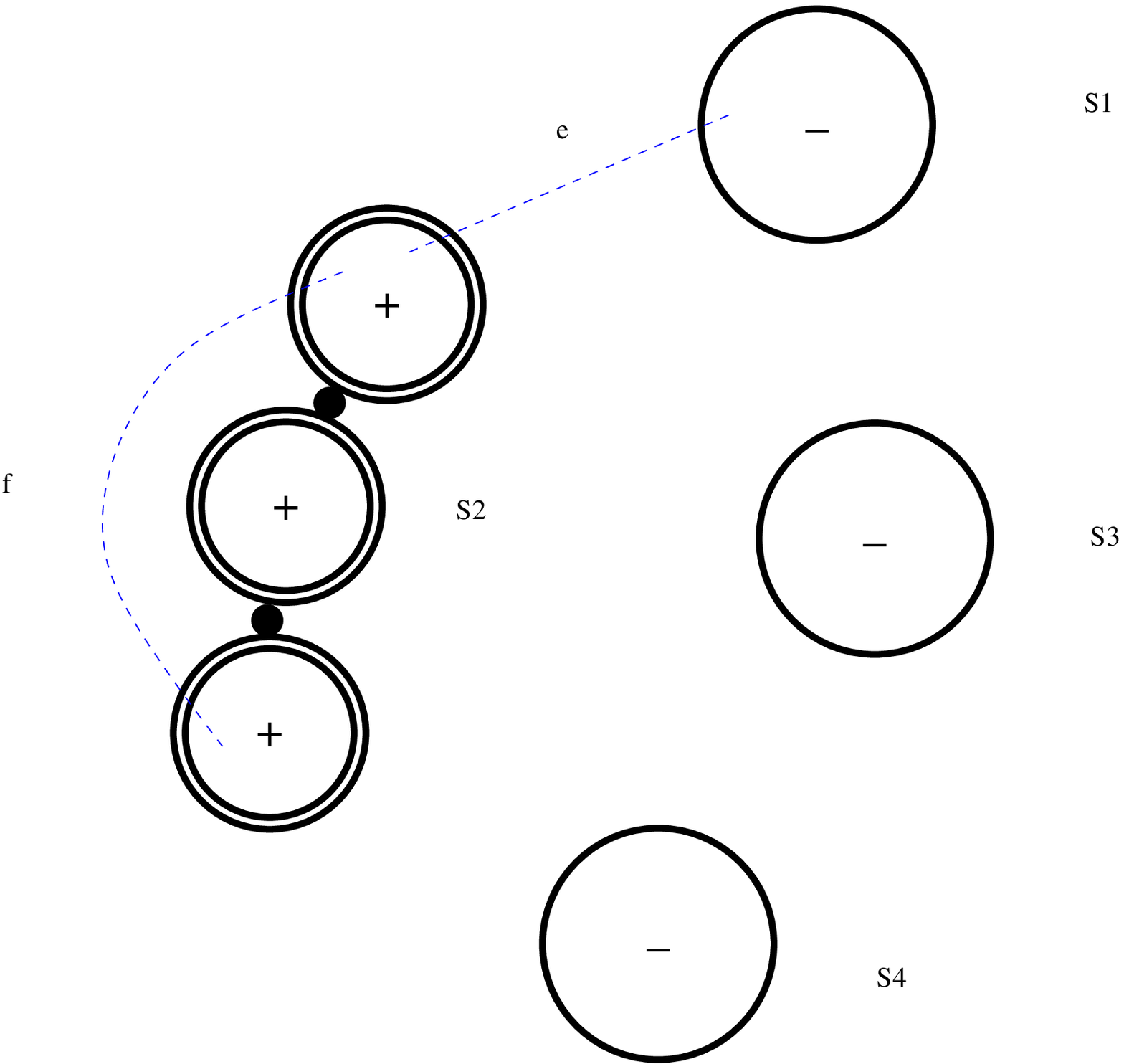}
}
\caption{$e,f$ in $\Sigma \dof (Y-\{e,f\})$}
\label{ef}
\end{center}
\end{figure}

We shall show that for $e,f$ there is no exact connected 2-separation $(A, B\cup \{e,f\})$ of $M(\Sigma) \dof (Y-\{e,f\})$ such that $(A\cup \{e,f\}, B)$ is an exact connected 1 or 2-separation of $M(\Sigma) \dof (Y-\{e,f\})$. Equivalently, for every exact connected 2-separation $(A, B\cup \{e,f\})$ of $M(\Sigma) \dof (Y-\{e,f\})$, $(A\cup \{e,f\}, B)$ is neither an exact connected 1 or 2-separation. Assume that there exists an exact connected 2-separation $(A, B\cup \{e,f\})$ of $M(\Sigma) \dof (Y-\{e,f\})$ since otherwise $Y$ is not a u-cocircuit. For convenience we shall denote $\Sigma \dof (Y-\{e,f\})$ as $\Sigma'$. 

\begin{claim}\label{exact}
If there exists an exact connected 2-separation $(A, B\cup \{e,f\})$ in $M(\Sigma')$ and $\Sigma'$ consists only of unbalanced connected components then $\Sigma'$ is connected.
\end{claim}
\begin{proof}
Suppose that $(A, B\cup \{e,f\})$ is an exact connected 2-separation of $M(\Sigma')$. Hence $\Sigma'[A ], \Sigma'[B \cup \{e,f\}]$ are connected, $r(A)+r(B\cup \{e,f\})=r(M(\Sigma'))+1$ (1) and $min\{|A |,|B\cup \{e,f\}|\} \geq 2$. By assumption $r(\Sigma')=|v(\Sigma')|$. Assume first that $\Sigma'[B \cup \{e,f\}]$ is balanced. Then $r(B \cup \{e,f\})=|v(\Sigma'[B \cup \{e,f\}])|-1$. If $ \Sigma'[A]$ is balanced, $r(A)=|v(\Sigma'[A])|-1$ and from (1) $|V(A) \cap V(B\cup \{e,f\})|=3$.Otherwise $\Sigma'[A]$ is unbalanced, $r(A)=|v(\Sigma'[A])|$ and from (1) $|V(A) \cap V(B\cup \{e,f\})|=2$. Now assume that $\Sigma'[B \cup \{e,f\}]$ is unbalanced. Hence $r(B \cup \{e,f\})=|v(\Sigma'[B \cup \{e,f\}])|$. If $ \Sigma'[A]$ is balanced then $r(A)=|v(\Sigma'[A])|-1$  and from (1) $|V(A) \cap V(B \cup \{e,f\})|=2$. Otherwise $\Sigma'[A]$ is unbalanced, $r(A)=|v(\Sigma'[A])|$ and from (1) $|V(A) \cap V(B \cup \{e,f\})|=1$. In all cases $\Sigma'[B \cup \{e,f\}], \Sigma'[A]$ have common vertices and are connected which implies that $\Sigma'$ is connected.
\end{proof}

By the above claim we deduce that $\Sigma \dof Y$ consists of one unbalanced component and by proposition~\ref{k-bisep} that $(A, B\cup \{e,f\})$ is a connected 2-biseparation of $\Sigma'$. In addition by definition of $2$-biseparation $(A, B\cup \{e,f\})$ satisfies one of the following three cases: \\ 
\noindent
\underline{{\bf Case 1.a:}} $|V(A) \cap V(B \cup \{e,f\})|=3$ and $\Sigma'[A], \Sigma'[B \cup \{e,f\}]$ are  both connected and balanced. \\
\noindent
$S_{1}$ is contained neither to $\Sigma'[A]$ nor to $\Sigma'[B \cup \{e,f\}]$. Thereby the edges of the negative cycles of $S_{1}$ are partitioned in $\Sigma'[A]$ and $\Sigma'[B \cup \{e,f\}]$. Since $e,f$ belong to the edges of $\Sigma'[B \cup \{e,f\}]$, the aforementioned partitioning creates at least two common vertices between $\Sigma'[A]$ and $\Sigma'[B \cup \{e,f\}]$. By the fact that $\Sigma_{2} \cup \{f\}$ is unbalanced there are two distinct paths in $\Sigma'$ from the endvertex of $e$ in $\Sigma_{2}$ to the endvertices of $f$ in $\Sigma_{2}$. Since $\Sigma'[B \cup \{e,f\}]$ is balanced there is one edge of these paths contained in $\Sigma'[A]$ implying that $|V(A) \cap V(B \cup \{e,f\})| \geq 4$ which is a contradiction.  \\
\noindent
\underline{{\bf Case 1.b:}} $|V(A) \cap V(B \cup \{e,f\})|=2$ and exactly one of $\Sigma'[A], \Sigma'[B \cup \{e,f\}]$ is balanced. 
\noindent
By the fact that $\Sigma'[\Sigma_{2} \cup \{f\}]$ is unbalanced, in $\Sigma'$ there are two distinct paths from the endvertex of $e$ in $\Sigma_{2}$ to the endvertices of $f$ in $\Sigma_{2}$. Moreover $\Sigma'[B \cup \{e,f\}]$ contains edges of $\Sigma_{2}$ since otherwise $|V(A) \cap V(B \cup \{e,f\})| \geq 3$. If $\Sigma'[B \cup \{e,f\}]$ is unbalanced then $\Sigma'[A]$ is balanced. Hence either $S_{1}$ is contained in $\Sigma'[B \cup \{e,f\}]$ or $\Sigma'[B \cup \{e,f\}]$ contains edges of $S_{1}$. In the first case $S_{1}$ is contained in $\Sigma'[B \cup \{e,f\}]$ implying that $\Sigma'[A]$ is a subgraph of $\Sigma_{2}$. Then $|V(A \cup \{e,f\}) \cap V(B) | \geq 4$. Furthermore $\Sigma'[A \cup \{e,f\}]$ is balanced and $\Sigma'[B]$ is unbalanced. By corollary~\ref{2-bic} $\Sigma'$ is vertically 2-biconnected. Thus $(A \cup \{e,f\}, B)$ is neither 1 or 2-biseparation in $\Sigma'$. Since $M(\Sigma')$ is connected, $(A \cup \{e,f\}, B)$ cannot be 1-separation. By corollary~\ref{k-bisep} $(A \cup \{e,f\}, B)$ is not an exact 2-separation in $M(\Sigma')$ with connected parts. If $(A \cup \{e,f\}, B)$ is a 2-separation in $M(\Sigma')$ it must be an exact 2-separation since $M(\Sigma')$ is not 3-connected. Thereby $(A \cup \{e,f\}, B)$ has not connected parts. Therefore $(A\cup \{e,f\}, B)$ is neither a connected 1 or 2-separation in $M(\Sigma')$. In the second case $\Sigma'[B \cup \{e,f\}]$ contains edges of $S_{1}$. By assumption $|V(A) \cap V(B \cup \{e,f\})|=2$ implying that $\Sigma_{2} \subseteq \Sigma'[B \cup \{e,f\}]$. Then $|V(A \cup \{e,f\}) \cap V(B) | \geq 5$. Therefore $(A \cup \{e,f\}, B)$ is neither 1 or 2-biseparation in $\Sigma'$. If $\Sigma'[B \cup \{e,f\}]$  is balanced, then $\Sigma'[A]$ is unbalanced. If $S_{1} \subseteq \Sigma'[A]$ then $\Sigma_{2}$ cannot be contained in $\Sigma'[B \cup \{e,f\}]$ and $|V(A) \cap V(B \cup \{e,f\})| \geq 3$ which is a contradiction. Hence there are edges of $S_{1}$ belonging to $\Sigma'[A]$. Since $|V(A) \cap V(B \cup \{e,f\})|=2$ it follows that $\Sigma_{2} \subseteq \Sigma'[B \cup \{e,f\}]$ and $\Sigma'[A] \subseteq S_{1}$. Therefore $|V(A \cup \{e,f\}) \cap V(B) | \geq 3$ and $(A \cup \{e,f\}, B)$ is neither 1 or 2-biseparation in $\Sigma'$. Using the same arguments as above $(A\cup \{e,f\}, B)$ is neither a connected 1 or 2-separation in $M(\Sigma')$.  \\
\noindent
\underline{{\bf Case 1.c:}} $|V(A) \cap V(B \cup \{e,f\})|=1$ and both $\Sigma'[A], \Sigma'[B \cup \{e,f\}]$ are unbalanced. 
\noindent
Since $e,f$ are edges of $\Sigma'[B \cup \{e,f\}]$ there are no edges of $\Sigma_{2}$ belonging to $\Sigma'[A]$ since otherwise $|V(A) \cap V(B \cup \{e,f\})| \geq 2$. Hence $\Sigma_{2}$ is contained to $\Sigma'[B \cup \{e,f\}]$. Moreover since $|V(A) \cap V(B \cup \{e,f\})|=1$, $S_{1} \subseteq \Sigma'[A]$. Therefore $(S_{1}, \Sigma_{2} \cup \{e,f\})$ is the specified 2-biseparation of $\Sigma'$ with connected parts. However $(S_{1} \cup \{e,f\}, \Sigma_{2})$ is neither 1 or 2-biseparation of $\Sigma'$ and by proposition \ref{k-bisep} $(A \cup \{e,f\}, B)$ is neither a connected 1 or 2-separation of $M(\Sigma')$. In all the above cases we reached a contradiction due to hypothesis that $Y$ is a $\mathcal{U}$-cocircuit. 

Assume that $Y$ is a $\mathcal{U}$-cocircuit of $M(\Sigma)$ and an unbalancing bond in $\Sigma$ such that $\Sigma \dof Y$ consists of one balanced component denoted also by $\Sigma_{2}$ and more than one unbalanced components. Let $S_{1},S_{2}$ be two of them. Consider two elements of $Y$ with $r(\{e,f\})=2$ such that $e$ has one endvertex in $\Sigma_{2}$ and the other in $S_{1}$ while $f$ has one endvertex in $\Sigma_{2}$ and the other in $S_{2}$. By definition of deletion in matroids if $Y \in \mathcal{C}^{*}(M(\Sigma))$ then $\{e,f\} \in \mathcal{C}^{*}(M(\Sigma'))$. Thus, $\{e,f\}$ is an unbalancing bond in both of the above cases in $\Sigma'$. Assume that there exists an exact connected 2-separation $(A, B\cup \{e,f\})$ of $M(\Sigma')$. Then by claim~\ref{exact}, $\Sigma'$ is connected. By corollary \ref{k-bisep} $(A, B\cup \{e,f\})$ is a connected 2-biseparation in $\Sigma'$. Thus $\Sigma \dof Y$ has exactly two unbalanced components $S_{1}, S_{2}$. By definition of 2-biseparation $(A, B\cup \{e,f\})$ satisfies one of the following three cases:\\
\noindent
\underline{{\bf Case 2.a:}} $|V(A) \cap V(B\cup \{e,f\})|=3$ and $ \Sigma'[A], \Sigma'[B \cup \{e,f\}]$ are both connected and balanced. \\
\noindent
Since $ \Sigma'[A], \Sigma'[B \cup \{e,f\}]$ are both connected and balanced, they must both contain edges of $S_{1}$ and $S_{2}$. Therefore $\Sigma'[A]$ is disconnected which is a contradiction. \\
\noindent
\underline{{\bf Case 2.b:}} $|V(A) \cap V(B\cup \{e,f\})|=2$ and exactly one of $ \Sigma'[A], \Sigma'[B \cup \{e,f\}]$ is balanced. \\
\noindent
Suppose that $\Sigma'[B\cup \{e,f\}]$ is balanced and that it contains edges of $S_{2}$. $\Sigma'[B\cup \{e,f\}]$ contains $e,f$ and $|V(A) \cap V(B\cup \{e,f\})|=2$ which imply that $S_{1} \subseteq \Sigma'[A]$ and $\Sigma_{2} \subseteq \Sigma'[B \cup \{e,f\}]$. Then $(A \cup \{e,f\}, B)$ is not connected 1 or 2-biseparation of $\Sigma'$. Thus $(A \cup \{e,f\}, B)$ is not exact connected 1 or 2-separation of $M(\Sigma')$. Otherwise $\Sigma'[B \cup \{e,f\}]$ does not contain edges of $S_{2}$. Since $S_{1}$ cannot be a subgraph of $\Sigma'[B\cup \{e,f\}]$, $\Sigma'[A]$ contains edges of $S_{1}$. It follows that $\Sigma'[A]$ is disconnected which is a contradiction. Suppose now that $\Sigma'[B \cup \{e,f\}]$ is unbalanced and wlog that it contains edges of $S_{2}$. Since $|V(A) \cap V(B \cup \{e,f\})|=2$ either $S_{1} \subseteq \Sigma'[B \cup \{e,f\}]$, $S_{2} \subseteq \Sigma'[B \cup \{e,f\}]$ and $\Sigma'[A] \subseteq \Sigma_{2}$ or $S_{1} \subseteq \Sigma'[B \cup \{e,f\}]$ and $\Sigma_{2} \subseteq \Sigma'[B \cup \{e,f\}]$ and $\Sigma'[A] \subseteq S_{2}$. In both cases $(A\cup \{e,f\}, B)$ is not connected 1 or 2-separation of $M(\Sigma')$. Otherwise $\Sigma'[B\cup \{e,f\}]$ does not contain edges of $S_{2}$ which is a contradiction since $\Sigma'[A]$ is balanced.\\
\noindent
\underline{{\bf Case 2.c:}} $|V(A) \cap V(B)\cup \{e,f\}|=1$ and both $\Sigma'[A], \Sigma'[B\cup \{e,f\}]$ are connected and unbalanced. \\
\noindent
Since both $\Sigma'[A], \Sigma'[B\cup \{e,f\}]$ are connected and unbalanced they must contain edges of $S_{1}$ and $S_{2}$, which is a contradiction since $\Sigma'[A]$ is disconnected. 

Conversely, assume that $Y$ is an unbalancing bond in $\Sigma$ with one unbalanced connected component. The unbalanced and the balanced component of $\Sigma'$ shall be denoted by $\Sigma_{1}$  and $\Sigma_{2}$ respectively. Moreover, consider two elements $e,f$ of $Y$ such that $r(\{e,f\})=2$. We shall show that for any two elements $e,f \in Y$ with $r(\{e,f\}=2$ there exists an exact connected 2-separation $(A, B\cup \{e,f\})$ of $M(\Sigma')$ such that $(A\cup \{e,f\}, B)$ is an exact connected 1 or 2-separation of $M(\Sigma')$. We encounter the following four cases: 

\begin{figure}[hbtp]
\begin{center}
\mbox{
\subfigure[Case 3.a]
{
\psfrag{e}{\footnotesize $e$}
\psfrag{f}{\footnotesize $f$}
\psfrag{S1}{\footnotesize $\Sigma_{1}$}
\psfrag{S2}{\footnotesize $\Sigma_{2}$}
\includegraphics*[scale=0.30]{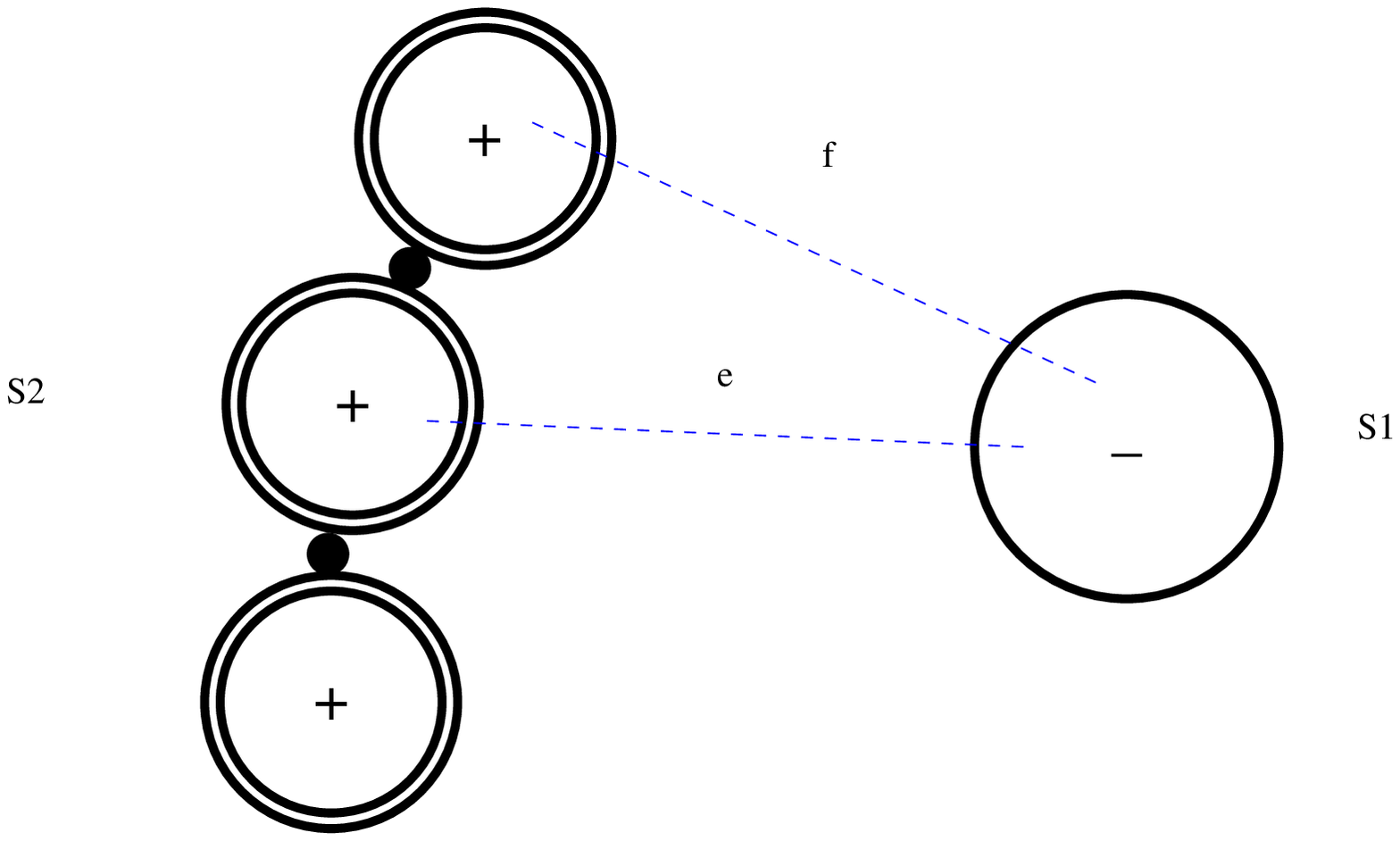}
}\quad
\subfigure[Case 3.b]
{
\psfrag{e}{\footnotesize $e$}
\psfrag{f}{\footnotesize $f$}
\psfrag{S1}{\footnotesize $\Sigma_{1}$}
\psfrag{S2}{\footnotesize $\Sigma_{2}$}
\includegraphics*[scale=0.30]{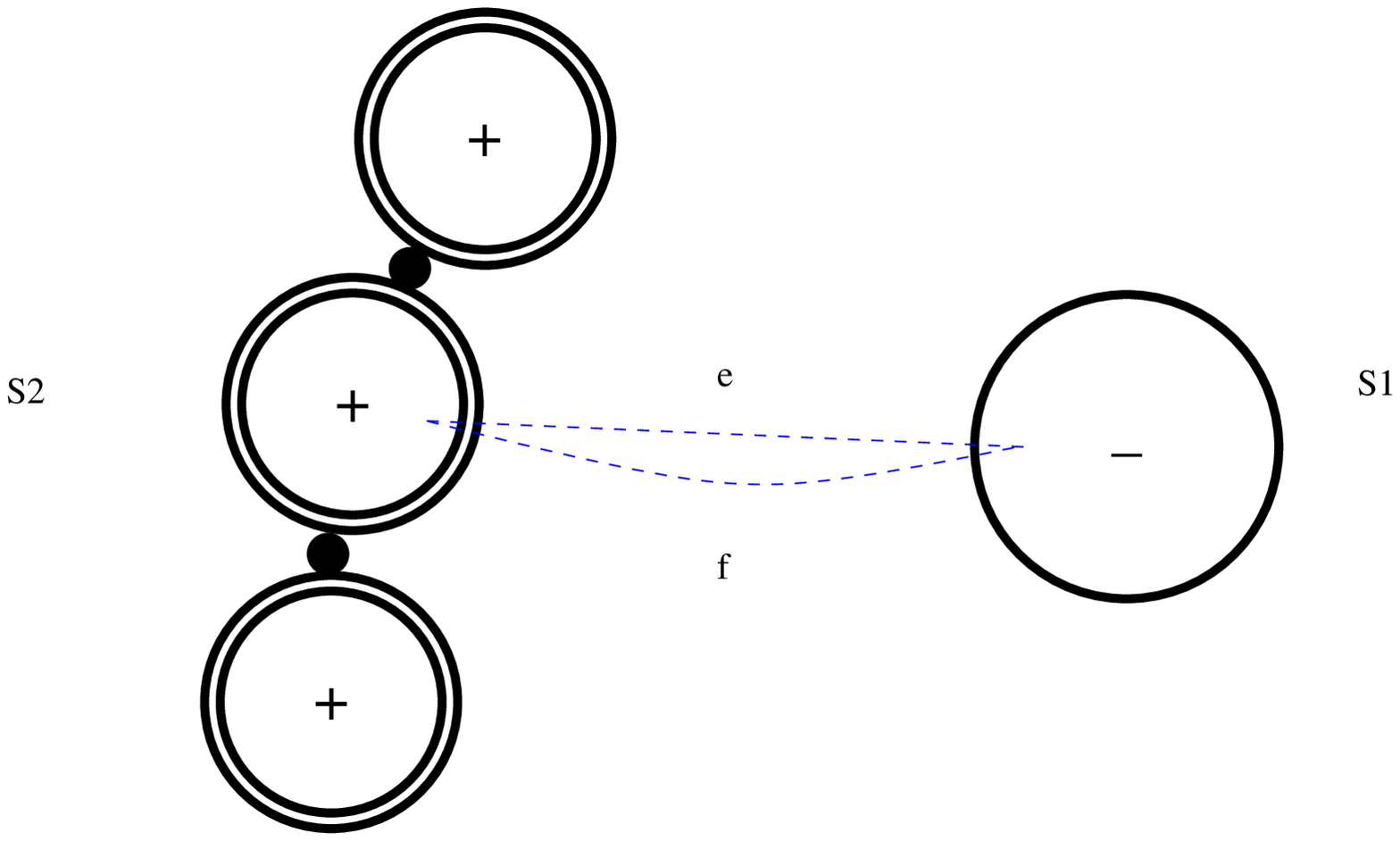}
}}
\caption{$e,f$ in $\Sigma \dof (Y-\{e,f\})$}
\label{efab}
\end{center}
\end{figure}

\noindent
\underline{{\bf Case 3.a:}} $e,f$ have distinct endvertices in $\Sigma_{1}$ and $\Sigma_{2}$ (see figure~\ref{efab}(a)). \\
\noindent
$(\Sigma_{1},\Sigma_{2} \cup \{e,f\})$ is a connected 2-biseparation in $\Sigma'$ independently of the sign of $e,f$ since $min\{|\Sigma_{1}|,|\Sigma_{2} \cup \{e,f\}|\} \geq 2$ and $|V(\Sigma_{1}) \cap V(\Sigma_{2} \cup \{e,f\})|=2$. By corollary~\ref{k-bisep}, $(\Sigma_{1},\Sigma_{2} \cup \{e,f\})$ is an exact connected 2-separation in $M(\Sigma')$. Then $(\Sigma_{1} \cup \{e,f\},\Sigma_{2})$ is a connected 2-biseparation in $\Sigma' $ and an exact connected 2-separation in $M(\Sigma')$. \\
\noindent
\underline{{\bf Case 3.b:}} $e,f$ have the same endvertex in $\Sigma_{1}$ and $\Sigma_{2}$ (see figure~\ref{efab}(b)). \\
\noindent
Since $r(\{e,f\})=2$, $e,f$ must have different sign. Then $(\Sigma_{1},\Sigma_{2} \cup \{e,f\})$ is a connected 2-biseparation in $\Sigma'$ since $min\{|\Sigma_{1}|,|\Sigma_{2} \cup \{e,f\}|\} \geq 2$ and $|V(\Sigma_{1}) \cap V(\Sigma_{2} \cup \{e,f\})|=1$.  By proposition~\ref{k-bisep} $, (\Sigma_{1},\Sigma_{2} \cup \{e,f\})$ is an exact connected 2-separation in $M(\Sigma')$. Then $(\Sigma_{1} \cup \{e,f\},\Sigma_{2})$ is a connected 1-biseparation in $\Sigma'$ and an exact connected 1-separation in $M(\Sigma')$. 

\begin{figure}[hbtp]
\begin{center}
\mbox{
\subfigure[Case 3.c]
{
\psfrag{e}{\footnotesize $e$}
\psfrag{f}{\footnotesize $f$}
\psfrag{S1}{\footnotesize $\Sigma_{1}$}
\psfrag{S2}{\footnotesize $\Sigma_{2}$}
\includegraphics*[scale=0.30]{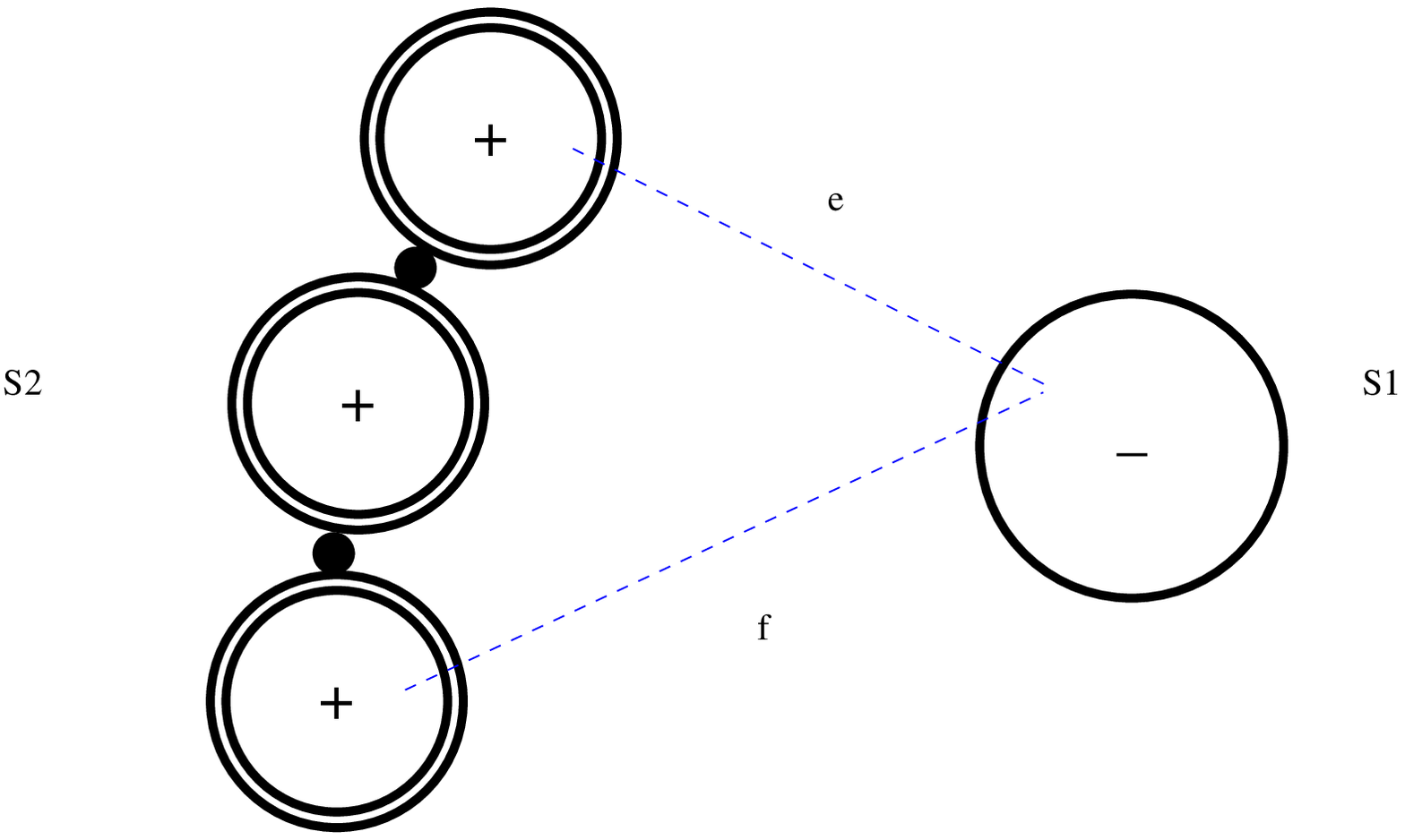}
}\quad
\subfigure[Case 3.d]
{
\psfrag{e}{\footnotesize $e$}
\psfrag{f}{\footnotesize $f$}
\psfrag{S1}{\footnotesize $\Sigma_{1}$}
\psfrag{S2}{\footnotesize $\Sigma_{2}$}
\includegraphics*[scale=0.30]{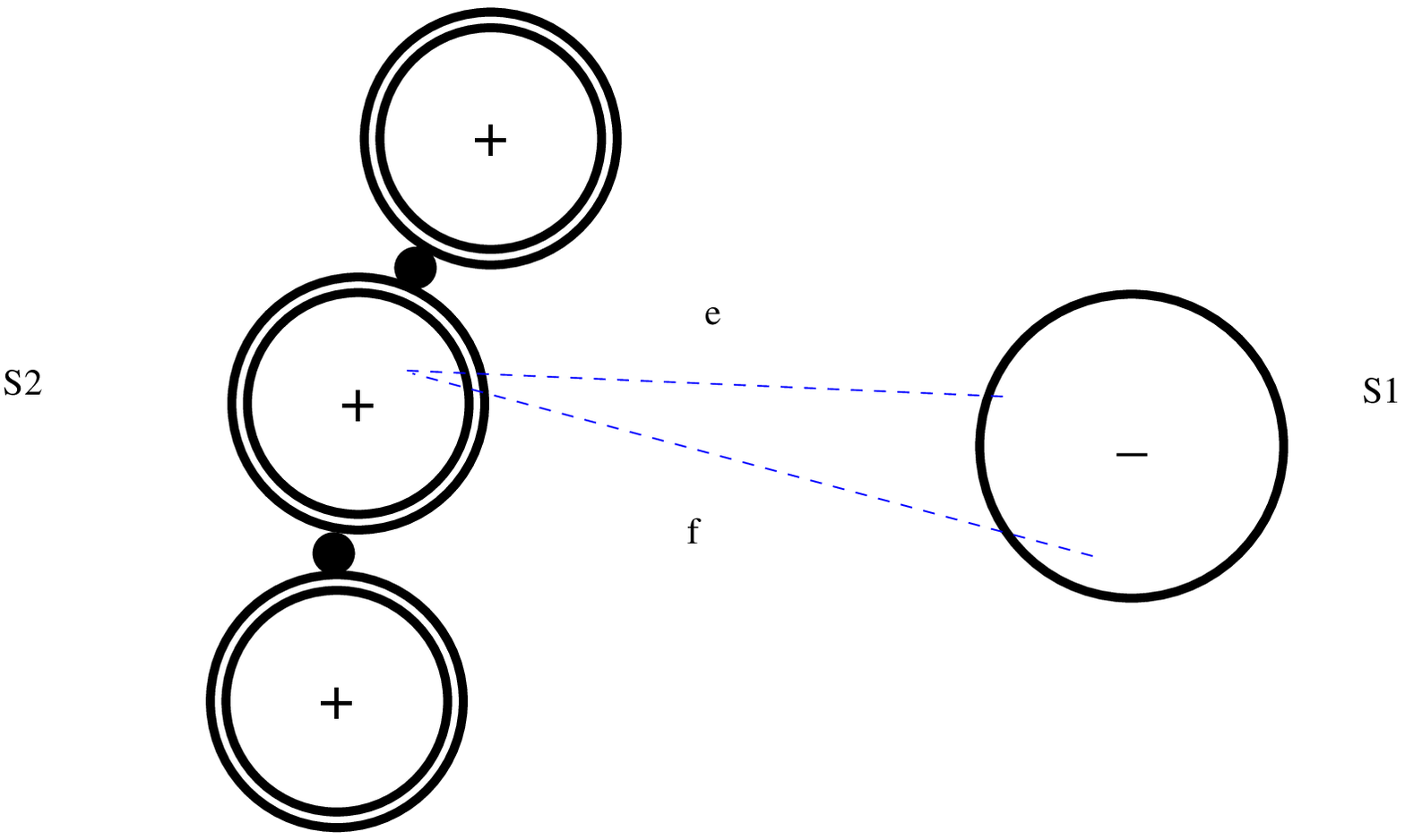}
}}
\caption{$e,f$ in $\Sigma \dof (Y-\{e,f\})$}
\label{efcd}
\end{center}
\end{figure}

\noindent
\underline{{\bf Case 3.c:}} $e,f$ have the same endvertex in $\Sigma_{1}$ but not in $\Sigma_{2}$ (see figure~\ref{efcd}(a)). \\
\noindent
If $e,f$ have different sign $(\Sigma_{1},\Sigma_{2} \cup \{e,f\})$ is a connected 2-biseparation in $\Sigma' $ since $\Sigma[\Sigma_{2} \cup \{e,f\}]$ is connected and unbalanced. Moreover, $(\Sigma_{1},\Sigma_{2} \cup \{e,f\})$ is an exact connected 2-separation in $M(\Sigma')$. Then $(\Sigma_{1} \cup \{e,f\},\Sigma_{2})$ is a connected 2-biseparation in $\Sigma'$ and an exact connected 2-separation in $M(\Sigma')$. Otherwise, $e,f$ have the same sign and $(\Sigma_{1} \cup \{e,f\},\Sigma_{2})$ is a connected 2-biseparation in $\Sigma'$ and an exact connected 2-separation in $M(\Sigma')$. Then $(\Sigma_{1},\Sigma_{2} \cup \{e,f\})$ is a connected 1-biseparation in $\Sigma'$ and an exact connected 1-separation in $M(\Sigma')$. \\
\noindent
\underline{{\bf Case 3.d:}} $e,f$ have the same endvertex in $\Sigma_{2}$ but not in $\Sigma_{1}$ (see figure~\ref{efcd}(b)). \\
\noindent
If $e,f$ have different sign $(\Sigma_{1},\Sigma_{2} \cup \{e,f\})$ is a connected 2-biseparation in $\Sigma' $ since $\Sigma[\Sigma_{2} \cup \{e,f\}]$ is connected, balanced and $|V(\Sigma_{1}) \cap V(\Sigma_{2} \cup \{e,f\})|=2$. Moreover, $(\Sigma_{1},\Sigma_{2} \cup \{e,f\})$ is an exact connected 2-separation in $M(\Sigma')$. Then $(\Sigma_{1} \cup \{e,f\},\Sigma_{2})$ is a connected 2-biseparation in $\Sigma \dof (Y-\{e,f\})$ and an exact connected 2-separation in $M(\Sigma')$. Otherwise, $e,f$ have the same sign and $(\Sigma_{1}, \Sigma_{2}\cup \{e,f\})$ is a connected 2-biseparation in $\Sigma'$ and an exact connected 2-separation in $M(\Sigma')$. Then $(\Sigma_{1}\cup \{e,f\}, \Sigma_{2})$ is a connected 1-biseparation in $\Sigma'$ and an exact connected 1-separation in $M(\Sigma')$. Therefore $Y$ is a $\mathcal{U}$-cocircuit of $M(\Sigma)$.

\end{proof}

\section{Signed-graphic matroid of $T_{6}$} \label{sec_T6}

\begin{lemma} \label{t6}
If $Y \in \mathcal{C}^{*}(M(T_{6}))$ is nongraphic and an unbalancing bond in $T_{6}$ then 
\begin{enumerate}[(i)]
\item[(i)]$Y$ is bridge-separable in $M(T_{6})$ and
\item[(ii)]$Y$ is a star of a vertex in $T_{6}$.
\end{enumerate}
\end{lemma}
\begin{proof}
Suppose that $Y$ is an unbalancing bond in $T_{6}$ such that the core is not a B-necklace. In all possible cases $Y$ is the star of a vertex. Consider for example $Y=\{-3,3,6,2\}$ or $Y=\{3,6,4,-6\}$ in figure \ref{T6}. If $Y$ is a double bond in $T_{6}$, consider $Y=\{1,2,5,-6,-4\}$ in figure \ref{T6}. Then the bridges of $Y$ are $B_{1}=\{-5\}, B_{2}=\{-1,-2,-3,4,6\}$. It is easy to check that in both cases $Y$is bridge-separable. 
\end{proof}

\begin{figure}[hbtp] 
\begin{center}
\includegraphics*[scale=0.65]{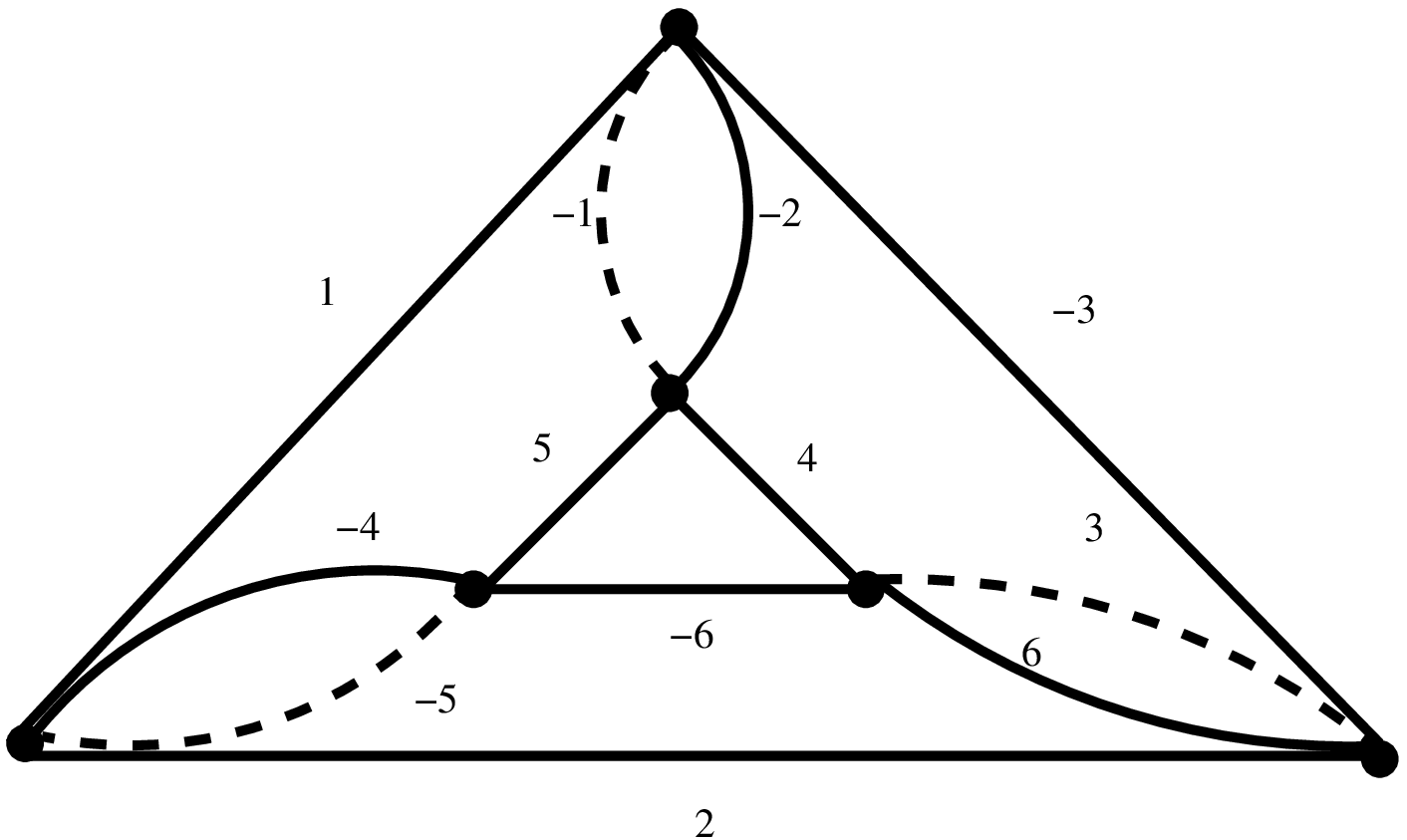}
\end{center} 
\caption{$T_{6}$}
\label{T6}
\end{figure}

\section{Decomposition of $GF(4), \neg GF(2)$ signed-graphic matroids} \label{sec_main_result}

\begin{proposition}
If $Y \in \mathcal{C^*}(M)$ then $Y \in \mathcal{C^*}(M \cto (U^{-} \cup Y))$.
\end{proposition}
\begin{proof}
The cocircuits of $M \cto (U^{-} \cup Y)$ are the circuits of $(M \cto (U^{-} \cup Y))^*=M^* \dto (U^{-} \cup Y)$, which are the cocircuits of $M$ contained in $U^{-} \cup Y$, so $Y$ is a cocircuit of $M \cto (U^{-} \cup Y)$.
\end{proof}

If $Y$ is a bridge-separable cocircuit of a matroid $M$ and $U^{-},U^{+}$ are the two classes of all avoiding bridges then we will call \emph{ $\mathcal{U}$-minor} the matroid $M \cto (U \cup Y)$, where $U \in \{U^{-},U^{+}\}$. The connectivity of $M$ is not inherited to its $\mathcal{U}$-minor matroids. This observation for signed-graphic matroids is stated in the next claim.

\begin{claim}
If $Y$ is a bridge-separable cocircuit of an internally 4-connected quaternary nonbinary matroid $M$ and $U^{-}$ is a class of all avoiding bridges then $M \cto (U^{-} \cup Y)$ is not internally 4-connected.
\end{claim}

Suppose that $Y$ is a bridge-separable cocircuit of an internally 4-connected quaternary nonbinary signed-graphic matroid $M$ with a jointless cylindrical signed-graphic representation. Consider for example the cylindrical signed graph $\Sigma$ in Figure~\ref{cylint4}. A bridge-separable nongraphic cocircuit of $M(\Sigma)$ is $Y_{40}=\{3,6,-4,5,-9,-1\}$. The bridges of $Y$ in $M(\Sigma)$ are $B_{1}=\{4\}$, $B_{2}=\{1,2,7,-2,-3,-5,-6,-7,-8\}$. Then $M(\Sigma \cof 4)$ is 2-connected.

\begin{lemma}
If $Y$ is a cocircuit of an internally 4-connected matroid $M$ and $U^{-}$ a class of all avoiding bridges of $Y$ in $M$ then $M \cto (U^{-} \cup Y)$ is connected.
\end{lemma}
\begin{proof}
Assume on the contrary that there is a separator $S \subseteq U^{-} \cup Y$ of $M \cto (U^{-} \cup Y)$. Then $r_{M \cto (U^{-} \cup Y)}(S) +r_{M \cto (U^{-} \cup Y)}((U^{-} \cup Y)-S)=r_{M \cto (U^{-} \cup Y)}(U^{-} \cup Y)$. Equivalently, $r_{M \cof U^{+}}(S) + r_{M \cof U^{+}}((U^{-} \cup Y) -S)=r_{M \cof U^{+}}(U^{-} \cup Y)$. By proposition 3.1.6 in \cite{Oxley:2011} since $U^{+} \subseteq E$ and $S \subseteq E-U^{+}$ and $(U^{-} \cup Y)-S \subseteq E-U^{+}$ it holds that $r_{M}(U^{+} \cup S)-r_{M}(U^{+})+ r_{M}(E-S)=r(E)$. Consider a basis $B_{S}$ of $M\dto S$ and a basis $B_{U^{+}}$ of $M\dto U^{+}$. Then $r_{M}(S)=|B_{S}|$ and $r_{M}(U^{+})=|B_{U^{+}}|$. By definition of bases of matroids $B_{S}\cup B_{U^{+}}$ is a basis of $M\dto (U^{+} \cup S)$. Therefore, $r_{M}(U^{+}\cup S)=|B_{S} \cup B_{U^{+}}|$. Since $S \subseteq U^{-} \cup Y$, $U^{+},S$ are disjoint sets of $E(M)-Y$ and thus of $E(M)$. Furthermore $B_{S} \cap B_{U^{+}}= \emptyset $ and $|B_{S} \cup B_{U^{+}}|=$ $ |B_{S}| + |B_{U^{+}}| $= $r_{M}(S)+ r_{M}(U^{+})$. Thus, $r_{M}(S)+ r_{M}(E-S)=r(E)$ a contradiction since $M$ is internally 4-connected.
\end{proof} 


\begin{proposition}
If $Y$ is a nongraphic bridge-separable cocircuit of $M$ with $U^{-}, U^{+}$ the two classes of all avoiding bridges of $Y$ and $M \cto (U^{+} \cup Y)$ is graphic then $Y$ is a nongraphic cocircuit of $M \cto (U^{-} \cup Y)$.
\end{proposition}
\begin{proof}
$Y$ is a cocircuit of $M \cto (U^{-} \cup Y)$ and nongraphic by hypothesis. Then $M \dof Y$ contains a minor $H$ isomorphic to one of the excluded minors of the class of graphic matroids, $F_{7}, F_{7}^{*}, M^{*}(K_{5}), M^{*}(K_{3,3})$. Since each excluded minor is connected and $H$ contains no element of $Y$, then $H$ is contained in a bridge of $Y$ in $M$. For any bridge $B \in U^{-}$ of $Y$ in $M$, since $B$ is a separator of $M \dof Y$ then it is also a separator of $M \cto (U^{-} \cup Y) \dof Y$. Moreover $M \cto U^{+}$ is graphic which implies that $H$ is not contained in any bridge of $U^{+}$. Therefore $H$ is a minor of a bridge of $U^{-}$ and $M \cto (U^{-} \cup Y) \dof Y$ is nongraphic. 
\end{proof}

\emph{Unbalancing cocircuit} will be called the nongraphic bridge-separable cocircuit $Y$ of a quaternary nonbinary connected matroid $M$ so that exactly one of $M \cto (U^{-} \cup Y)$, $M \cto (U^{+} \cup Y)$ is non-graphic where $U^{-}, U^{+}$ are the two classes of all avoiding bridges of $Y$. A special case of an unbalancing cocircuit is the star cocircuit. \emph{Star cocircuit} is a nongraphic cocircuit $Y$ of a quaternary nonbinary connected matroid $M$ such that all bridges of $Y$ in $M$ are avoiding. Since $Y$ is nongraphic, it follows that $M \cto (U^{-} \cup Y)$ is non-graphic where $U^{-}$ is the class of all avoiding bridges of $Y$. A \emph{double cocircuit} is the nongraphic bridge-separable cocircuit $Y$ of a quaternary nonbinary connected matroid $M$ so that both $M \cto (U^{-} \cup Y)$, $M \cto (U^{+} \cup Y)$ are non-graphic where $U^{-}, U^{+}$ are the two classes of all avoiding bridges of $Y$. 

\begin{proposition} \label{nonbinary}
If $M(\Sigma)$ is a connected quaternary matroid and $Y$ is a double bond in $\Sigma$ then $M(\Sigma \cto (\Sigma_{2} \cup Y))$ is non-binary, where $\Sigma_{2}$ is the balanced component of $\Sigma \dof Y$.
\end{proposition}
\begin{proof}
$M(\Sigma)$ is a connected matroid implying that $\Sigma$ is a connected signed graph. Joints may appear both to the unbalanced components and to the balanced component of $\Sigma \dof Y$, $\Sigma_{2}$. Since we contract the unbalanced components in order to obtain $\Sigma \cto (\Sigma_{2} \cup Y)$, the edges of the unbalancing part of $Y$ become half-edges at their other endvertex in $\Sigma_{2}$. The balancing part of $Y$ may contain more than one elements which have all the same negative sign. Hence $\Sigma \cto (\Sigma_{2} \cup Y)$ has $K_{0}$ as a minor and $M(\Sigma \cto (\Sigma_{2} \cup Y))$ is non-binary.
\end{proof}

\begin{proposition}
If $M(\Sigma)$ is an internally 4-connected quaternary nonbinary matroid with a jointless signed-graphic representation $\Sigma$, then $Y$ is an \emph{unbalancing cocircuit} of $M(\Sigma)$ if and only if $Y$ is an unbalancing bond in $\Sigma$ with core not graphic (i.e., not B-necklace and not having a balancing vertex).
\end{proposition}
\begin{proof}
By hypothesis $\Sigma$ is either cylindrical signed graph or isomorphic to $T_{6}$. Assume that $Y$ is an unbalancing bond in $\Sigma$ such that the core is not a B-necklace and not having a balancing vertex. Hence $Y$ is nongraphic. By proposition \ref{connectedcomponents} $\Sigma \dof Y$ consists of one unbalanced component $\Sigma_{1}$ and one balanced component $\Sigma_{2}$. By theorem \ref{th_Ybridgeseparable} and \ref{t6} $Y$ is bridge-separable. Moreover, $M(\Sigma \cto (\Sigma_{1} \cup Y))$ is signed-graphic not graphic. Either $\Sigma$ is cylindrical signed graph or isomorphic to $T_{6}$, to obtain $\Sigma \cto (\Sigma_{2} \cup Y)$ we contract the unbalanced component and all edges of $Y$ become half-edges at their other endvertex. The only negative cycles in $\Sigma \cto (\Sigma_{2} \cup Y)$ are joints. Thus $M(\Sigma \cto (\Sigma_{2} \cup Y))$ is graphic. For the converse, suppose that $Y$ is an unbalancing cocircuit of $M(\Sigma)$ and a double bond in $\Sigma$. We will also denote the unbalanced component by $\Sigma_{1}$ and the balanced component of $\Sigma \dof Y$ by $\Sigma_{2}$. By proposition \ref{uniqueedge} the balancing part of $Y$ consists of one edge $e$ with both endvertices at $\Sigma_{2}$. The edges of the unbalancing part of $Y$ become half-edges at their other endvertex in $\Sigma_{2}$, since we contract $\Sigma_{1}$. Apart from joints, $\Sigma \cto (\Sigma_{2} \cup Y)$ contains another negative cycle consisting of $e$.Therefore $\Sigma \cto (\Sigma_{2} \cup Y)$ contracts to $K_{0}$. Thereby $M(\Sigma \cto (\Sigma_{2} \cup Y))$ is not graphic. Furthermore $M(\Sigma \cto (\Sigma_{1} \cup Y))$ is also signed-graphic, not graphic which leads to a contradiction.
\end{proof}

\begin{corollary}
If $M(\Sigma)$ is an internally 4-connected quaternary nonbinary matroid with a jointless signed-graphic representation $\Sigma$, then $Y$ is an \emph{star cocircuit} of $M(\Sigma)$ if and only if $Y$ is a star bond in $\Sigma$ with core not graphic (i.e., not B-necklace and not having a balancing vertex).
\end{corollary}

\begin{proposition}
If $M$ is an internally 4-connected quaternary non-binary matroid, $Y$ is an unbalancing cocircuit of $M$ and $M \cto (U^{-} \cup Y)=M(\overline{\Sigma})$ where $U^{-}$ a class of all avoiding bridges of $Y$ and $\overline{\Sigma}$ jointless signed graph, then $Y$ is an unbalancing bond in $\overline{\Sigma}$.
\end{proposition}
\begin{proof}
By hypothesis $Y$ is a nongraphic bridge-separable cocircuit of $M$. Let $U^{-},U^{+}$ be the two classes of all avoiding bridges of $Y$. Assume without loss of generality that $M \cto (U^{+} \cup Y)$ is graphic. Then $Y$ is a nongraphic cocircuit of $M \cto (U^{-} \cup Y)$. Moreover $Y$ is an unbalancing cocircuit of   $M \cto (U^{-} \cup Y)$ and specifically a star cocircuit. Suppose that $Y$ is a double bond in $\overline{\Sigma}$ and $\overline{\Sigma_{1}}$, $\overline{\Sigma_{2}}$ are the unbalanced and the balanced component of $\overline{\Sigma} \dof Y$ respectively. Furthermore assume that $U_{1},U_{2}$ are two classes of bridges of $U^{-}$ such that $M \cto (U_{1} \cup Y)=M(\overline{\Sigma_{1}} \cup Y)$ and $M \cto (U_{2} \cup Y)=M(\overline{\Sigma_{2}} \cup Y)$. Since $Y$ is a nongraphic cocircuit of $M(\overline{\Sigma})$, $M(\overline{\Sigma_{1}} \cup Y)$ is non-graphic and by proposition \ref{nonbinary} $M(\overline{\Sigma_{2}} \cup Y)$ is non-graphic. This is a contradiction to $Y$ being an unbalancing cocircuit of $M \cto (U^{-} \cup Y)$. 
\end{proof}

\subsection{2-sum} \label{subsec_2_sum}

Throughout this section we consider a connected matroid $M$ which has an exact 2-separation $(X_{1},X_{2})$. Hence $M=M_{1} \oplus_{2} M_{2}$ where $M_{1},M_{2}$ matroids with ground sets $E(M_{1})=X_{1} \cup z$ and $E(M_{2})=X_{2} \cup z$. Thus the ground set of $M$ is $E(M)=(E(M_{1} \cup E(M_{2}))-z$ and $E(M_{1}) \cap E(M_{2})=z$. We consider $M_{1}$ as an one element extension of $M \dto X_{1}$ with $z$.

The elements of a cocircuit $Y$ of $M$ may be partitioned to $M_{1},M_{2}$. Then the elements of $Y$ in $M_{1}$ with $z$ constitute a cocircuit of $M_{1}$. Respectively for $M_{2}$. If $Y$ is contained in one of $M_{1}$ or $M_{2}$ then $Y$ is a cocircuit of this matroid. This is stated in the following lemma.

\begin{lemma}
If $M=M_{1} \oplus_{2} M_{2}$ and $(Y_{1},Y_{2})$ is a partition of $Y \in \mathcal{C}^{*}(M)$ such that $Y_{1} \subseteq E(M_{1})$ and $Y_{2} \subseteq E(M_{2})$ then one of the following holds:
\begin{enumerate}
\item[(i)] both $Y_{i}$ are non-empty and $Y_{i} \cup z \in \mathcal{C}^{*}(M_{i})$ ,
\item[(ii)] $Y \subseteq E(M_{i})$, where $i=1$ or $2$ and $Y \in \mathcal{C}^{*}(M_{i})$.
\end{enumerate}
\end{lemma}
\begin{proof}
For $(i)$ assume that both $Y_{i}$ are non-empty. By definition of $2$-sum, $\mathcal{C}(M_{1})=\mathcal{C}(M\dto X_{1}) \cup \{(C \cap X_{1}) \cup z:$ $C$ circuit of $M$ meeting both $X_{1},X_{2}\}$. Since $Y$ is a circuit of $M^{*}$ meeting both $X_{1},X_{2}$ then $(Y \cap X_{1}) \cup z \in \mathcal{C}(M^{*}_{1})$ and $Y_{1} \cup z \in \mathcal{C}^{*}(M_{1})$. Similarly $Y_{2} \cup z \in \mathcal{C}^{*}(M_{2})$.

For $(ii)$ wlog assume that $Y \subseteq E(M_{1})$, so $Y \subseteq X_{1}$. Thus $Y \in \mathcal{C}(M^{*} \dto X_{1})$. Since $M_{1}$ is obtained by extending $M \dto X_{1}$ by $z$, $ \mathcal{C}(M ^{*}\dto X_{1}) \subseteq \mathcal{C}(M^{*}_{1})$. Therefore $Y \in \mathcal{C}^{*}(M_{1})$.
\end{proof}

If the elements of a cocircuit $Y$ in $M$ are partitioned to $M_{1},M_{2}$, then the bridges of $Y$ in $M$ are partitioned also to $M_{1},M_{2}$. Otherwise $Y \subseteq E(M_{i})$, where $i=1$ or $2$ and the 2-separation of $M$ is in a bridge of $Y$.

\begin{lemma}
If $M=M_{1} \oplus_{2} M_{2}$ and $(Y_{1},Y_{2})$ is a partition of $Y \in \mathcal{C}^{*}(M)$ such that $Y_{1} \subseteq E(M_{1})$ and $Y_{2} \subseteq E(M_{2})$ then one of the following holds:
\begin{enumerate}
\item[(i)] both $Y_{i}$ are non-empty and the bridges of $Y_{i} \cup z$ in $M_{i}$, $i=1,2$ are the bridges of $Y$ in $M$ that are contained in $E(M_{i})$ ,
\item[(ii)] $Y \subseteq E(M_{i})$, where $i=1$ or $2$ and the bridges of $Y$ in $M_{i}$ are the bridges of $Y$ in $M$ apart from one bridge $B$ that contains $z$ and $B \oplus_{2} M_{2}$ is a bridge of $Y$ in $M$.
\end{enumerate}
\end{lemma}
\begin{proof}
For $(i)$ assume that both $Y_{i}$ are non-empty. We shall show that $X_{i}-Y_{i}$ , $i=1,2$ is a separator of $M \dof Y$. Assume on the contrary that there exists a circuit $C \in \mathcal{C}(M \dof Y)$ such that $C \cap (X_{1}-Y_{1}) \neq \emptyset$ and $ C \cap (X_{2}-Y_{2}) \neq \emptyset$. Since $(E(M)-Y)-(X_{1}-Y_{1})=X_{2}-Y_{2}$,  $E(M)-Y$ must be a separator of $M \dof Y$. Thus $M \dof Y$ is connected which is a contradiction. Hence $M \dof Y = M \dto (X_{1}-Y_{1})$ $\oplus$ $M \dto (X_{2}-Y_{2})$ $=M_{1} \dof (Y_{1} \cup z)$ $\oplus$ $M_{2} \dof (Y_{2} \cup z)$. Therefore every bridge of $Y$ in $M$ is contained either to $E(M_{1})$ or $E(M_{2})$.

For $(ii)$ wlog assume that $Y \subseteq E(M_{1})$, so $Y \subseteq X_{1}$. It is known that  $M_{1} \oplus_{2} M_{2}=P(M_{1},M_{2}) \dof z$ where $P(M_{1},M_{2})$ is denoted the matroid which is obtained by the parallel connection of $M_{1},M_{2}$. By [proposition 7.1.15 in \cite{Oxley:2011}], $e \in E(M_{1})-z$, $(M_{1} \oplus_{2} M_{2}) \dof e$$=P(M_{1},M_{2}) \dof p \dof e$$=P(M_{1},M_{2}) \dof e \dof p$$=P(M_{1} \dof e,M_{2}) \dof p$$ =(M_{1} \dof e)  \oplus_{2} M_{2}$. It follows that $(M_{1} \oplus_{2} M_{2}) \dof Y$$=(M_{1} \dof Y)  \oplus_{2} M_{2}$. The unique bridge of $Y$ in $M_{1}$ denoted by $B$ which is not bridge of $Y$ in $M$ contains $z$. $M_{1} \oplus_{2} M_{2}$ is connected if and only if $M_{1},M_{2}$ are connected matroids. Then $B \oplus_{2} M_{2}$ is a bridge of $Y$ in $M$ because it is minimal connected subset of $E(M)-Y$.
\end{proof}

We remind that star cocircuit is a cocircuit $Y$ of $M$ such that all bridges of $Y$ in $M$ are avoiding. 

\begin{lemma}
If $M=M_{1} \oplus_{2} M_{2}$ and $(Y_{1},Y_{2})$ is a partition of a star cocircuit $Y \in \mathcal{C}^{*}(M)$ such that $Y_{1} \subseteq E(M_{1})$ and $Y_{2} \subseteq E(M_{2})$ then one of the following holds:
\begin{enumerate}
\item[(i)] both $Y_{i}$ are non-empty and $Y_{i} \cup z$ is a star cocircuit of $M_{i}$ ,
\item[(ii)] $Y \subseteq E(M_{i})$, where $i=1$ or $2$ and $Y$ is a star cocircuit of $M_{i}$. 
\end{enumerate}
\end{lemma}
\begin{proof}
For $(i)$ assume that both $Y_{i}$ are non-empty. We will show $(i)$ for $i=1$. The same arguments can be applied also for $i=2$.The bridges of $Y_{1} \cup z$ in $M_{1}$ are the bridges of $Y$ in $M$ contained in $E(M_{1})$. We shall show that any two bridges $B_{1},B_{2}$ of $Y$ in $M$ such that $B_{1},B_{2} \subseteq E(M_{1})$ are avoiding bridges of $Y_{1} \cup z$ in $M_{1}$. Specifically we shall prove that there exist $H_{1}^{*}\in{\mathcal{C}^{*}(M_{1} \cto (B_{1}\cup{Y_{1} \cup z})|{Y_{1} \cup z})}$ and $H_{2}^{*}\in{\mathcal{C}^{*}(M_{1} \cto (B_{2}\cup{Y_{1} \cup z})|{Y_{1} \cup z})}$ such that $H_{1}^{*} \cup H_{2}^{*}=Y_{1} \cup z$. 

\end{proof}


We formulate a decomposition characterisation for the class of quaternary non-binary signed-graphic matroids. The decomposition is performed by deleting a nongraphic u-cocircuit. 

\begin{theorem}
Let $M$ be an internally 4-connected quaternary nonbinary matroid and $Y$ be a nongraphic cocircuit of $M$. Then $M$ is signed-graphic with a jointless signed-graphic representation if and only if
\begin{enumerate}[(i)]
\item $Y$ is bridge-separable,
\item for the classes $U^{+},U^{-}$ of all avoiding bridges of $Y$, $M \cto (U^{+} \cup Y)$ is graphic and $M \cto (U^{-} \cup Y)$ is signed-graphic with a jointless signed-graphic representation.
\end{enumerate}
\end{theorem}

\begin{proof}
Assume that there is a jointless signed-graph $\Sigma$ such that $M=M(\Sigma)$. By theorem~\ref{th_SliDe}, $\Sigma$ is either cylindrical or isomorphic to $T_{6}$. Assume first that $\Sigma$ is cylindrical. Since $M(\Sigma)$ is internally 4-connected, there are no 1,2-separations in $M(\Sigma)$ but there are minimal 3-separations. Then $\Sigma$ is 3-biconnected and there are only 3-biseparations whose exactly one part has rank 3. Thus $\Sigma$ is 2-connected. By hypothesis $Y$ is a $\mathcal{U}$-cocircuit of $M(\Sigma)$. Then $Y$ is an unbalancing bond in $\Sigma$ and $\Sigma \dof Y$ consists of one unbalanced and one balanced component. By theorem~\ref{th_Ybridgeseparable} $Y$ is bridge-separable in $M(\Sigma)$. By corollary~\ref{Yclasses} let $U^{-}$ be the class of all avoiding bridges which correspond to the separates of the unbalanced component and let $U^{+}$ be the other class of all avoiding bridges which correspond to the separates of the balanced component. Thus $M\cto (U^{+} \cup Y)=M(\Sigma)\cto (U^{+} \cup Y)=M(\Sigma \cto (U^{+} \cup Y))$ is graphic since $\Sigma \cto (U^{+} \cup Y)$ is joint unbalanced. Moreover, $M\cto (U^{-} \cup Y)=M(\Sigma)\cto (U^{-} \cup Y)=M(\Sigma \cto (U^{-} \cup Y))$ is signed-graphic as a minor of $M(\Sigma)$. Specifically $\Sigma \cto (U^{-} \cup Y)$ is cylindrical and $M.(U^{-} \cup Y)$ has a jointless signed-graphic representation. Assume that $\Sigma$ is isomorphic to $T_{6}$. 

Conversely, assume that conditions (i) and (ii) hold. If there is one bridge of $Y$ in $M$, then $M=M \cto (U^{-} \cup Y)$ is signed-graphic since otherwise $M$ is graphic which is a contradiction. Otherwise, there are at least two bridges of $Y$ and since $Y$ is bridge-separable they can be partitioned into two classes $U^{+},U^{-}$ such that all the members of the same class are avoiding. By assumption $M.(U^{+} \cup Y)$ is graphic and $M.(U^{-} \cup Y)$ is signed-graphic with a jointless signed-graphic representation. Any $Y \in \mathcal{C^*}(M)$ is also a cocircuit of $M \cto(U \cup Y)$ where $U \in \{U^{+},U^{-} \}$.  The cocircuits of $M \cto (U \cup Y)$ are the circuits of $M \dto (U \cup Y)$ which are the cocircuits of $M$ which are contained in $U \cup Y$. 


\section{Figures} \label{sec_figures}
In the following figures, a solid line will depict an edge of positive sign while a dashed line will depict an edge of negative sign. In Figure~\ref{cylint4} all double bonds of the cylindrical signed graph are graphic cocircuits in the corresponding signed-graphic matroid. The core of the deletion of a double bond of the cylindrical signed graph is either a B-necklace or has a balancing vertex.

\begin{figure}[h]
\begin{center}
\includegraphics*[scale=0.45]{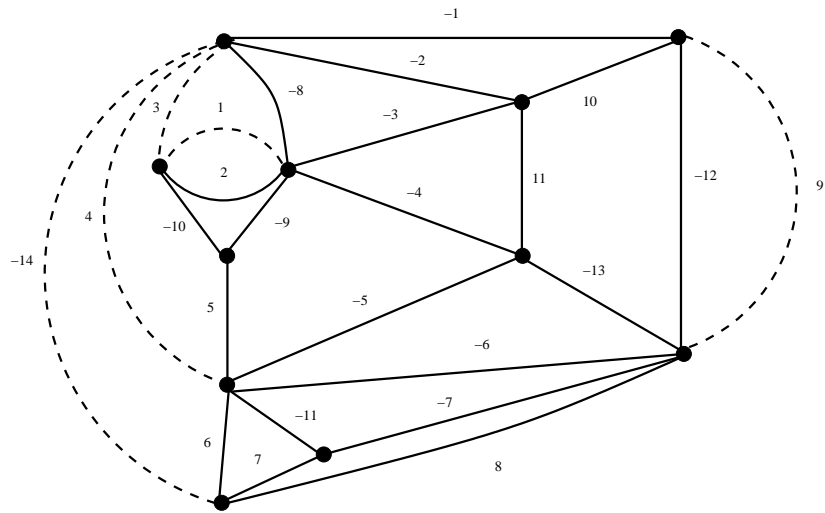}
\end{center} 
\caption{Cylindrical signed graph whose $M(\Sigma)$ is 3-connected, cyl3conn}
\label{cyl3conn}
\end{figure}

\begin{figure}[h]
\begin{center}
\includegraphics*[scale=0.45]{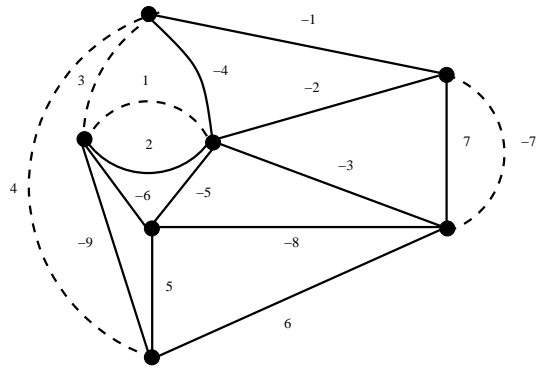}
\end{center} 
\caption{Cylindrical signed graph whose $M(\Sigma)$ is internally 4-connected, cylint4}
\label{cylint4}
\end{figure}

\begin{example}
$Y_{146}=\{8,10,-6,-7,9,-13,-1\}$ is a double bond of the cylindrical signed graph in Figure~\ref{cyl2conn} and a nongraphic cocircuit of its associated signed-graphic matroid. The bridges of $Y$ in $M(\Sigma)$ are $B_{1}=\{1,2,3,4,5,11,-2,-3,-4,-5,-8,-9,-10\}$, $B_{2}=\{6,7,-11\}$, $B_{3}=\{-12\}$. The sets of cocircuits of matroids $M(\Sigma) \cto (B_{i} \cup Y) \dto Y), i=1,2,3$, for each $B_{i}$ are the following. \\
\noindent
$\mathcal{C}^{*}(M \cto (B_{1} \cup Y) \dto Y)=\{\{-13\}, \{-1\},\{-6,-7,8\},\{9\},\{10\}\}$  \\
$\mathcal{C}^{*}(M \cto (B_{2} \cup Y) \dto Y)=\{-7\},\{8\},\{-1,-6,9,-13,10\}\} $ \\
$\mathcal{C}^{*}(M \cto (B_{2} \cup Y) \dto Y)=\{8,-6,-7,-13,9\},\{-1,8,-6,-7,-13,10\},\{9,10,-1\}\}$. \\
We observe that all bridges are avoiding with each other.
\end{example}
\end{proof}

\begin{figure}[h]
\begin{center}
\includegraphics*[scale=0.45]{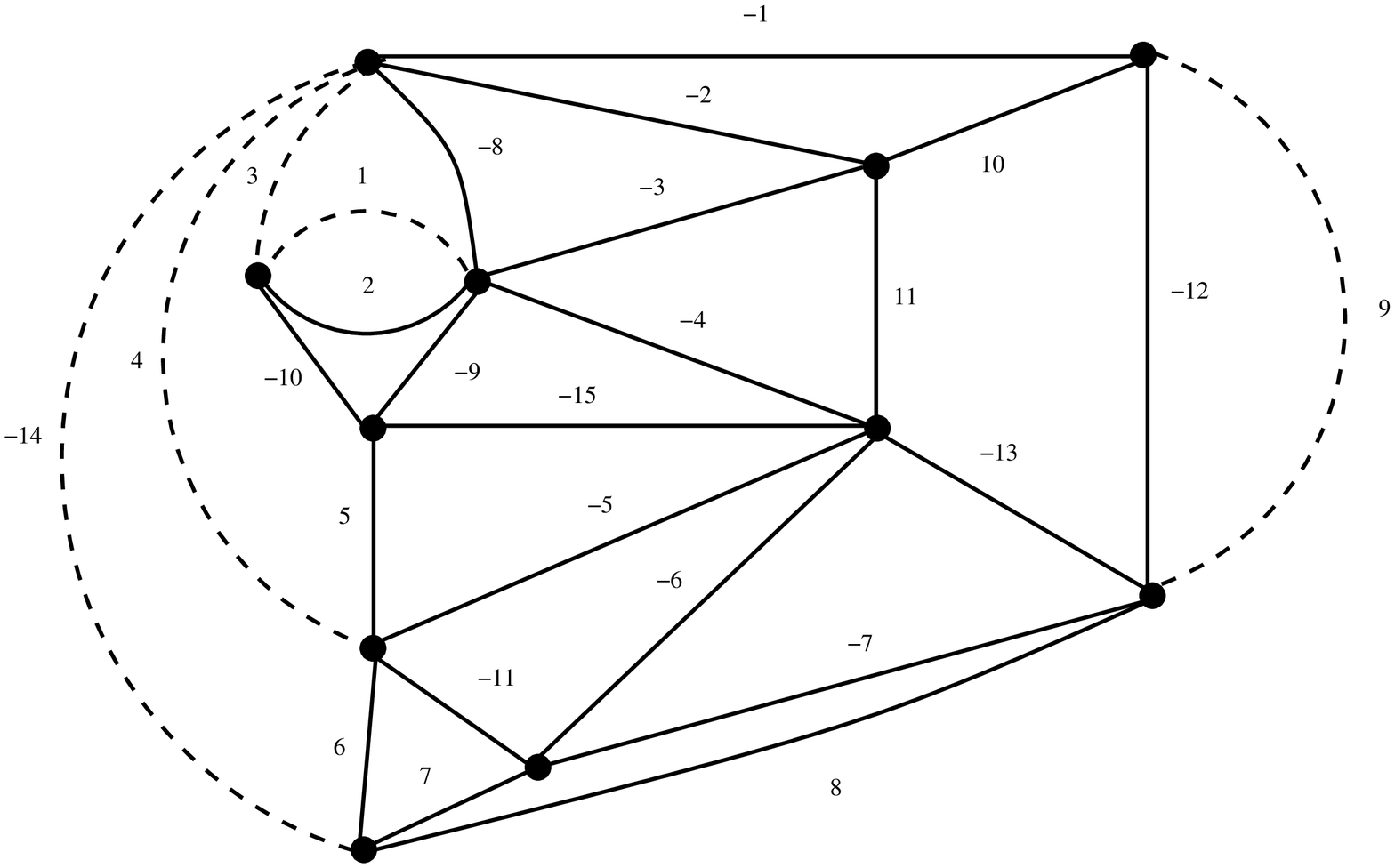}
\end{center} 
\caption{Cylindrical signed graph whose $M(\Sigma)$ is internally 4-connected, cyli4}
\label{cyli4}
\end{figure}

\bibliographystyle{plain}

\end{document}